\documentclass[11pt,twoside]{article}
\newcommand{\ifims}[2]{#1}   
\newcommand{\ifAMS}[2]{#1}   
\newcommand{\ifau}[3]{#2}  

\newcommand{\ifbook}[2]{#1}   

\ifbook{

  }
  {

  }

\def\thetitle{Finite Sample Bernstein -- von Mises Theorem for Semiparametric Problems}
\def\theruntitle{Finite Sample BvM Theorem for Semiparametric Problems}

\def\theabstract{
The classical parametric and semiparametric Bernstein -- von Mises (BvM) results are reconsidered in a non-classical setup allowing finite samples and model misspecification. 
  In the case of a finite dimensional nuisance parameter we obtain an upper bound on the error of Gaussian approximation of the posterior distribution for the target parameter 
  which is explicit in the dimension of the nuisance and target parameters. 
  This helps to identify the so called \emph{critical dimension} \( p \) of the full  parameter for which the BvM result is applicable. 
  In the important i.i.d. case, we show that the condition ``\( p^{3} / n \) is small'' is sufficient for BvM result to be valid under general assumptions on the model. 
  We also provide an example of a model with the phase transition effect: the statement of the BvM theorem fails when the dimension \( p \) approaches \( n^{1/3} \). The results are extended to the case of infinite dimensional parameters with the nuisance parameter from a Sobolev class. 
  In particular we show near normality of the posterior if the smoothness parameter \(s \) exceeds 3/2.
}

\def\kwdp{62F15}
\def\kwds{62F25,62H12}

\def\thekeywords{prior, posterior, Bayesian inference, semiparametric, critical dimension}

\def\authorb{Vladimir Spokoiny}
\def\runauthorb{spokoiny, v.}
\def\addressb{
    Weierstrass Institute and \\ Humboldt University Berlin, \\ Moscow Institute of
    Physics and Technology \\
    Mohrenstr. 39, \\
    10117 Berlin, Germany}
\def\emailb{spokoiny@wias-berlin.de}
\def\affiliationb{Weierstrass-Institute, Humboldt University Berlin, and
Moscow Institute of Physics and Technology}
\def\thanksb{
  The author is partially supported by Laboratory for Structural Methods of Data Analysis in Predictive Modeling, MIPT, RF government grant, ag. 11.G34.31.0073. Financial support by the German Research Foundation (DFG) through the Collaborative Research Center 649 ''Economic Risk'' is gratefully acknowledged.
}

\def\authora{Maxim Panov}
\def\runauthora{panov, m.}
\def\addressa{
    Moscow Institute of Physics and Technology, \\
    Institute for Information Transmission Problems of RAS, \\
    Datadvance Company, \\
    Pokrovsky blvd. 3 building 1B, \\
    109028 Moscow, Russia
    }
\def\emaila{maxim.panov@datadvance.net}
\def\affiliationa{Moscow Institute of Physics and Technology, Institute for Information Transmission Problems of RAS and Datadvance}
\def\thanksa{
  The author is partially supported by Laboratory for Structural Methods of Data Analysis in Predictive Modeling, MIPT, RF government grant, ag. 11.G34.31.0073.
}

\date{June 15, 2014 \vspace{5cm}}

\usepackage{amsmath,amssymb,amsthm}
\usepackage{natbib}
\usepackage{epsfig,graphicx}
\usepackage{comment}
\usepackage{color}
\usepackage{srcltx}
\usepackage[mathscr]{eucal}
\usepackage[math]{easyeqn}
\usepackage{etoolbox}

\ifims{
\textheight=23cm
\textwidth=14.8cm
\topmargin=0pt
\oddsidemargin=1.0cm
\evensidemargin=1.0cm
\linespread{1.3}
\renewenvironment{abstract}
    {\centerline{\textbf{Abstract}}\bigskip
      \begin{center}
       \begin{minipage}{11cm}
        \begin{small}
    }
    {   \end{small}
       \end{minipage}
      \end{center}
     \bigskip
    }

}{ 
}

\numberwithin{equation}{section}
\numberwithin{figure}{section}
\newcounter{example}[section]
\numberwithin{example}{section}
\newcounter{remark}[section]
\numberwithin{remark}{section}
\newtheorem{theorem}{Theorem}[section]

\newtheorem{lemma}[theorem]{Lemma}
\newtheorem{corollary}[theorem]{Corollary}

\newtheorem{exmp}[example]{Example}
\newtheorem{rmrk}[remark]{Remark}
\newenvironment{example}{\begin{exmp}\rm}{\end{exmp}}
\newenvironment{remark}{\begin{rmrk}\rm}{\end{rmrk}}

\renewcommand{\textfraction}{0.00}
\renewcommand{\topfraction}{1}
\renewcommand{\bottomfraction}{1}

\begin{document}
\thispagestyle{empty}
\ifims{
\title{\thetitle}
\ifau{ 
  \author{
    \authora
    \ifdef{\thanksa}{\thanks{\thanksa}}{}
    \\[5.pt]
    \addressa \\
    \texttt{ \emaila}
  }
}
{  
  \author{
    \authora
    \ifdef{\thanksa}{\thanks{\thanksa}}{}
    \\[5.pt]
    \addressa \\
    \texttt{ \emaila}
    \and
    \authorb
    \ifdef{\thanksb}{\thanks{\thanksb}}{}
    \\[5.pt]
    \addressb \\
    \texttt{ \emailb}
  }
}
{   
  \author{
    \authora
    \ifdef{\thanksa}{\thanks{\thanksa}}{}
    \\[5.pt]
    \addressa \\
    \texttt{ \emaila}
    \and
    \authorb
    \ifdef{\thanksb}{\thanks{\thanksb}}{}
    \\[5.pt]
    \addressb \\
    \texttt{ \emailb}
    \and
    \authorc
    \ifdef{\thanksc}{\thanks{\thanksc}}{}
    \\[5.pt]
    \addressc \\
    \texttt{ \emailc}
  }
}

\maketitle
\pagestyle{myheadings}
\markboth
 {\hfill \textsc{ \small \theruntitle} \hfill}
 {\hfill
 \textsc{ \small
 \ifau{\runauthora}
      {\runauthora and \runauthorb}
      {\runauthora, \runauthorb, and \runauthorc}
 }
 \hfill}
\begin{abstract}
\theabstract
\end{abstract}

\ifAMS
    {\par\noindent\emph{AMS 2000 Subject Classification:} Primary \kwdp. Secondary \kwds}
    {\par\noindent\emph{JEL codes}: \kwdp}

} 
{ 
\begin{frontmatter}
\title{\thetitle}


\runtitle{\theruntitle}

\ifau{ 
\begin{aug}
    \author{\authora\ead[label=e1]{\emaila}}
    \address{\addressa \\
     \printead{e1}}
\end{aug}

 \runauthor{\runauthora}
\affiliation{\affiliationa} }
{ 
\begin{aug}
    \author{\authora\ead[label=e1]{\emaila}\thanksref{t21}}
    \and
    \author{\authorb\ead[label=e2]{\emailb}\thanksref{t22}}
    
    \address{\addressa \\
     \printead{e1}}
    \address{\addressb \\
     \printead{e2}}
    \thankstext{t21}{\thanksa}
    \thankstext{t22}{\thanksb}
    \affiliation{\affiliationa, \affiliationb} 
    \runauthor{\runauthora and \runauthorb}
\end{aug}
} 
{ 
\begin{aug}
    \author{\authora\ead[label=e1]{\emaila}\thanksref{t21}}
    \and
    \author{\authorb\ead[label=e2]{\emailb}\thanksref{t22}}
    \and
    \author{\authorc\ead[label=e3]{\emailc}\thanksref{t23}}
    
    \address{\addressa \\
     \printead{e1}}
    \address{\addressb \\
     \printead{e2}}
    \address{\addressc \\
     \printead{e3}}
    \thankstext{t21}{\thanksa}
    \thankstext{t22}{\thanksb}
    \thankstext{t23}{\thanksc}
    \affiliation{\affiliationa, \affiliationb, \affiliationc} 
    \runauthor{\runauthora, \runauthorb, and \runauthorc}
\end{aug}}

\begin{abstract}
\theabstract
\end{abstract}



\end{frontmatter}
} 

\def\ND{\cc{N}}
\def\Bernoulli{\mathrm{Bernoulli}}
\def\Vola{\mathrm{Vola}}
\def\Poisson{\mathrm{Poisson}}
\def\ag{\mathrm{ag}}
\def\glob{\operatorname{glob}}
\def\blk{\operatorname{block}}
\def\lin{\operatorname{lin}}
\def\cond{\, \big| \,}

\def\rdl{\epsilon}
\def\rd{\bb{\rdl}}
\def\rddelta{\delta}
\def\rdomega{\varrho}
\def\rddeltab{\rddelta^{*}}
\def\rhorb{\rhor^{*}}

\def\wv{\bb{w}}
\def\varthetav{\bb{\vartheta}}
\def\Lr{\breve{L}}
\def\zetavr{\breve{\zetav}}
\def\etavr{\breve{\etav}}
\def\xivr{\breve{\xiv}}

\def\rdb{\rd}
\def\rdm{\underline{\rdb}}

\def\taub{\tau_{\rdb}}
\def\taum{\tau_{\rdm}}
\def\kappab{\kappa_{\rd}}
\def\deltab{\delta_{\rd}}

\def\taubGP{\tau_{\rdb,\GP}}
\def\taumGP{\tau_{\rdm,\GP}}
\def\kappabGP{\kappa_{\rd,\GP}}
\def\deltabGP{\delta_{\rd,\GP}}
\def\nubm{\nu_{\rd}}
\def\uub{u_{\rd}}
\def\uubGP{u_{\rd,\GP}}
\def\nubmGP{\nu_{\rd, G}}

\def\rG{\rd,\GP}

\def\LinSp{\mathrm{L}}
\def\Id{I\!\!\!I}
\def\Ind{\operatorname{1}\hspace{-4.3pt}\operatorname{I}}

\def\BG{\mathcal{R}}
\def\bg{r}
\def\fmup{\phi}
\def\rg{r}
\def\uc{u_{c}}
\def\muc{\mu_{c}}
\def\mud{\mu_{0}}
\def\xxd{\xx_{0}}
\def\yyd{\yy_{0}}
\def\gmd{\gm_{0}}

\def\ms{m^{*}}
\def\Inv{A}
\def\InvT{\Inv^{\T}}
\def\Invt{\tilde{\Inv}}

\def\ssize{N}
\def\nsize{{n}}

\def\rhor{\omega}

\def\LT{L}
\def\LGP{\LT_{\GP}}
\def\La{\mathbb{L}}
\def\Lab{\La_{\rdb}}
\def\Lam{\La_{\rdm}}

\def\DP{D}
\def\DPc{\DP_{0}}
\def\DPb{\DP_{\rdb}}
\def\DPm{\DP_{\rdm}}

\def\LabGP{\La_{\rdb,\GP}}
\def\LamGP{\La_{\rdm,\GP}}

\def\DPbGP{\DP_{\rdb,\GP}}
\def\DPmGP{\DP_{\rdm,\GP}}
\def\riskbGP{\riskt_{\rdb,\GP}}

\def\gmi{\mathtt{b}}
\def\gmiid{\mathtt{g}_{1}}
\def\kullbi{\Bbbk}
\def\Thetasi{\Theta_{\loc}}
\def\rri{\mathtt{u}}
\def\rris{\rri_{0}}

\def\Ipc{\bb{\mathrm{f}}}
\def\IF{\Bbb{F}}
\def\IFc{\IF_{0}}
\def\IFb{\IF_{\rdb}}
\def\IFm{\IF_{\rdm}}

\def\DF{\cc{D}}
\def\DFc{\DF_{0}}
\def\DFb{\DF_{\rdb}}
\def\DFm{\breve{\DF}_{\rd}}
\def\DFm{\DF_{\rdm}}

\def\DPr{\breve{\DP}}
\def\VF{\cc{V}}
\def\VFc{\VF_{0}}

\def\HHc{\HH_{0}}
\def\HHb{\HH_{\rd}}
\def\HHm{\HH_{\rdm}}

\def\xib{\xi^{*}}
\def\xivb{\xiv_{\rdb}}
\def\xivm{\xiv_{\rdm}}
\def\CAm{\underline{\CA}}
\def\CAb{\CA}

\def\penr{\operatorname{pen}}
\def\pen{\mathfrak{t}}
\def\PEN{\operatorname{PEN}}
\def\RSS{\operatorname{RSS}}
\def\med{\operatorname{med}}

\def\ex{\mathrm{e}}
\def\entrl{\mathbb{Q}}
\def\entrlb{\entrl}
\def\entr{\entrl}

\def\kullb{\cc{K}} 
\def\kullbc{\kullb^{c}}

\def\gm{\mathtt{g}}
\def\gmc{\gm_{c}}
\def\gmb{\gm}
\def\gmbm{\gmb_{1}}

\def\yy{\mathtt{y}}
\def\yyc{\yy_{c}}
\def\xx{\mathtt{x}}
\def\xxc{\xx_{c}}
\def\tc{t_{c}}

\def\alp{\alpha}
\def\alpn{\rho}
\def\gmu{\mathfrak{a}}

\def\losst{\varrho}
\def\loss{\wp}
\def\lossp{u}
\def\closs{g}

\def\riskt{\cc{R}}
\def\emprisk{\ell}
\def\bias{b}
\def\bern{q}

\def\TT{\nsize}

\def\Pone{P}
\def\Pf{\P_{f(\cdot)}}
\def\Ef{\E_{f(\cdot)}}
\def\Ps{\P_{\thetas}}
\def\Es{\E_{\thetas}}
\def\Pu{\P_{\upsilons}}
\def\Eu{\E_{\upsilons}}

\def\Pvs{\P_{\thetavs}}
\def\Evs{\E_{\thetavs}}

\def\UPd{w}
\def\nunup{\nu_{1}}
\def\rru{\rr_{1}}
\def\rups{\rr_{0}}
\def\rupsb{\rups^{*}}
\def\rrf{\rr^{\flat}}
\def\rupd{\rr_{\circ}}

\def\smooths{\mathbb{S}}
\def\smooth{\smooths_{1}}

\def\elli{\bar{\ell}}

\def\K{K}

\def\Psir{\breve{\Psi}}

\def\af{a}
\def\afs{\af^{*}}

\def\kapla{\varkappa}

\newcommand{\mlew}[1]{\tilde{\thetav}_{#1}}
\newcommand{\mlea}[1]{\hat{\thetav}_{#1}}
\newcommand{\mluw}[1]{\tilde{\theta}_{#1}}
\newcommand{\mlua}[1]{\hat{\theta}_{#1}}
\newcommand{\penm}[1]{\boldsymbol{m}_{#1}}

\def\Pdom{\mu_{0}}
\def\PDOM{\bb{\mu}_{0}}
\def\EDOM{\E_{0}}

\def\mk{m}
\def\Mk{\cc{M}}
\def\SV{\cc{S}}

\def\Cs{E}
\def\Csd{\Cs^{\circ}}
\def\Ca{A}
\def\CS{\cc{E}}
\def\CA{\cc{A}}
\def\CAb{\CA_{\rd}}
\def\CAC{\CA_{\CoFu}}

\def\Ccb{m_{\rdb}}
\def\Ccm{m_{\rdm}}
\def\CcbGP{m_{\rdb,\GP}}
\def\CcmGP{m_{\rdm,\GP}}

\def\etas{\eta^{*}}

\def\omegav{\bb{\phi}}
\def\omegavs{\omegav^{*}}
\def\omegavc{\omegav'}

\def\nuvs{\nuv^{*}}
\def\nuvc{\nuv'}

\def\nunu{\nu_{0}}
\def\numu{\nu_{1}}
\def\nupi{\nu^{+}}
\def\nubu{\beta}

\def\nus{\nu}
\def\nusb{\nus}
\def\nusr{\nus^{\bracketing}}
\def\Nusb{\mathbb{N}}
\def\Nusr{\mathbb{N}^{\diamond}}

\def\dist{d}
\def\distd{\mathfrak{a}}

\def\hatk{\kappa}
\def\ko{k^{\circ}}

\def\qqq{\mathfrak{q}}
\def\ppp{{s}}
\def\Cqq{C(\qqq)}
\def\Cqqb{C^{\diamond}(\qqq)}
\def\Crho{C(\mrho)}
\def\Cqqm{\log(4)}
\def\Cqpr{(\qqq \rrp + \dimp / 2)}

\def\Cdima{\mathfrak{e}_{0}}
\def\Cdimb{\mathfrak{e}_{1}}
\def\cdima{\mathfrak{c}_{0}}
\def\cdimb{\mathfrak{c}_{1}}
\def\cdim{\mathfrak{c}}

\def\rdomega{\varrho}
\def\deltaD{\delta}
\def\alphai{\alpha_{1}}
\def\alphaii{\alpha_{2}}
\def\alphaiii{\alpha_{3}}
\def\alphaiv{\alpha_{4}}

\def\err{\diamondsuit}
\def\errbm{\bar{\err}_{\rdomega}}
\def\errm{\err_{\rdm}}
\def\errb{\err_{\rdb}}

\def\errbGP{\err_{\rdomega,\GP}}
\def\errmGP{\err_{\rdm,\GP}}
\def\errbmGP{\bar{\err}_{\rd,\GP}}

\def\errs{\err_{\rdomega}^{*}}
\def\deltas{\alpha}

\def\xivbGP{\xiv_{\rdb,\GP}}
\def\xivmGP{\xiv_{\rdm,\GP}}

\def\SP{S}
\def\GP{G}
\def\GPt{\GP_{0}}
\def\GPn{\GP_{1}}
\def\gp{g}
\def\gs{s}

\def\errbGP{\err_{\rdb,\GP}}
\def\errmGP{\err_{\rdm,\GP}}
\def\errpmGP{\err_{\GP}^{\pm}}

\def\LCS{\cc{C}}

\def\DPGP{\DP_{\GP}}
\def\thetavsGP{\thetavs_{\GP}}

\def\LL{\cc{L}}
\def\LLb{\LL^{*}}
\def\LLh{\cc{L}}

\def\YY{\cc{Y}}
\def\LP{L^{\circ}}

\def\modcnrd{\mathfrak{A}}

\def\pens{\pi}
\def\pnn{\mathfrak{g}}
\def\pnnd{\mathfrak{u}}
\def\pnndGP{\pnnd_{\GP}}

\def\confpr{\mathfrak{c}}
\def\confprb{\confpr^{*}}

\def\pn{\pens^{*}}
\def\penInt{\mathfrak{D}}
\def\penH{\mathbb{H}}
\def\pmu{\mathfrak{u}}
\def\Closs{\cc{R}}

\def\dimp{p}
\def\riskb{\riskt_{\rdb}}
\def\dimpp{\dimp+1}
\def\BB{I\!\!B}
\def\vA{\mathtt{v}}

\def\deficiency{\Delta}
\def\spread{\Delta}
\def\dimtotal{\dimp^{*}}

\def\thetav{\bb{\theta}}
\def\thetavs{\thetav^{*}}
\def\thetavc{\thetav'}
\def\thetavd{\thetav^{\circ}}
\def\thetavdc{\thetav^{\sharp}}
\def\dthetavs{\thetav,\thetavs}

\def\thetas{\theta^{*}}
\def\thetac{\theta'}
\def\thetad{\theta^{\circ}}
\def\thetab{\theta^{\dag}}
\def\thetavb{\thetav^{\dag}}

\def\vtheta{\vartheta}
\def\vthetav{\bb{\vtheta}}
\def\prior{\Pi}

\def\Gam{\Xi}
\def\Gam{\mathcal{S}}
\def\RG{R}
\def\Psu{\Upsilon}
\def\Phim{\breve{\Phi}}

\def\Proj{P}

\def\gammavs{\gammav^{*}}
\def\gammavd{\gammav^{\circ}}
\def\etavs{\etav^{*}}
\def\etavd{\etav^{\circ}}
\def\etavc{\etav'}

\def\taus{\tau_{0}}
\def\taup{\lceil \tau \rceil}

\def\sigmas{{\sigma^{*}}}
\def\Sigmas{\Sigma_{0}}

\def\upsilonc{\upsilon'}
\def\upsilond{\upsilon^{\circ}}
\def\upsilonp{{\upsilon}^{*}}
\def\upsilonm{{\upsilon}_{*}}
\def\upsilonvs{\upsilonv^{*}}
\def\upsilons{\upsilon^{*}}
\def\upsilonb{\bar{\upsilon}}
\def\upsilonvd{\upsilonv^{\circ}}

\def\ups{\bb{\upsilon}}
\def\upss{\ups_{0}}
\def\upsc{\ups^{\prime}}
\def\upsd{\ups^{\circ}}
\def\upsdc{\ups^{\sharp}}
\def\upsdu{\ups^{\flat}}

\def\Ups{\varUpsilon}
\def\Upsd{\Ups^{\circ}}
\def\Upss{\Ups_{\circ}}
\def\UpsP{\Ups^{c}}

\def\Thetas{\Theta_{0}}
\def\ThetasGP{\Theta_{0,\GP}}
\def\varthetav{\bb{\vartheta}}

\def\glink{g}

\def\fvs{\fv}
\def\fs{f}
\def\fb{\fv^{\dag}}

\def\uc{\uv'}
\def\ud{\uv^{\circ}}
\def\uvs{\uv^{*}}
\def\us{u^{*}}
\def\vs{v^{*}}

\def\reps{\epsilon}
\def\eps{\epsilon}

\def\repsc{\reps_{0}}
\def\repsb{\reps^{*}}
\def\repsg{g}

\def\lu{\delta}
\def\lub{\bar{\lu}}

\def\Uu{U}
\def\UU{\cc{Y}}
\def\UUM{\cc{M}}
\def\UP{\cc{U}}
\def\up{\mathfrak{u}}

\def\VP{V}
\def\VPc{\VP_{0}}
\def\VPV{\cc{U}}
\def\VPVc{\cc{\VPV}_{0}}
\def\VPGP{\VP_{\GP}}
\def\VPSP{\VP_{\SP}}

\def\VV{H}
\def\GV{\cc{G}}
\def\GVS{S}

\def\VVb{\VV^{*}}
\def\VVc{\VV_{0}}
\def\vv{\bb{h}}
\def\vva{g}
\def\vp{\mathbf{v}}
\def\vpc{\vp_{0}}
\def\VVca{\VV}
\def\Vtt{H}

\def\DG{D}

\def\Vn{V_{0}}
\def\vn{v_{0}}

\def\norm{\mathfrak{c}}
\def\normc{\delta}
\def\norma{c}

\def\egridd{\cc{E}_{\delta}}
\def\penb{\varkappa}

\def\dotzeta{\dot{\zeta}}
\def\mes{\pi}
\def\mesl{\Lambda}
\def\cprr{F}

\def\lambdam{\gm_{1}}
\def\lambdaB{{\lambda}^{*}}
\def\lambdac{{\lambda'}}

\def\cla{{b}}
\def\fis{\mathfrak{a}}
\def\fiss{\fis_{1}}

\def\Vd{{V}}
\def\vd{\bar{v}}

\def\klim{k^{\circ}}
\def\midm{\mid \!}

\def\Ldrift{M}
\def\ldrift{m}
\def\mY{b}
\def\Lvar{D}
\def\lvar{\sigma}

\def\Mubcu{\Upsilon}
\def\Dthetav{\bb{u}}

\def\B{\cc{B}}
\def\BD{\B^{\circ}}
\def\BU{B}
\def\BI{\B^{*}}

\def\mub{\mu^{*}}
\def\mubc{\mu}
\def\mubcb{\mubc^{*}}
\def\Mubc{\mathbb{M}}
\def\Mubcb{\mathrm{M}}

\def\zzc{\zz_{c}}
\def\ww{w}
\def\wwc{\ww_{c}}

\def\norms{\circ} 
\def\rs{\rr_{\norms}}
\def\yys{\yy_{\norms}}
\def\xxs{\xx_{\norms}}
\def\zzs{\zz_{\norms}}
\def\uu{\mathtt{u}}
\def\uus{\uu_{\norms}}
\def\mus{\mu_{\norms}}
\def\gms{\gm_{\norms}}
\def\wws{\ww_{\circ}}

\def\srho{s}
\def\mrho{\varrho}

\def\Lmgf{\mathfrak{M}}
\def\Lmgfb{\Lmgf^{*}}

\def\lmgf{\mathfrak{m}}
\def\lmgfb{\lmgf^{*}}

\def\Expzeta{\mathfrak{N}}
\def\expzeta{\mathfrak{s}}

\def\rr{\mathtt{r}}
\def\rrb{\rr^{*}}
\def\rru{\rr_{\circ}}
\def\rrc{\rr'}
\def\rs{r_{*}}

\def\zz{\mathfrak{z}}
\def\zzb{\tilde{\zz}}
\def\tt{\mathfrak{t}}
\def\zb{z_{\rd}}
\def\zzg{\zz_{1}}
\def\zzQ{\zz_{0}}
\def\zzq{\zz}

\def\Cr{\mathfrak{c}}
\def\Crp{\mathfrak{C}}
\def\Crl{\mathfrak{r}}
\def\Crlp{\mathfrak{R}}
\def\Crlq{\cc{T}}
\def\Crlmu{\cc{M}}

\def\zetah{\zeta_{h}}
\def\GG{G}
\def\HH{H}
\def\pG{p}
\def\pH{q}
\def\hh{H^{*}}

\def\mubch{\mubc_{1}}
\def\rhoh{\rho_{1}}
\def\CoFuh{\CoFu_{1}}
\def\dimh{p_{1}}
\def\VPh{\VP_{1}}
\def\VPt{\VP_{0}}

\def\LLh{L_{1}}
\def\pnndh{\pnnd_{1}}

\def\LCS{C}
\def\Ac{A_{0}}
\def\Ab{A_{\rd}}
\def\DPrb{\DPr_{\rdb}}
\def\DPrm{\DPr_{\rdm}}
\def\Cb{\cc{C}_{\rdb}}
\def\Ub{\cc{U}_{\rdb}}
\def\zetavrb{\zetavr_{\rd}}
\def\xivrb{\breve{\xiv}_{\rd}}
\def\VPrb{\breve{\VP}_{\rdb}}
\def\Larb{\breve{\La}_{\rdb}}
\def\Larm{\breve{\La}_{\rdm}}

\def\deltav{\bb{\delta}}

\def\score{\nabla}
\def\scorer{\breve{\nabla}}

\def\LCS{C}
\def\Ac{A_{0}}
\def\Bc{B_{0}}
\def\AF{A}
\def\Ab{A_{\rdb}}
\def\Am{A_{\rdm}}
\def\DPrc{\DPr_{0}}
\def\DPrb{\DPr_{\rdb}}
\def\DPrm{\DPr_{\rdm}}
\def\Cb{\cc{C}_{\rdb}}
\def\Cm{\cc{C}_{\rdm}}
\def\Ub{\cc{U}_{\rdb}}
\def\deltav{\bb{\delta}}
\def\nuv{\bb{\nu}}
\def\xivrb{\breve{\xiv}_{\rd}}
\def\VPrb{\breve{\VP}_{\rdb}}
\def\Larb{\breve{\La}_{\rdb}}
\def\Lar{\breve{\La}}
\def\Larm{\breve{\La}_{\rdm}}
\def\VH{Q}
\def\VHc{\VH_{0}}
\def\zetavrm{\zetavr_{\rdm}}
\def\N{\mathbb{N}}

\def\Span{\operatorname{span}}
\def\Exc{{\square}}
\def\UUs{U_{\circ}}
\def\errbm{\errb^{*}}
\def\corrDF{\nu}
\def\BBr{\breve{\BB}}
\def\taua{\tau}
\def\AssId{\mathcal{I}}
\def\assId{\iota}
\def\AFD{\cc{A}}

\def\BanX{\cc{X}}
\def\basX{\ev}
\def\apprX{\alpha}
\def\fvs{\fv^{*}}
\def\lkh{\ell}
\def\Bc{B_{0}}
\def\dimn{\dimp_{\nsize}}
\def\betan{\beta_{\nsize}}


\def\xivGP{\xiv_{\GP}}
\def\dimA{\mathtt{p}}
\def\dimAGP{\dimA}
\def\dime{\dimA_{e}}
\def\dimG{\dimA_{\GP}}
\def\dimS{\dimA_{s}}
\def\nubm{\nu_{\rd}}
\def\uub{u_{\rd}}
\def\uubGP{u_{\rd,\GP}}

\def\priorden{\pi}
\def\xivGP{\xiv_{\GP}}
\def\dimAGP{\dimA}
\def\nubm{\nu_{\rd}}
\def\uub{u_{\rd}}
\def\uubGP{u_{\rd,\GP}}

\def\CR{\mathcal{C}}
\def\CRb{\CR_{\rdb}}
\def\vthetavb{\bar{\vthetav}}
\def\Covpost{\mathfrak{S}}

\def\Db{\DP_{+}}
\def\Dm{\DP_{-}}
\def\uvb{\uv_{+}}
\def\uvm{\uv_{-}}
\def\uud{\omega}
\def\taub{\delta}
\def\Lip{L}
\def\Xb{X_{+}}
\def\Xm{X_{-}}
\def\deltam{\delta_{-}}
\def\betauv{\delta}
\def\betab{\betauv_{1}}
\def\betaf{\betauv_{2}}
\def\upsv{\bb{\varkappa}}
\def\upsvb{\bar{\upsv}}
\def\rhob{\varrho}
\def\alpb{\alp_{1}}
\def\betap{\betauv_{3}}
\def\Ec{\E^{\circ}}
\def\ff{f}
\def\fpos{g}
\def\fneg{h}
\def\alpb{\alp_{+}}
\def\alpm{\alp_{-}}

\def\kappak{\kappa}
\def\kappas{\kappak^{*}}
\def\Kappak{\cc{K}}
\def\DPk{\DP_{\kappak}}
\def\VPk{\VP_{\kappak}}

\def\ts{s}
\def\tsv{\bb{\ts}}
\def\mm{\kappa}
\def\mmc{\mm'}
\def\mmd{\mm^{\circ}}
\def\mmo{\mm^{*}}
\def\mmmmo{\mm,\mmo}
\def\mmt{\tilde{\mm}}
\def\mma{\hat{\mm}}
\def\pp{z}

\def\LLL{L_{1}}
\def\LLr{L_{0}}
\def\muL{\mu_{1}}
\def\mur{\mu_{0}}

\def\LmgfL{\Lmgf_{1}}
\def\Lmgfr{\Lmgf_{0}}
\def\Lmgfm{\Lmgf_{1}}

\def\Kappa{\cc{K}}
\def\CoFu{\cc{C}}
\def\CoFuc{\CoFu_{0}}
\def\CoFub{\CoFu^{*}}
\def\CoFuL{\CoFu_{1}}
\def\CoFur{\CoFu_{0}}
\def\CAL{\CA_{1}}
\def\CAr{\CA_{0}}
\def\CAzz{\cc{A}}

\def\pnnL{\pnn_{1}}
\def\pnnr{\pnn_{0}}
\def\ttd{\delta}
\def\alphaL{\alpha_{1}}
\def\alphar{\alpha_{0}}
\def\alpharL{\alpha}
\def\rat{\mathfrak{t}}
\def\mquad{\nquad}
\def\zzL{\zz_{1}}
\def\zzr{\zz_{0}}

\def\mmset{\mathcal{I}}
\def\xex{u}
\def\dcm{q}
\def\dc{g}
\def\dcL{\dc_{1}}
\def\dcr{\dc_{0}}
\def\kk{k}

\def\cpen{\tau}

\def\dens{f}
\def\jj{j}
\def\JJ{\cc{J}}
\def\Zphi{Z}
\def\Zphiv{\bb{\Zphi}}

\def\nuu{\mathfrak{u}}
\def\nud{\mathfrak{u}_{0}}
\def\nun{c_{\nuu}}
\def\rhork{\kullb}
\def\GH{\mbox{GH}}
\def\HYP{\mbox{HYP}}
\def\NIG{\mbox{NIG}}
\def\IR{{\rm I\!R}}
\def\taggr{b}
\def\penm{\boldsymbol{m}}
\def\Crlp{\cc{R}}

\def\Mh{M}
\def\Mht{\Mh^{c}}

\def\Mhh{\Mh^{-}}
\def\Mhc{G}
\def\Lh{L_{1}}
\def\Uh{\cc{U}}
\def\wloc{w}
\def\Bias{B}
\def\bias{b}
\def\ExpzetaU{\Expzeta_{1}}
\def\vpci{\vp_{i,0}}
\def\IFci{\IF_{i,0}}

\def\erqb{\Circle_{\rdb}}
\def\erqm{\Circle_{\rdm}}
\def\errqm{\errm^{*}}
\def\errqb{\errb^{*}}
\def\Nsize{N}
\def\VVD{\VV_{1}}
\def\AA{A}
\def\Wloc{W}

\def\tups{\pen_{0}}
\def\rupd{\rr_{\circ}}
\def\VVb{\VVc}
\def\BP{B}
\def\bp{b}

\def\gps{s}
\def\GK{\cc{G}}

\def\zzGP{\zz_{\GP}}

\def\entrlq{\entrl_{1}}
\def\entrlg{\entrl_{2}}
\def\kb{k^{*}}

\def\rderr{\chi}
\def\Excgr{\diamondsuit}
\def\Excgrb{\diamondsuit^{*}}
\def\Thetat{\bar{\Theta}}
\def\biasGP{\bb{\bias}_{\GP}}
\def\QL{W}
\def\QLG{\mathcal{W}}
\def\BPGP{\QLG_{\GP}}
\def\BBGP{\BB_{\GP}}

\def\xxn{\xx_{\nsize}}
\def\fisGP{\mathtt{w}_{\GP}}
\def\risktGP{\riskt_{\GP}}

\def\dimq{q}
\def\nul{\mathrm{o}}
\def\Thetan{\Theta_{\nul}}
\def\thetavn{\thetav_{\nul}}
\def\thetavsn{\thetavs_{\nul}}
\def\tilden#1{\tilde{#1}_{\nul}}
\def\tildeGP#1{\tilde{#1}_{\GP}}
\def\xivn{\xiv_{\nul}}
\def\xivrGP{\xivr_{\GP}}
\def\DPcc{\DP_{\nul}}
\def\DPnGP{\DP_{1,\GP}}
\def\DPnGPr{\breve{\DP}_{1,\GP}}
\def\nablan{\nabla_{\nul}}
\def\scoren{\score_{\nul}}
\def\AnGP{A_{\nul,\GP}}

\def\testst{T}
\def\TGP{\testst_{\GP}}

\def\VPD{\VP_{2}}

\def\entrlB{\entrl_{1}}
\def\SB{W}
\def\dimq{q}
\def\QQ{\mathbb{H}}
\def\QQg{\QQ_{2}}
\def\QQq{\QQ_{1}}
\def\FF{F}
\def\LaGP{\La_{\GP}}

\def\zzQ{\zz_{\QQ}}
\def\zzAA{\zz_{\FF}}
\def\zzAAA{\zz_{\FF,\SB}}
\def\cdimc{\cdima}
\def\fisGP{\fis_{\GP}}
\def\rdomegab{\rdomega^{*}}

\def\lambdaGP{\lambda_{\GP}}
\def\wGP{\mathrm{w}_{\GP}}

\def\uudm{\mathtt{w}}
\def\lambdaB{\lambda_{\BB}}

\def\lambdav{\bb{\lambda}}
\def\etavd{\etav_{\circ}}
\def\thetavb{\breve{\thetav}}
\def\vthetavd{\Ec \vthetav}
\def\Covd{S_{\circ}}
\def\Covpostd{\Covpost_{\circ}}
\def\IS{\mathcal{I}}
\def\etas{\eta^{*}}
\def\Po{\operatorname{Po}}
\def\IF{\Bbb{F}}
\def\etavb{\bar{\etav}}
\def\etavd{\etav^{\circ}}
\def\Pc{\P^{\circ}}
\def\xxn{\xx_{\nsize}}
\def\CRd{\CR^{\circ}}

\def\CONST{\mathtt{C} \hspace{0.1em}}

\def\dimB{\mathtt{p}_{\BB}}

\def\nub{\nu}
\def\VPD{\VP_{2}}
\def\SB{W}
\def\dimq{q}
\def\dimqb{\dimq^{*}}
\def\QQ{\mathbb{H}}
\def\QQg{\QQ_{2}}
\def\QQq{\QQ_{1}}
\def\FF{F}

\def\qq{z}
\def\qqBB{\qq_{\BB}}
\def\qqQ{\qq_{\QQ}}
\def\qqAA{\qq_{\FF}}
\def\qqAAA{\qq_{\FF,\SB}}
\def\rderr{\chi}
\def\Excgr{\diamondsuit}

\def\Ccb{m}
\def\Ccm{m}
\def\BB{B}

\def\vthetavd{\vthetav^{\circ}}
\def\Indru{\Ind_{\rups}}

\def\BBh{U}
\def\betav{\bb{\beta}}
\def\DD{U}
\def\hsp{\tau}
\def\fiD{a}

\def\Prior{\Pi}
\def\prior{\pi}

\def\In{\mathcal{I}}
\def\KK{\cc{K}}
\def\qqu{\qq^{*}}
\def\ws{\omega}

\newcommand{\tobedone}[1]{\par\textbf{\color{red}To be done:} {\color{magenta}#1}}

\renewcommand{\(}{$\,}
\renewcommand{\)}{\,$}

\def\nquad{\hspace{-1cm}}
\def\eqdef{\stackrel{\operatorname{def}}{=}}
\def\tod{\stackrel{d}{\longrightarrow}}
\def\tow{\stackrel{w}{\longrightarrow}}
\def\toP{\stackrel{\P}{\longrightarrow}}

\newcommand{\cc}[1]{\mathscr{#1}}
\newcommand{\bb}[1]{\boldsymbol{#1}}

\renewcommand{\bar}[1]{\overline{#1}}
\renewcommand{\hat}[1]{\widehat{#1}}
\renewcommand{\tilde}[1]{\widetilde{#1}}

\renewcommand{\Gamma}{\varGamma}
\renewcommand{\Pi}{\varPi}
\renewcommand{\Sigma}{\varSigma}
\renewcommand{\Delta}{\varDelta}
\renewcommand{\Lambda}{\varLambda}
\renewcommand{\Psi}{\varPsi}
\renewcommand{\Phi}{\varPhi}
\renewcommand{\Theta}{\varTheta}
\renewcommand{\Omega}{\varOmega}
\renewcommand{\Xi}{\varXi}
\renewcommand{\Upsilon}{\varUpsilon}
\def\nn{\nonumber \\}

\def\suml{\sum\limits}
\def\supl{\sup\limits}
\def\maxl{\max\limits}
\def\infl{\inf\limits}
\def\intl{\int\limits}
\def\liml{\lim\limits}
\def\Cov{\operatorname{Cov}}
\def\Var{\operatorname{Var}}
\def\arginf{\operatornamewithlimits{arginf}}
\def\argsup{\operatornamewithlimits{argsup}}
\def\argmax{\operatornamewithlimits{argmax}}
\def\argmin{\operatornamewithlimits{argmin}}
\def\val{\operatorname{val}}

\def\D{\boldsymbol{D}}
\def\dd{\operatorname{d}}
\def\tr{\operatorname{tr}}
\def\I{I\!\!I}
\def\R{I\!\!R}
\def\E{I\!\!E}
\def\P{I\!\!P}
\def\X{\mathfrak{X}}
\def\kappa{\varkappa}
\def\Const{\mathrm{Const.} \,}
\def\cdt{\boldsymbol{\cdot}}
\def\tm{\!\times\!}
\def\T{\top}
\def\diag{\operatorname{diag}}
\def\diam{\operatorname{diam}}
\def\rank{\operatorname{rank}}
\def\loc{\operatorname{loc}}

\def\av{\bb{a}}
\def\bv{\bb{b}}
\def\cv{\bb{c}}
\def\dv{\bb{d}}
\def\ev{\bb{e}}
\def\fv{\bb{f}}
\def\gv{\bb{g}}
\def\hv{\bb{h}}
\def\iv{\bb{i}}
\def\jv{\bb{j}}
\def\kv{\bb{k}}
\def\lv{\bb{l}}
\def\mv{\bb{m}}
\def\nv{\bb{n}}
\def\ov{\bb{o}}
\def\pv{\bb{p}}
\def\qv{\bb{q}}
\def\rv{\bb{r}}
\def\sv{\bb{s}}
\def\tv{\bb{t}}
\def\uv{\bb{u}}
\def\vv{\bb{v}}
\def\wv{\bb{w}}
\def\xv{\bb{x}}
\def\yv{\bb{y}}
\def\zv{\bb{z}}

\def\Cv{\bb{C}}
\def\Gv{\bb{G}}
\def\Mv{\bb{M}}
\def\Sv{\bb{S}}
\def\Uv{\bb{U}}
\def\Xv{\bb{X}}
\def\Yv{\bb{Y}}
\def\Zv{\bb{Z}}

\def\alphav{\bb{\alpha}}
\def\epsv{\bb{\varepsilon}}
\def\etav{\bb{\eta}}
\def\gammav{\bb{\gamma}}
\def\varepsilonv{\bb{\varepsilon}}
\def\phiv{\bb{\phi}}
\def\psiv{\bb{\psi}}
\def\tauv{\bb{\tau}}
\def\upsilonv{\bb{\upsilon}}
\def\xiv{\bb{\xi}}
\def\zetav{\bb{\zeta}}

\def\Psiv{\bb{\Psi}}
\def\CONST{\mathtt{C}}

\def\itemv{\vfill\item}
\newenvironment{myslide}[1]
    {\begin{frame}\frametitle{#1}\vfill}
    {\vfill\end{frame}}

\def\vsp{\vspace{0.05\textheight} \vfill}
\def\summarysign{\resizebox{0.08\textwidth}{0.08\textheight}{\includegraphics{summary}}\,}
\def\nix{}
\def\wpu{$\bullet$}

\def\btri{\vfill{\( \blacktriangleright \) }}
\def\btrir{\vfill{\( \blacktriangleright \) }}

\newcommand{\mygraphics}[3]{\begin{center}
    \resizebox{#1\textwidth}{#2\textheight}{\includegraphics{#3}}
    \end{center}
}

\newcommand{\mybox}[3]{\begin{center}
    \resizebox{#1\textwidth}{#2\textheight}{#3}
    \end{center}
}

\newenvironment{eqnh}
{
    \setbeamercolor{postit}{fg=black,bg=hellgelb} 
    \begin{beamercolorbox}[center,wd=\textwidth]{postit} 
    \begin{eqnarray*}}
    {\end{eqnarray*}\end{beamercolorbox}
}


\numberwithin{equation}{section}
\numberwithin{figure}{section}

\renewcommand{\textfraction}{0.00}
\renewcommand{\topfraction}{1}
\renewcommand{\bottomfraction}{1}

\def\L{\mathbb{L}}

\def\RR{\mathbb{R}}
\def\dimtotal{p}
\def\dimp{q}
\def\dimn{\dimtotal_{n}}
\def\Proj{\Pi_{0}}
\def\Cc{m}
\def\Upsilons{\Ups_{0}}
\def\DPcb{B}
\def\DPrd{\DP_{1}}

\def\Mn{M_{\nsize}}
\def\bA{\breve{A}}
\def\cA{\bA_{\dimh}}
\def\Ij{\mathcal{I}}
\def\NU{\mathbb{H}}
\def\epv{\bb{e}}
\def\Edn{{E}}

\def\msize{m}
\def\scorem{\score_{\msize}}
\def\scorerm{\scorer_{\msize}}
\def\HHcm{\HH_{\msize}}
\def\Acm{A_{\msize}}
\def\DPcm{\DP_{\msize}}
\def\DPrcm{\DPr_{\msize}}
\def\DFcm{\DF_{\msize}}
\def\LLm{\LL_{\msize}}
\def\etavm{\etav_{\msize}}
\def\etavsm{\etavs_{\msize}}
\def\xivrm{\xivr_{\msize}}
\def\xivm{\xiv_{\msize}}
\def\upsilonvm{\upsilonv_{\msize}}
\def\upsilonvsm{\upsilonv_{\msize}^{*}}
\def\Upsm{\Ups_{\msize}}
\def\thetavdm{\thetavd_{\msize}}
\def\thetavsm{\thetav_{\msize}^{*}}
\def\vthetavbm{\vthetavb_{\msize}}
\def\Covpostm{\Covpost_{\msize}}
\def\errSieveParam{\alpha_{\msize}}
\def\errSieveFisher{\beta_\msize}

\def\eps{\varepsilon}
\def\Fs{\cc{F}}
\def\fpp{\mathrm{h}}

\section{Introduction}
\label{seq: introduction}
  The prominent Bernstein -- von Mises (BvM) theorem claims that the posterior measure is asymptotically normal with the mean close to the maximum likelihood estimator (MLE) and the posterior variance is nearly the inverse of the total Fisher information matrix.
  The BvM result provides a theoretical background for Bayesian computations of the MLE and its variance.
  Also it justifies usage of elliptic credible sets based on the first two moments of the posterior.
  The classical version of the BvM Theorem is stated for the standard parametric setup with a fixed parametric model and large samples; see \cite{LeCam1990, VanDerVaart2000} for a detailed overview.
  However, in modern statistic applications one often faces very complicated models involving a lot of parameters and with a limited sample size.
  This requires an extension of the classical results to such non-classical situation.
  We mention \cite{Cox1993, Fr1999, Gh1999, Jo2012} and references therein for some special phenomena arising in the Bayesian analysis when the parameter dimension increases.
  Already consistency of the posterior distribution in nonparametric and semiparametric models is a nontrivial problem; cf. \cite{Schwartz1965} and \cite{Barron1996}.
  Asymptotic normality of the posterior measure for these classes of models is even more challenging; see e.g. \cite{Shen2002}.
  Some results for particular semi and nonparametric problems are available from \cite{Kim2004, Kim2006, Le2011, CaNi2013}.
  \cite{ChKo2008} obtained a version of the BvM statement based on a high order expansion of the profile sampler.
  The recent paper \cite{BiKl2012} extends the BvM statement from the classical parametric case to a rather general i.i.d. framework.
  \cite{Ca2012} studies the semiparametric BvM result for Gaussian process functional priors. In \cite{RiRo2012} semiparametric BvM theorem is derived for linear functionals of density and in forthcoming work \cite{CaRo2013} the result is generalized to a broad class of models and functionals. 
  However, all these results are limited to the asymptotic setup and to some special classes of models like i.i.d. or Gaussian.

  In this paper we reconsider the BvM result for the parametric component of a general semiparametric model.
  An important feature of the study is that the sample size is fixed, we proceed with just one sample.
  A finite sample theory is especially challenging because the most of notions, methods and tools in the classical theory are formulated in the asymptotic setup with the growing sample size. Only few finite sample general results are available; see e.g. the recent paper \cite{BoMa2011}.
  This paper focuses on the semiparametric problem when the full parameter is large or infinite dimensional but the target is low dimensional.
  In the Bayesian framework, the aim is the marginal of the posterior corresponding to the target parameter; cf. \cite{Ca2012}.
  Typical examples are provided by functional estimation, estimation of a function at a point, or simply by estimating a given subvector of the parameter vector.
  An interesting feature of the semiparametric BvM result is that the nuisance parameter appears only via the effective score and the related efficient Fisher information; cf. \cite{BiKl2012}.
  The methods of study heavily rely on the notion of the hardest parametric submodel.
  In addition, one assumes that an estimate of the nuisance parameter is available which ensures a certain accuracy of estimation; see \cite{ChKo2008} or \cite{BiKl2012}.
  This essentially simplifies the study but does not allow to derive a qualitative relation between the full dimension of the parameter space and the total available information in the data.

  Some recent results study the impact of a growing parameter dimension \( \dimn \) on the quality of Gaussian approximation of the posterior.
  We mention \cite{Gh1999,Gh2000}, \cite{BoGa2009}, \cite{Jo2012} and \cite{Bo2011} for specific examples.
  See the discussion after Theorem~\ref{theorem: bvmIid} below for more details.

  In this paper we show that the \emph{bracketing} approach of \cite{Sp2011} can be used for obtaining a finite sample semiparametric version of Bernstein -- von Mises theorem even if the full parameter dimension grows with the sample size.
  The ultimate goal of this paper is to quantify the so called critical parameter dimension for which the BvM result can be applied.
  Our approach neither relies on a pilot estimate of the nuisance and target parameter nor it involves the notion of the hardest parametric submodel.
  In the case of finite dimensional nuisance the obtained results only require some smoothness of the log-likelihood function, its finite exponential moments, and some identifiability conditions.
  Further we specify this result to the i.i.d. setup and show that the imposed conditions are satisfied if \( \dimn^{3}/\nsize \) is small.
  We present an example showing that the dimension \( \dimn = O(\nsize^{1/3}) \) is indeed critical and the BvM result starts to fail if \( \dimn \) grows over \( \nsize^{1/3} \). 
  If the nuisance is infinite dimensional then additionally some smoothness of nonparametric part is required. 
  We state the general BvM results and show that the smoothness \(s > 3/2\) is sufficient for the validity of the BvM result for linear and generalized linear models with the nuisance parameter 
  from Sobolev class.

  Now we describe our setup.
  Let \( \Yv \) denote the observed random data, and \(\P\) denote the data distribution. The parametric statistical model assumes that the unknown data distribution \(\P\) belongs to a given parametric family \( (\P_{\upsilonv}) \):
  \begin{EQA}[c]
    \Yv \sim \P = \P_{\upsilonvs} \in (\P_{\upsilonv}, \, \upsilonv \in \Ups),
  \end{EQA}
  where \(\Ups\) is some parameter space and \(\upsilonvs \in \Ups\) is the true value of parameter.
  In the semiparametric framework, one attempts to recover only a low dimensional component \(\thetav\) of the whole parameter \(\upsilonv\).
  This means that the target of estimation is
  \begin{EQA}[c]
    \thetavs \eqdef \Proj \upsilonvs,
  \end{EQA}
  for some mapping \(\Proj : \Ups \to \R^{\dimp}\), and \(\dimp \in \N\) stands for the dimension of the target.
  Usually in the classical semiparametric setup, the vector \(\upsilonv\) is represented as \(\upsilonv = (\thetav, \etav)\), where \(\thetav\) is the target of analysis while \(\etav\) is the \emph{nuisance parameter}. We refer to this situation as \((\thetav, \etav)\)-setup and our presentation follows this setting. An extension to the \(\upsilonv\)-setup with \(\thetav = \Proj \upsilonv\) is straightforward. Also for simplicity we first develop our results for the case when the total parameter space \(\Ups\) is a subset of the Euclidean space of dimensionality \(\dimtotal\).

  Another issue addressed in this paper is the model misspecification. In the most of practical problems, it is unrealistic to expect that the model assumptions are exactly fulfilled, even if some rich nonparametric models are used. This means that the true data distribution \(\P\) does not belong to the considered family \((\P_{\upsilonv}\, , \upsilonv \in \Ups)\). The ``true'' value \(\upsilonvs\) of the parameter \(\upsilonv\) can defined by
  \begin{EQA}[c]
    \upsilonvs = \argmax_{\upsilonv \in \Ups} \E \LL(\upsilonv),
  \label{upssdefBs}
  \end{EQA}
  where
  \(\LL(\upsilonv) = \log \frac{d\P_{\upsilonv}}{d\PDOM}(\Yv)\) is the log-likelihood
  function of the family \( (\P_{\upsilonv}) \)
  for some dominating measure \( \PDOM \).  
  Under model misspecification, \(\upsilonvs\) defines the best parametric fit to \(\P\) by the considered family; cf. \cite{Chernozhukov2003}, \cite{KlVa2006, KlVa2012} and references therein.
  The target \(\thetavs\) is defined by the mapping \(\Proj\):
  \begin{EQA}[c]
    \thetavs \eqdef \Proj \upsilonvs.
  \label{projSemiTrue}
  \end{EQA}

  Now we switch to the Baeysian set-up.
  Let \(\prior\) be a prior measure on the parameter set \(\Upsilon\).
  Below we study the properties of the posterior measure which is the random measure on
  \(\Upsilon\) describing the conditional distribution of \(\upsilonv\) given \(\Yv\)
  and obtained by normalization of the product
  \(\exp \bigl\{\LL(\upsilonv) \bigr\} \prior(d\upsilonv)\).
  This relation is usually written as
  \begin{EQA}
    \upsilonv \mid \Yv
    & \propto &
    \exp \bigl\{ \LL(\upsilonv) \bigr\} \, \prior(d\upsilonv) .
  \label{bayesForm}
  \end{EQA}
  An important feature of our analysis is that \( \LL(\upsilonv)\) is not assumed to be the true log-likelihood. This means that a model misspecification is possible and the underlying data distribution can be beyond the considered parametric family. In this sense, the Bayes formula \eqref{bayesForm} describes a \emph{quasi posterior}; \cite{Chernozhukov2003}.
  Below we show that smoothness of the log-likelihood function \( \LL(\upsilonv)\) ensures a kind of a Gaussian approximation of the posterior measure.
  Our focus is to describe the accuracy of such approximation as a function of the parameter dimension \( \dimtotal \) and the other important characteristics of the model.

  We suppose that the prior measure \(\prior\) has a positive density
  \(\priorden(\upsilonv)\) w.r.t. to the Lebesgue measure on
  \(\Upsilon\): \(\prior(d\upsilonv) = \priorden(\upsilonv) d \upsilonv\). Then
  \eqref{bayesForm} can be written as
  \begin{EQA}
    \upsilonv \mid \Yv
    & \propto &
    \exp \bigl\{ \LL(\upsilonv) \bigr\} \, \priorden(\upsilonv) .
  \label{bayesFormDen}
  \end{EQA}
  The famous Bernstein -- von Mises (BvM) theorem claims that the posterior centered by any efficient estimator \( \tilde{\upsilonv} \) of the parameter \( \upsilonvs \) (for example MLE) and scaled by the total Fisher information matrix is nearly standard normal:
  \begin{EQA}[c]
  \label{bvmFullAsympt}
    \DFc (\upsilonv - \tilde{\upsilonv}) \mid \Yv
    \tow
    \ND(0,\Id_{\dimtotal}) \, ,
  \end{EQA}
  where \(\Id_{\dimtotal}\) is an identity matrix of dimension \(\dimtotal\).

  An important feature of the posterior distribution is that it is entirely known and can be numerically assessed.
  If we know in addition that the posterior is nearly normal, it suffices to compute its mean and variance for building the concentration and credible sets. 
  The BvM result does not require prior distribution to be proper and the phenomenon can be observed in the case of improper priors as well (for examples, see \cite{Bochkina2012}).

  In this work we investigate the properties of the posterior distribution for the target parameter 
  \( \vthetav = \Proj \upsilonv \).
  In this case \eqref{bayesFormDen} can be written as
  \begin{EQA}
    \vthetav \mid \Yv
    & \propto &
    \int\exp \bigl\{ \LL(\upsilonv) \bigr\} \, \priorden(\upsilonv) d \etav.
  \label{bayesFormDenPartial}
  \end{EQA}
  The BvM result in this case transforms into
  \begin{EQA}[c]
  \label{BvmSemiAsympt}
    \DPrc (\vthetav - \tilde{\thetav}) \mid \Yv
    \tow
    \ND(0,\Id_{\dimp}) \, ,
  \end{EQA}
  where \(\Id_{\dimp}\) is an identity matrix of dimension \(\dimp\),
  \( \tilde{\thetav} = \Proj \tilde{\upsilonv} \),
  and \( \DPrc^{2} \) is given in \eqref{fisherPartial}.

  We consider two important classes of priors, namely non-informative and flat priors.
  Our goal is to show under mild conditions that the posterior distribution of the target parameter \eqref{bayesFormDenPartial} is close to a prescribed Gaussian law even for finite samples.
  The other important issue is to specify the conditions on the sample size and the dimension of the parameter space for which the BvM result is still applicable.

\section{BvM Theorem with a finite dimensional nuisance}
\label{sec: bvmTargetFinite}
\label{mainDefinitions}

This section presents our main results for the case of a finite dimensional parameter \( \upsilonv \),
  i.e. \(\dim(\Upsilon) = \dimtotal < \infty\). One of the main elements of our construction is a \(\dimtotal \times \dimtotal \) matrix \( \DFc^{2} \) which is defined similarly to the Fisher information matrix:
  \begin{EQA}
  \label{fisherFull}
    \DFc^{2}
    & \eqdef &
    - \nabla^{2} \E \LL(\upsilonvs) .
  \end{EQA}
  Here and in what follows we implicitly assume that the log-likelihood function
  \( \LL(\upsilonv) \) is sufficiently smooth in \( \upsilonv \),
  \( \nabla \LL(\upsilonv) \) stands for its gradient and
  \( \nabla^{2} \E \LL(\upsilonv) \) for the Hessian of the expectation
  \( \E \LL(\upsilonv) \), and the true value \( \upsilonvs \) is due to \eqref{upssdefBs}.
  Also define the score 
  \begin{EQA}[c]
    \xiv
    \eqdef
    \DFc^{-1} \nabla \LL(\upsilonvs).
  \label{effScore}
  \end{EQA}
  The definition of \( \upsilonvs \) implies \( \nabla \E \LL(\upsilonvs) = 0 \) and hence,
  \( \E \xiv = 0 \).

  For the \( (\thetav,\etav) \)-setup, we
  consider the block representation of the vector \( \nabla \LL(\upsilonvs) \) and of the matrix and \(
  \DFc^{2} \) from \eqref{fisherFull}:
  \begin{EQA}[c]
    \nabla \LL(\upsilonvs)
    =
    \left(
      \begin{array}{c}
        \score_{\thetav} \\
        \score_{\etav}
      \end{array}
    \right),
    \quad
    \DFc^{2}
    =
    \left(
      \begin{array}{cc}
        \DPc^{2} & \Ac \\
        \Ac^{\T} & \HHc^{2}
      \end{array}
    \right).
  \label{blockGradFischer}
  \end{EQA}
  Define also the \( \dimp \times \dimp \) matrix \( \DPrc^{2} \) and
  random vectors \( \scorer_{\thetav} , \xivr \in \R^{\dimp} \) as
  \begin{EQA}
  \label{fisherPartial}
    \DPrc^{2}
    & \eqdef &
    \DPc^{2} - \Ac \HHc^{-2} \Ac^{\T} ,
    \\
    \scorer_{\thetav}
    & \eqdef &
    \score_{\thetav} - \Ac \HHc^{-2} \score_{\etav} ,
    \\
    \xivr
    & \eqdef &
    \DPrc^{-1} \scorer_{\thetav} .
  \end{EQA}
  The \( \dimp \times \dimp \) matrix \( \DPrc^{2} \) is usually called the efficient
  Fisher information matrix, while the random vector \( \xivr \in \R^{\dimp} \) is the efficient score.
  Everywhere in the text for a vector \(\av\) we denote by \(\|\av\|\) its Euclidean norm and for a matrix \(A\) we denote by \(\|A\|\) its operator norm.

\subsection{Conditions}
  \label{sec: conditions}
  Our results assume a number of conditions to be satisfied.
  The list is essentially as in \cite{Sp2011}, one can find there some discussion and
  examples showing that the conditions are not restrictive and are fulfilled in most of
  classical models used in statistical studies like i.i.d., regression or Generalized Linear
  models. 
  The conditions are split into local and global.
  The local conditions only describe the properties of the process \(\LL(\upsilonv)\)
  for \(\upsilonv \in \Upsilons(\rups) \) with some fixed value \(\rups\):
  \begin{EQA}[c]
    \Upsilons(\rups)
    \eqdef
    \bigl\{\upsilonv \in \Ups \colon \|\DFc (\upsilonv - \upsilonvs)\| \le \rups\bigr\}.
  \end{EQA}
  The global conditions have to be fulfilled on the whole \(\Ups\).
  Define the stochastic component \(\zeta(\upsilonv)\) of \(\LL(\upsilonv)\):
  \begin{EQA}[c]
    \zeta(\upsilonv) \eqdef \LL(\upsilonv) - \E \LL(\upsilonv).
  \end{EQA}
  We start with some exponential moments conditions.

  \begin{description}
    \item[\( \bb{(E\!D_{0})} \)]
      There exists a constant \(\nunu>0\), a positive symmetric
      \(\dimtotal\times\dimtotal\) matrix \(\VFc^{2}\)
      satisfying
      \( \Var\{\nabla\zeta(\upsilonvs)\} \le \VFc^{2} \),
      and a constant \( \gm > 0 \) such that
      \begin{EQA}[c]
        \sup_{\gammav \in \R^{\dimtotal}} \log\E \exp\left\{
              \mubc \frac{\langle \nabla \zeta(\upsilonvs),\gammav \rangle}
              {\| \VFc \gammav \|}
              \right\}
        \le
        \frac{\nunu^{2} \mubc^{2}}{2}, ~ |\mubc| \le \gm.
      \end{EQA}

    \item[\( \bb{(E\!D_{2})} \)]
      There exists a constant \( \rhor > 0 \) and for each \(\rr > 0\) a constant \(\gm(\rr) > 0\) such that for
      all \( \upsilonv \in \Upsilons(\rr) \):
      \begin{EQA}[c]
        \sup_{\gammav_{1}, \gammav_{2} \in \R^{\dimtotal}} \log \E \exp\left\{
        \frac{\mubc}{\rhor} \frac{\gammav_{1}^{\T} \nabla^{2} \zeta(\upsilonv) \gammav_{2}}{\|\DFc \gammav_{1}\| \cdot \|\DFc \gammav_{2}\|}
        \right\}
        \le
        \frac{\nunu^{2} \mubc^{2}}{2}, ~ |\mubc| \le \gm(\rr).
      \end{EQA}
  \end{description}

  The next condition is needed to ensure some smoothness properties of expected log-likelihood \(\E \LL(\upsilonv)\) in the local zone \(\upsilonv \in \Upsilons(\rups)\). Define
  \begin{EQA}[c]
    \DFc^{2}(\upsilonv) 
    \eqdef
    -\nabla^{2} \E \LL(\upsilonv).
  \end{EQA}
  Then \(\DFc^{2} = \DFc^{2}(\upsilonvs)\).

  \begin{description}
    \item[\( \bb{(\cc{L}_0)} \)]
      There exists a constant \(\rddelta(\rr)\) such that it holds on
      the set \(\Upsilons(\rr)\) for all \(\rr \le \rups\)

      \begin{EQA}[c]
        \bigl|\DFc^{-1} \DFc^{2}(\upsilonv)\DFc^{-1} - \Id_{\dimtotal}\bigr|
        \le
        \rddelta(\rr).
      \end{EQA}
  \end{description}

  The global identification condition is:
  \begin{description}
    \item[\( \bb{(\cc{L}{\rr})} \)]
      For any \(\rr\) there exists a value \(\gmi(\rr) > 0\),
      such that \(\rr \gmi(\rr) \to \infty\), \(\rr \to \infty\) and
      \begin{EQA}[c]
        -\E \LL(\upsilonv,\upsilonvs)
        \ge
        \rr^{2} \gmi(\rr) \quad \text{for all \( \upsilonv \) with } 
        \rr = \|\DFc (\upsilonv - \upsilonvs)\|.
      \end{EQA}
  \end{description}

  Finally we specify the identifiability conditions. 
  We begin by representing the information and the covariance matrices in block form:
  \begin{EQA}
    \DFc^{2}
    =
    \left(
      \begin{array}{cc}
        \DPc^{2} & \Ac \\
        \Ac^{\T} & \HHc^{2} \\
      \end{array}
    \right),
    & \quad &
    \VFc^{2}
    =
    \left(
      \begin{array}{cc}
        \VPc^{2} & B_0
        \\
        B_0^{\T}  & \VHc^{2}
        \\
      \end{array}
    \right) .
  \label{blockGradFisherCond}
  \end{EQA}
  The \emph{identifiability conditions} in \cite{Sp2011} ensure that the matrix \(\DFc^{2}\) is positive and satisfied \(\fis^{2} \DFc^{2} \ge \VFc^{2}\) for some \(\fis > 0\). 
  Here we restate these conditions in the special block form which is specific for the \((\thetav, \etav)\)-setup.

  \begin{description}
    \item[\( \bb{(\AssId)} \)]
      There are constants \( \fis > 0 \) and \( \corrDF < 1 \) such that
      \begin{EQA}[c]
        \fis^{2} \DPc^{2}
        \ge
        \VPc^{2},
        \qquad
        \fis^{2} \HHc^{2}
        \ge
        \VHc^{2} ,
        \qquad
        \fis^{2} \DFc^{2}
        \ge
        \VFc^{2} .
      \label{regularityCond}
      \end{EQA}
      and
      \begin{EQA}[c]
        \| \DPc^{-1} \Ac \HHc^{-2} \Ac^{\T} \DPc^{-1} \|
        \le
        \corrDF.
      \label{regularityCorrCond}
      \end{EQA}
    \end{description}
  The quantity \(\corrDF\) bounds the angle between the target and nuisance subspaces in the tangent space. The regularity condition \((\AssId)\) ensures that this angle is not too small and hence, the target and nuisance parameters are identifiable. 
  In particular, the matrix \(\DPrc^{2}\) from \eqref{fisherPartial} is well posed under \((\AssId)\). 
  The bounds in \eqref{regularityCond} are given with the same constant \(\fis\) only for simplifying the notation. One can show that the last bound on \(\DFc^{2}\) follows from the first two and \eqref{regularityCorrCond} with another constant \(\fis'\) depending on \(\fis\) and \(\corrDF\) only.

\subsection{The main results}
  First we state the BvM result about the properties of the \( \vthetav \)-posterior
  given by \eqref{bayesFormDenPartial} in case of uniform prior that is,
  \(\priorden(\upsilonv) \equiv 1\) on \(\Upsilon\).  
  Define
  \begin{EQA}[c]
    \vthetavb
    \eqdef
    \E \bigl( \vthetav \cond \Yv \bigr) ,
    \qquad
    \Covpost^{2}
    \eqdef
    \Cov (\vthetav \cond \Yv)
    \eqdef
    \E \bigl\{ (\vthetav - \vthetavb) (\vthetav - \vthetavb)^{\T} \cond \Yv \bigr\}.
  \label{posteriorMoments}
  \end{EQA}
  Also define
  \begin{EQA}[c]
    \thetavd
    \eqdef
    \thetavs + \DPrc^{-1} \xivr .
  \label{effectiveTargetEstimate}
  \end{EQA}
  This random point can be viewed as first order approximation of the profile MLE \( \tilde{\thetav} \). 
  Below we present a version of the BvM result in the considered nonasymptotic setup
  which claims that \( \vthetavb \) is close to
  \( \thetavd \), \( \Covpost^{2} \) is nearly equal to \( \DPrc^{-2} \),
  and \( \DPrc \bigl( \vthetav - \thetavd \bigr) \) is nearly standard normal conditionally on \( \Yv \).

  We suppose that a large constant \( \xx \) is fixed
  which specifies random events \( \Omega(\xx) \) of \emph{dominating probability}.
  We say that a generic random set \( \Omega(\xx) \) is of dominating probability if
  \begin{EQA}[c]
    \P\bigl( \Omega(\xx) \bigr)
    \ge
    1 - \CONST \ex^{-\xx} .
  \label{setDominProb}
  \end{EQA}
  The notation \( \CONST \) for a generic absolute constant and
  \( \xx \) for a positive value ensuring that \( \ex^{-\xx} \) is negligible.
  By \( \Omega(\xx) \) we denote a random event of dominating probability with
  \( \P\bigl( \Omega(\xx) \bigr) \ge 1 - \CONST \ex^{-\xx} \). 
  The exact values of \(\CONST\) will be specified in each particular case.
  The formulation of the results also involve the radius \( \rups \) and the spread 
  \( \spread(\rups,\xx) \).
  The radius \( \rups \) separates the local zone \( \Upss(\rups) \) which is a vicinity of the central
  point \( \upss \), and its complement \( \Ups \setminus \Upss(\rups) \) for which we establish a 
  large deviation result.
  The \emph{spread} value \( \spread(\rups,\xx) \) measures the quality of local approximation 
  of the log-likelihood \( L(\ups,\upss) \) by a quadratic process \( \La(\ups,\upss) \):
\begin{EQA}
	\spread(\rups, \xx) 
	& \eqdef &
      \bigl\{ \rddelta(\rups) 
      + 6 \nunu \, \qqQ(\xx) \, \rhor \bigr\}\, \rups^{2} .
\label{spreaddef}
\end{EQA}
Here the term \( \rddelta(\rups) \rups^{2} \) measures the error of a quadratic approximation of the expected 
log-likelihood \( L(\upsilonv) \) due to \( (\LL_{0}) \), while the second term 
\( 6 \nunu \, \qqQ(\xx) \, \rhor \, \rups^{2} \) controls the stochastic term and involves the entropy 
of the parameter space which is involved in the definition of \( \qqQ(\xx) \).  
A precise formulation is given in Theorem~\ref{theorem: basicQuadApprox} below.

  \begin{theorem}
  \label{theorem: bvmTarget}
    Suppose the conditions of Section~\ref{sec: conditions}. 
    Let the prior be uniform on \(\Upsilon\). 
    Then there exists a random event \( \Omega(\xx) \) of  probability at least \(1 - 4 \ex^{-\xx}\) such that it holds on \( \Omega(\xx) \)
    \begin{EQA}
    \label{meanPostBvm}
      \| \DPrc (\vthetavb - \thetavd) \|^{2}
      & \le &
      4 \spread(\rups, \xx) + 16 \ex^{-\xx},
      \\
      \bigl\| \Id_{\dimp} - \DPrc \Covpost^{2} \DPrc \bigr\|
      & \le &
      4 \spread(\rups, \xx) + 16 \ex^{-\xx},
    \label{covPostBvm}
    \end{EQA}
    where \(\vthetavb\) and \(\Covpost^{2}\) are from \eqref{posteriorMoments}.
    
    Moreover, on \(\Omega(\xx)\) for any measurable set \( A \subset \R^{\dimp} \)
    \begin{EQA}
      && \nquad
      \exp \bigl( -2\spread(\rups, \xx) - 8\ex^{-\xx} \bigr)
      \P\bigl( \gammav \in A \bigr)
      - \ex^{- \xx}
      \\
      & \le &
      \P\bigl( \DPrc (\vthetav - \thetavb) \in A \cond \Yv \bigr)
      \\
      \qquad
      &&
      \qquad \le
      \exp \bigl( 2\spread(\rups, \xx) + 5\ex^{-\xx} \bigr)
      \P\bigl( \gammav \in A \bigr), 
    \label{bvmResultMain}
    \end{EQA}
    where \(\gammav\) is a standard Gaussian vector  in \(\RR^{\dimp}\).
  \end{theorem}

  The condition ``\( \spread(\rups, \xx) \) is small'' yields the desirable BvM result, that is, the posterior measure after centering and standardization is close in total variation to the standard normal law. The classical asymptotic results immediately follow for many classical models (see discussion in Section~\ref{sec: bvmIid}). 
  The next corollary extends the previous result by using empirically computable objects.
  \begin{corollary}
    \label{corollary: bvmPost}
    Under the conditions of Theorem~\ref{theorem: bvmTarget} for any measurable set \( A \subset \R^{\dimp} \) a random event \( \Omega(\xx) \) of a dominating probability at least \(1 - 4 \ex^{-\xx}\)
    \begin{EQA}
      && \nquad
      \exp \bigl( -2\spread(\rups, \xx) - 8\ex^{-\xx} \bigr)
      \bigl\{\P\bigl( \gammav \in A \bigr) - \tau\bigr\}
      - \ex^{- \xx}
      \\
      & \le &
      \P\bigl( \Covpost^{-1} (\vthetav - \vthetavb) \in A \cond \Yv \bigr)
      \\
      \qquad
      &&
      \qquad \le
      \exp \bigl( 2\spread(\rups, \xx) + 5\ex^{-\xx} \bigr)
      \bigl\{\P\bigl( \gammav \in A \bigr) + \tau\bigr\}, 
    \label{bvmResultPost}
    \end{EQA}
    where 
    \(\gammav\) is a standard Gaussian vector  in \(\RR^{\dimp}\) and
\begin{EQA}
	\tau 
	& \eqdef & 
	\frac{1}{2} \Bigl(\dimp \spread^{2}(\rups, \xx)  
	+ \bigl\{ 1 + \spread(\rups, \xx) \bigr\}^{2} \spread^{2}(\rups, \xx) \Bigr) .
\label{taudefspruxx}
\end{EQA}
  \end{corollary}
  This corollary is important as in practical applications we do not know matrix \(\DPrc\) and vector \(\thetavd\), but matrix \(\Covpost^{-1}\) and vector \(\vthetavb\) can be computed by numerical computations. 
  If dimension \(\dimp\) is fixed the result becomes informative under the condition 
  ``\( \spread(\rups, \xx) \) is small''. 
  Moreover, the statement can be extended to situations when the target dimension 
  \( \dimp \) grows but \( \spread(\rups, \xx) \, \dimp^{1/2} \) is still small.

\subsection{Extension of Theorem~\ref{theorem: bvmTarget} to a flat prior}
  \label{sec: regularPrior}
  The results of Theorem~\ref{theorem: bvmTarget} for a non-informative prior can be extended to the case
  of a general prior \( \prior(d\upsilonv) \) with a density \( \priorden(\upsilonv) \)
  which is uniformly continuous and sufficiently flat on the local set \( \Upsilons(\rups)\).
  More precisely, let \( \priorden(\upsilonv) \) satisfy
  \begin{EQA}
    \sup_{\upsilonv \in \Upsilons(\rups)}
        \Bigl| \frac{\priorden(\upsilonv)}{\priorden(\upsilonvs)} - 1 \Bigr|
    & \le &
    \alpha(\rups),
    \qquad
    \sup_{\upsilonv \in \Upsilon} \frac{\priorden(\upsilonv)}{\priorden(\upsilonvs)}
    \le
    \CONST(\rups),
  \label{regularPriorDen}
  \end{EQA}
  where \( \alpha(\rups) \) is a small constant while \( \CONST(\rups) \) is any fixed constant.
  Then the results of Theorem~\ref{theorem: bvmTarget} continue to apply with an obvious correction of the approximation error.
  Indeed, for any local set \( A \subseteq \Upsilons(\rups) \) one can apply the bounds
  \begin{EQA}
    \int_{A} \exp\bigl\{ L(\upsilonv) \bigr\} \prior(\upsilonv) d\upsilonv
    & \leq &
    \ex^{ \alpha(\rups) } \prior(\upsilonvs) 
    \int_{A} \exp\bigl\{ L(\upsilonv) \bigr\} d\upsilonv,
    \\
    \int_{A} \exp\bigl\{ L(\upsilonv) \bigr\} \prior(\upsilonv) d\upsilonv
    & \geq &
    \ex^{ - \alpha(\rups) } \prior(\upsilonvs) 
    \int_{A} \exp\bigl\{ L(\upsilonv) \bigr\} d\upsilonv,
  \end{EQA}
  This particularly implies for each \( A \subset \Upsilons(\rups) \),
  \begin{EQA}
    \P_{\prior}(A \cond \Yv)
    & \leq &
    \exp \bigl\{ 2 \alpha(\rups) \bigr\} \P(A \cond \Yv) .
  \label{postBoundContPrior}
  \end{EQA}
  The tail probability of the complement \( \Upsilons^{c}(\rups) \eqdef \Upsilon \setminus \Upsilons(\rups) \) of \( \Upsilons(\rups) \) can be enlarged by \( \CONST(\rups) \) relative to the uniform prior:
  \begin{EQA}
    \int_{\Upsilons^{c}(\rups)} \exp\bigl\{ L(\upsilonv) \bigr\} \prior(\upsilonv) d\upsilonv
    & \leq &
    \CONST(\rups) \, \prior(\upsilonvs) 
    \int_{\Upsilons^{c}(\rups)} \exp\bigl\{ L(\upsilonv) \bigr\} d\upsilonv,
  \label{postTailIntContPrior}
  \end{EQA}
  hence 
  \begin{EQA}
    \P_{\prior}(\Upsilons^{c}(\rups) \cond \Yv)
    & \leq &
    \CONST(\rups) \P(\Upsilons^{c}(\rups) \cond \Yv).
  \label{postTailContPrior}
  \end{EQA}
  In particular, if the tail of the non-informative posterior satisfies 
  \( \P(\Upsilons^{c}(\rups) \cond \Yv) \leq \ex^{-\xx} \), then 
  \( \P_{\prior}(\Upsilons^{c}(\rups) \cond \Yv) \leq \CONST(\rups) \ex^{-\xx} \).
  \begin{theorem}
  \label{theorem: bvmGaussPrior}
    Suppose the conditions of Theorem~\ref{theorem: bvmTarget}.
    Let also \( \Prior = \ND(0,\GP^{-2}) \) be a Gaussian prior measure on
    \( \R^{\dimp} \) such that
    \begin{EQA}[c]
        \| \DFc^{-1} \GP^{2} \DFc^{-1} \| \, \leq \eps^{2} ,
    \label{gaussPriorFlatCond}
    \end{EQA}
    where \( \eps \) is a given constant.
    Then \eqref{postBoundContPrior} and \eqref{postTailContPrior} hold with 
    \( \CONST(\rups) \leq \exp(\| \GP \upsilonvs \|^{2}/2) \) and 
    \( \alpha(\rups) = \max\bigl\{ \eps \, \rups \| \GP \upsilonvs \|, \eps^{2} \rups^{2}/2 \bigr\} \). 
  \end{theorem}
  The result of Theorem~\ref{theorem: bvmGaussPrior} tells us that the BvM result holds 
  if the prior distribution is sufficiently flat. 

\section{Infinite dimensional nuisance}
\label{sec: infDimNuisance}
  This section describes how previous results can be extended to the case where the nuisance is infinite dimensional. 
  More specifically, 
  we consider \((\thetav, \fv)\)-setup, where \(\thetav \in \Theta \subseteq \RR^{\dimp}\) and 
  \(\fv \in \mathcal{H}\) for some Hilbert space \(\mathcal{H}\). 
  Suppose that in \(\mathcal{H}\) exists a countable basis \(\ev_{1}, \ev_{2}, \dots\). Then
  \begin{EQA}[c]
    \fv 
    = 
    \fv(\etav) 
    = 
    \sum_{j = 1}^{\infty} \eta_{j} \ev_{j} \in \mathcal{H}, 
  \end{EQA} 
  where a vector \(\etav  = \{\eta_{j}\}_{j = 1}^{\infty} \in \ell_{2}\) and 
  \(\eta_{j} = \langle \fv, \ev_{j} \rangle\).
  
  Let the likelihood for the full model be \(\LL(\thetav, \fv)\). 
  Denote with \(\upsilonv = (\thetav, \etav)\)
  \begin{EQA}[c]
    \LL(\upsilonv) = \LL(\thetav, \etav) = \LL(\thetav, \fv(\etav)). 
  \end{EQA}
  The underlying full and target parameters can be defined by maximizing the expected log-likelihood:
  \begin{EQA}[c]
    \upsilonvs = \argmax_{\upsilonv = (\thetav,\etav)} \E \LL(\upsilonv), 
    \qquad \thetavs = \Proj \upsilonvs.
  \end{EQA}
  Also define the information matrix \( \DFc^{2} \) for the full parameter \( \upsilonv \) 
  and the efficient information matrix \( \DPrc^{2} \) for the target \( \thetav \):
  \begin{EQA}
    \DFc^{2} 
    & \eqdef &
    \nabla^{2} \E[\LL(\upsilonvs)] \in Lin(\ell_{2}, \ell_{2}),
    \\
    \DPrc^{2} 
    & \eqdef &
    \left(\Proj \DFc^{-2} \Proj^\T\right)^{-1} \in \R^{\dimp \times \dimp}, 
  \end{EQA}
  where \(Lin(\ell_{2},\ell_{2})\) is the space of linear operators from \( \ell_{2}\) to \( \ell_{2}\).

  We apply the sieve approach and use an uninformative finite dimensional prior for 
  the parameters \( \thetav \) and \( \etavm \). 
  Consider a finite dimensional approximations  
  \(\etavm = \{\eta_{j}\}_{j = 1}^{m}\) of the full parameter \( \etav \) and the corresponding finite dimensional approximation of the log-likelihood:
  \begin{EQA}
    \LL (\upsilonvm) &=& \LL (\thetav, \etavm),
    \qquad 
    \upsilonvm = (\thetav, \etavm).
  \end{EQA}
Similarly to the finite dimensional case, introduce the quantities   
\begin{EQA}
    \upsilonvsm 
    &=&
    \argmax_{\upsilonvm \in \Upsm} \E \LL(\upsilonvm) ,
    \\
  %
  \label{fisherFullSemi}
    \DFcm^{2}
    & \eqdef &
    - \nabla^{2} \E \LL(\upsilonvsm) .
  \end{EQA}
  Here \( \nabla \LL(\upsilonvm) \) stands for its gradient and
  \( \nabla^{2} \E \LL(\upsilonvm) \) for the Hessian of the expectation
  \( \E \LL(\upsilonvm) \).
  Also define the score vector
  \begin{EQA}[c]
    \xivm
    \eqdef
    \DFcm^{-1} \nabla \LL(\upsilonvsm).
  \label{effScoreSieve}
  \end{EQA}
  The definition of \( \upsilonvsm \) implies \( \nabla \E \LL(\upsilonvsm) = 0 \) and hence,
  \( \E \xivm = 0 \).

  Again, we consider the block representation of the vector \( \nabla \LL(\upsilonvsm) \) and of the matrix and \(
  \DFcm^{2} \) from \eqref{fisherFullSemi}:
  \begin{EQA}[c]
    \nabla \LL(\upsilonvsm)
    =
    \left(
      \begin{array}{l}
        \score_{\thetav} \\
        \score_{\etavm}
      \end{array}
    \right),
    \quad
    \DFcm^{2}
    =
    \left(
      \begin{array}{cc}
        \DPcm^{2} & \Acm \\
        \Acm^{\T} & \HHcm^{2}
      \end{array}
    \right).
  \label{blockGradFisherSieve}
  \end{EQA}
  Define also the \( \dimp \times \dimp \) matrix \( \DPrcm^{2} \) and
  random vectors \( \scorer_{\thetav, \msize} , \xivrm \in \R^{\dimp} \) as
  \begin{EQA}
  \label{fisherSieve}
    \DPrcm^{2}
    &=&
    \DPcm^{2} - \Acm \HHcm^{-2} \Acm^{\T} ,
    \\
    \scorer_{\thetav, \msize}
    &=&
    \score_{\thetav} - \Acm \HHcm^{-2} \score_{\etavm} ,
    \\
    \xivrm
    &=&
    \DPrcm^{-1} \scorer_{\thetav, \msize}.
  \end{EQA}
  Approximation of the functional parameter \( \upsilonv \) by \( \upsilonv_{\msize} \) leads to
  two sources of bias. The first one is connected with approximation of the target parameter 
  \(\thetavs - \thetavsm\). 
  The second one is due to the difference between the efficient Fisher information 
  \(\DPrc^{2} \in \R^{\dimp \times \dimp}\) and its approximation 
  \(\DPrcm^{2} \in \R^{\dimp \times \dimp}\).
  Both errors of approximation are due to projection of the functional parameter onto the finite 
  dimensional space spanned by the first \( \msize \) basis functions. 
  The bias terms can be bounded under smoothness assumptions on the model and on the functional 
  nuisance parameter \( \fv \) using the standard methods of approximation theory.
  To avoid tedious calculus we simply assume a kind of consistency of the sieve approximation.
  \begin{description}
    \item[\( \bb{(B)} \)]
      For any \(\msize \in \mathbb{N}\) there exist constants \(\errSieveParam\) and \(\errSieveFisher\) such that
      \begin{EQA}
        \| \DPrc (\thetavs - \thetavsm) \|^{2}
        & \le &\
        \errSieveParam,
        \\
        \bigl\| \Id_{\dimp} - \DPrc \DPrcm^{-2} \DPrc \bigr\|
        & \le &\
        \errSieveFisher.
      \end{EQA}
  \end{description}
  For validity of our results we will need that the value of \(\msize\) is fixed in a proper way
  ensuring that the values of   \(\errSieveParam\) and \(\errSieveFisher\) are sufficiently small.
  These values can be upper bounded under usual smoothness conditions on \( \fv \),
  e.g. if \( \fv \) belongs to a Sobolev ball with a certain regularity;
  cf. \cite{Bo2011,BiKl2012,Ca2012}.
  See also an example of computing the quantities \( \errSieveParam \) and \( \errSieveFisher \) 
  in Section~\ref{sec: semiLinRegr} below.

  Consider a non-informative sieve prior on \( (\thetav,\etavm) \) and define
  \begin{EQA}
    \vthetavbm
    & \eqdef &
    \E \bigl( \vthetav \cond \Yv \bigr) ,
    \\
    \Covpostm^{2}
    & \eqdef &
    \Cov (\vthetav \cond \Yv)
    \eqdef
    \E \bigl\{ (\vthetav - \vthetavbm) (\vthetav - \vthetavbm)^{\T} \cond \Yv \bigr\}.
  \label{posteriorMomentsSieve}
  \end{EQA}
  Also define
  \begin{EQA}[c]
    \thetavdm
    \eqdef
    \thetavsm + \DPrcm^{-1} \xivrm.
  \label{effectiveTargetEstimate}
  \end{EQA}
  Now we are ready to state semiparametric version of Theorem~\ref{theorem: bvmTarget}.
  \begin{theorem}
  \label{theorem: bvmSemi}
    Suppose the conditions of Section~\ref{sec: conditions} and condition \((B)\). 
    Consider a non-informative prior on \( (\thetav,\etavm) \).
    Then there exists a random event \( \Omega(\xx) \) of a dominating probability at least \(1 - 4 \ex^{-\xx}\) such that it holds on \( \Omega(\xx) \)
    \begin{EQA}
    \label{meanPostBvmSemi}
      \| \DPrc (\vthetavbm - \thetavdm) \|^{2}
      & \le &
      (1 + \errSieveFisher) \spread^{*}(\rups, \xx) + \errSieveParam.
      \\
      \bigl\| \Id_{\dimp} - \DPrc \Covpostm^{2} \DPrc \bigr\|
      & \le &
      \errSieveFisher + (1 + \errSieveFisher) \spread^{*}(\rups, \xx),
    \label{covPostBvmSemi}
    \end{EQA}
    where \(\spread^{*}(\rups, \xx) = 4 \spread(\rups, \xx) + 16 \ex^{-\xx}\).
    Moreover, on \(\Omega(\xx)\) for any \( A \subset \R^{\dimp} \)
    \begin{EQA}
      && \nquad
      \exp \bigl( -2\spread(\rups, \xx) - 8\ex^{-\xx} \bigr)
      \bigl\{\P\bigl( \gammav \in A \bigr) - \tau\bigr\}
      - \ex^{- \xx}
      \\
      & \le &
      \P\bigl( \Covpostm^{-1} (\vthetav - \vthetavbm) \in A \cond \Yv \bigr)
      \\
      \qquad
      &&
      \qquad \le
      \exp \bigl( 2\spread(\rups, \xx) + 5\ex^{-\xx} \bigr)
      \bigl\{\P\bigl( \gammav \in A \bigr) + \tau\bigr\}, 
    \label{bvmResultMainSemi}
    \end{EQA}
    where 
    \begin{EQA}
	\tau 
	& \eqdef & 
	\frac{1}{2} \sqrt{\dimp \, \spread^{2}(\rups, \xx) 
	+ \bigl\{ 1 + \spread(\rups, \xx) \bigr\}^{2} \spread^{2}(\rups, \xx)}.
\label{tausedefs}
\end{EQA}
  \end{theorem}
  
\section{The i.i.d. case and critical dimension}
\label{sec: bvmIid}
  This section comments how the previously obtained general results can be linked to the classical asymptotic results in the statistical literature. The nice feature of the whole approach based on the local bracketing is that all the results are stated under the same list of conditions: once checked one can directly apply any of the mentioned results. Typical examples include i.i.d., GLM, and median regression models. Here we briefly discuss how the BvM result can be applied to one typical case, namely, to an i.i.d. experiment.

  Let \( \Yv = (Y_{1},\ldots,Y_{n})^{\T} \) be an i.i.d. sample from a measure \( P \).
  Here we suppose the conditions of Section~5.1 in \cite{Sp2011} on \( P \) and
  \( (P_{\upsilonv}) \) to be fulfilled. 
  We admit that the parametric assumption \( P \in (P_{\upsilonv}, \upsilonv \in \Upsilon) \) can be misspecified and consider the asymptotic setup with the full dimension \( \dimtotal = \dimn \) which depends on 
  \( \nsize \) and grows to infinity as \( \dimn \to \infty \). 
  \begin{theorem}
  \label{theorem: bvmIid}
    Suppose the conditions of Section~5.1 in \cite{Sp2011}.
    Let also \( \dimn \to \infty \) and \( \dimn^{3}/\nsize \to 0 \). 
    Then the result of Theorem~\ref{theorem: bvmTarget} holds with
    \( \spread(\rups, \xx) = \CONST \sqrt{\dimn^3/\nsize} \),
    \( \DFc^{2} = \nsize \IF_{\upsilonvs} \), where \( \IF_{\upsilonvs} \) is the Fisher information of
    \( (P_{\upsilonv}) \) at \( \upsilonvs \).
  \end{theorem}

  A similar result about asymptotic normality of the posterior in a linear regression model can be found in \cite{Gh1999}. However, the convergence is proved under the condition \(\dimn^{4} \log (\dimn) / \nsize \to 0\) which appears to be too strong. \cite{Gh2000} showed that the dimensionality constraint can be relaxed to \( \dimn^{3}/\nsize \to 0 \) for exponential models with a product structure. \cite{BoGa2009} proved the BvM result in a specific class of i.i.d. model with discrete probability distribution under the condition \( \dimn^{3}/\nsize \to 0 \). Further examples and the related conditions for Gaussian models are presented in \cite{Jo2012}.

\subsection{Critical dimension}
\label{sec: criticalDimension}
  This section discusses the issue of a \emph{critical dimension}. 
  Namely we show that the condition \( \dimn = o(\nsize^{1/3}) \) in Theorem~\ref{theorem: bvmIid}
  for the validity of the BvM result cannot be dropped or relaxed in a general situation.
  Namely, we present an example for which \( \dimn^{3} / \nsize \ge \beta^{2} > 0 \) and the posterior distribution does not concentrate around MLE.

  Let \( \nsize \) and \( \dimn \) be such that \( \Mn = \nsize / \dimn \) is an integer. We consider a simple Poissonian model with \( Y_{i} \sim \Poisson(\upsilon_{j}) \) for \( i \in \Ij_{j} \), where \( \Ij_{j} \eqdef \{ i: \lceil i/\Mn \rceil = j \} \) for \( j=1,\ldots,\dimn \) and \( \lceil x \rceil\) is the nearest integer greater or equal to \(x\).
  Let also \( u_{j} = \log \upsilon_{j} \) be the canonical parameter. The log-likelihood \( \LL(\uv) \) with \( \uv = (u_{1},\ldots,u_{\dimn}) \) reads as
  \begin{EQA}[c]
    \LL(\uv)
    =
    \sum_{j=1}^{\dimn} \bigl( Z_{j} u_{j} - \Mn e^{u_{j}} \bigr) ,
  \label{likelihoodPois}
  \end{EQA}
  where
  \begin{EQA}[c]
    Z_{j}
    \eqdef
    \sum_{i \in \Ij_{j}} Y_{i} \, .
  \label{partialObservSumPois} 
  \end{EQA}
  We consider the problem of estimating the mean of the \( u_{j} \)'s:
  \begin{EQA}[c]
    \theta
    =
    \frac{1}{\dimn} \bigl( u_{1} + \ldots + u_{\dimn} \bigr).
  \label{targetParamPois}
  \end{EQA}
  Below we study this problem in the asymptotic setup with
  \( \dimn \to \infty \) as \( \nsize \to \infty \)
  when the underlying measure \( \P \) corresponds to
  \( \us_{1} = \ldots = \us_{\dimn} = \us \) for some \( \us \) yielding \( \thetas = \us \).
  The value \( \us \) will be specified later.
  We consider an i.i.d. exponential prior on the parameters \( \upsilon_{j} \) of Poisson distribution:
  \begin{EQA}[c]
    \upsilon_{j}
    \sim
    \operatorname{Exp}(\mu).
  \end{EQA}
  Below we allow that \(\mu\) may depend on \( \nsize \).
  Our results are valid for \(\mu \le \CONST \sqrt{\frac{\nsize}{\log \nsize}} \).
  The posterior is Gamma distributed:
  \begin{EQA}[c]
    \upsilon_{j} \cond \Yv \sim \operatorname{Gamma}(\alpha_{j}, \mu_{j}),
  \end{EQA} 
  where \(\alpha_{j} = 1 + \sum_{i \in \Ij_{j}} Y_{i}\), \(\mu_{j} = \frac{\mu}{\Mn \mu + 1}\).

  First we describe the profile maximum likelihood estimator \(\tilde{\theta}_{\nsize}\) of the target parameter \( \theta \).
  The MLE for the full parameter \( \upsilonv \) reads as
  \( \tilde{\upsilonv} = (\tilde{\upsilon}_{1}, \ldots, \tilde{\upsilon}_{\dimn})^{\T} \) with
  \begin{EQA}[c]
    \tilde{\upsilon}_{j}
    =
    Z_{j} / \Mn.
  \end{EQA}
  Thus, the profile MLE \( \tilde{\theta}_{\nsize} \) reads as
  \begin{EQA}[c]
    \tilde{\theta}_{\nsize}
    =
    \frac{1}{\dimn} \sum_{j = 1}^{\dimn} \log (\tilde{\upsilon}_{j}) .
  \end{EQA}
  Furthermore, the efficient Fisher information \( \DPrc^{2} \) is equal to
  \( \dimn^{-1} \nsize \);
  see Lemma~\ref{lemma: effFisherCritExamplePois} below. 
  As \( \tilde{\theta}_{\nsize} \) is the profile MLE it is an efficient with 
  the asymptotic variance equal to \( \DPrc^{-2} \).

  \begin{theorem}
  \label{theorem: bvmPois}
    Let \( Y_{i} \sim \Poisson(\upsilons) \) for all \( i = 1,\ldots,\nsize \),
    \(\upsilons = 1 / \dimn \).
    Then
    \begin{enumerate}
      \item If \( \dimn^{3} / \nsize \to 0 \) as \( \dimn \to \infty \), then
        \begin{EQA}[c]
          \dimn^{1/2} \nsize^{1/2} \bigl(\theta - \tilde{\theta}_{\nsize}\bigr) \cond \Yv
          \tow
          \ND(0,1).
        \label{bvmPoisGood}
        \end{EQA}

      \item Let \( \dimn^{3} / \nsize \equiv \beta > 0 \). Then
        \begin{EQA}[c]
          \dimn^{1/2} \nsize^{1/2} \bigl(\theta - \tilde{\theta}_{\nsize}\bigr) \cond \Yv
          \tow
          \ND(\beta/2,1) .
        \label{bvmPoisInterm}
        \end{EQA}

      \item If \( \dimn^{3} / \nsize \to \infty \), but \( \dimn^{4} / \nsize^{3/2} \to 0 \), then
        \begin{EQA}[c]
          \dimn^{1/2} \nsize^{1/2} \bigl(\theta - \tilde{\theta}_{\nsize}\bigr) \cond \Yv
          \tow
          \infty.
        \label{bvmPoisBad}
        \end{EQA}
    \end{enumerate}
  \end{theorem}

  We carried out a series of experiments to numerically demonstrate the results of Theorem~\ref{theorem: bvmPois}. The dimension of parameter space was fixed \(\dimn = 10000\). Three cases were considered:
  \begin{enumerate}
    \item \(\dimn^{3/2} / \nsize^{1/2} = \frac{1}{\log \dimn}\), which corresponds to \(\dimn^{3} / \nsize \to 0, \nsize \to \infty\).
    \item \(\dimn^{3/2} / \nsize^{1/2} \equiv 1\).
    \item \(\dimn^{3/2} / \nsize^{1/2} = \log \dimn\), which corresponds to \(\dimn^{3} / \nsize \to \infty, \nsize \to \infty\).
  \end{enumerate}
  For each sample 10000 realizations of \(\Yv\) were generated from the exponential distribution \( \operatorname{Exp}(\upsilon_{*})\) and so were corresponding posterior values \(\theta \cond \Yv\).
  The resulting posterior distribution for three cases is demonstrated on
  Figure~\ref{figure: bvmPois}. 
  It can be easily seen that results of Theorem~\ref{theorem: bvmPois} are numerically confirmed.

  \begin{figure}
    \begin{center}
      \includegraphics[width=0.32\textwidth]{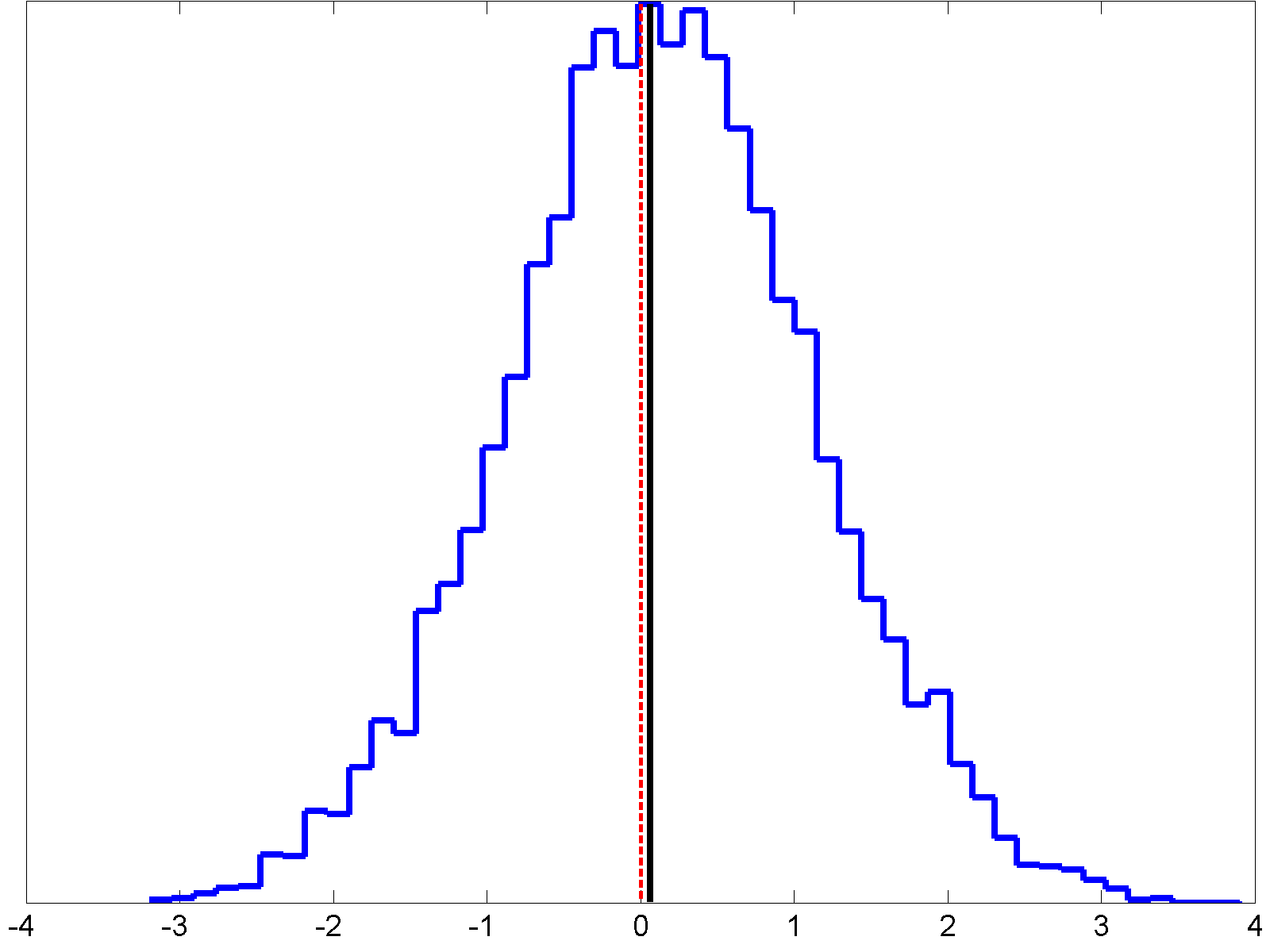}
      \includegraphics[width=0.32\textwidth]{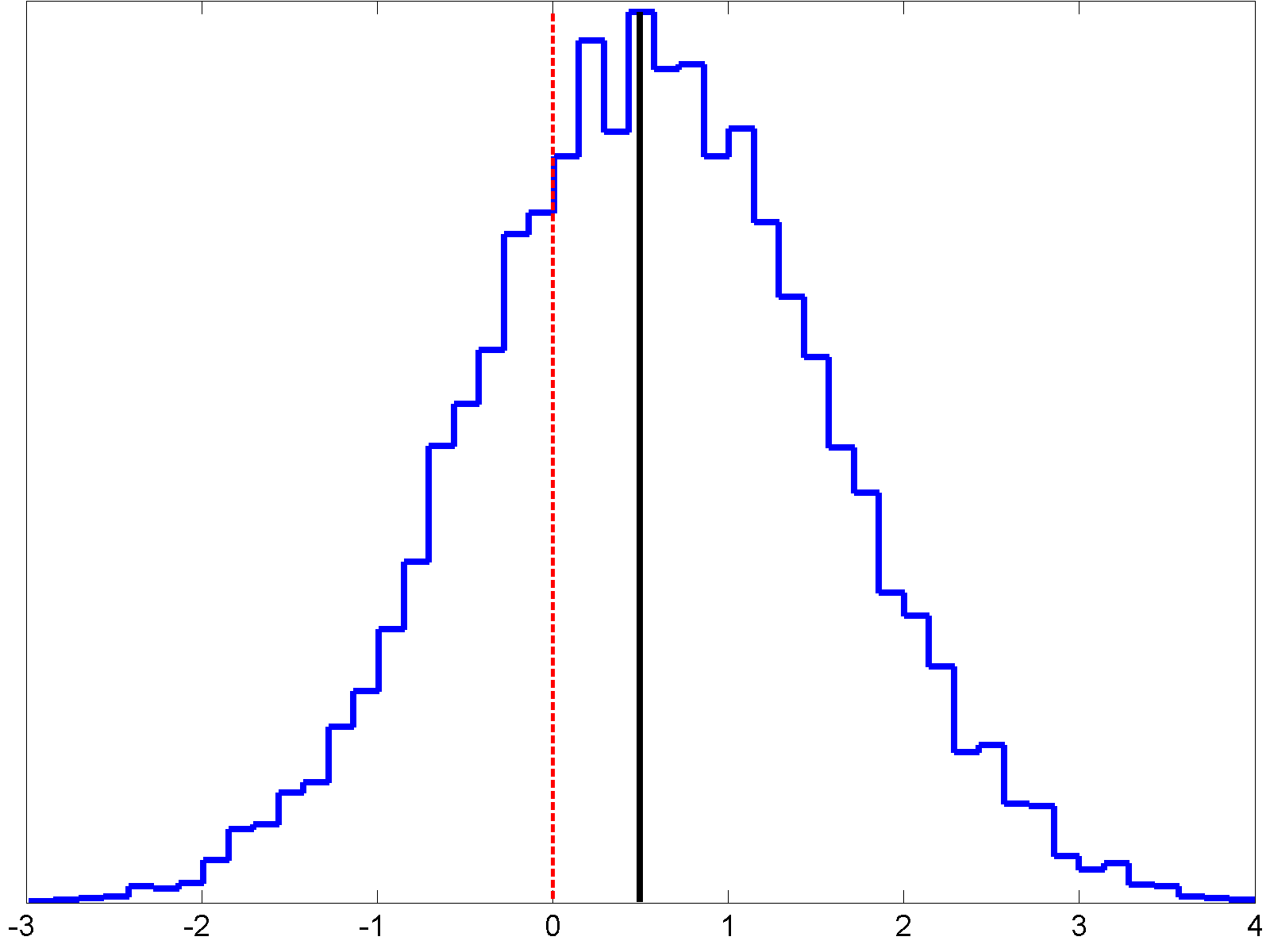}
      \includegraphics[width=0.32\textwidth]{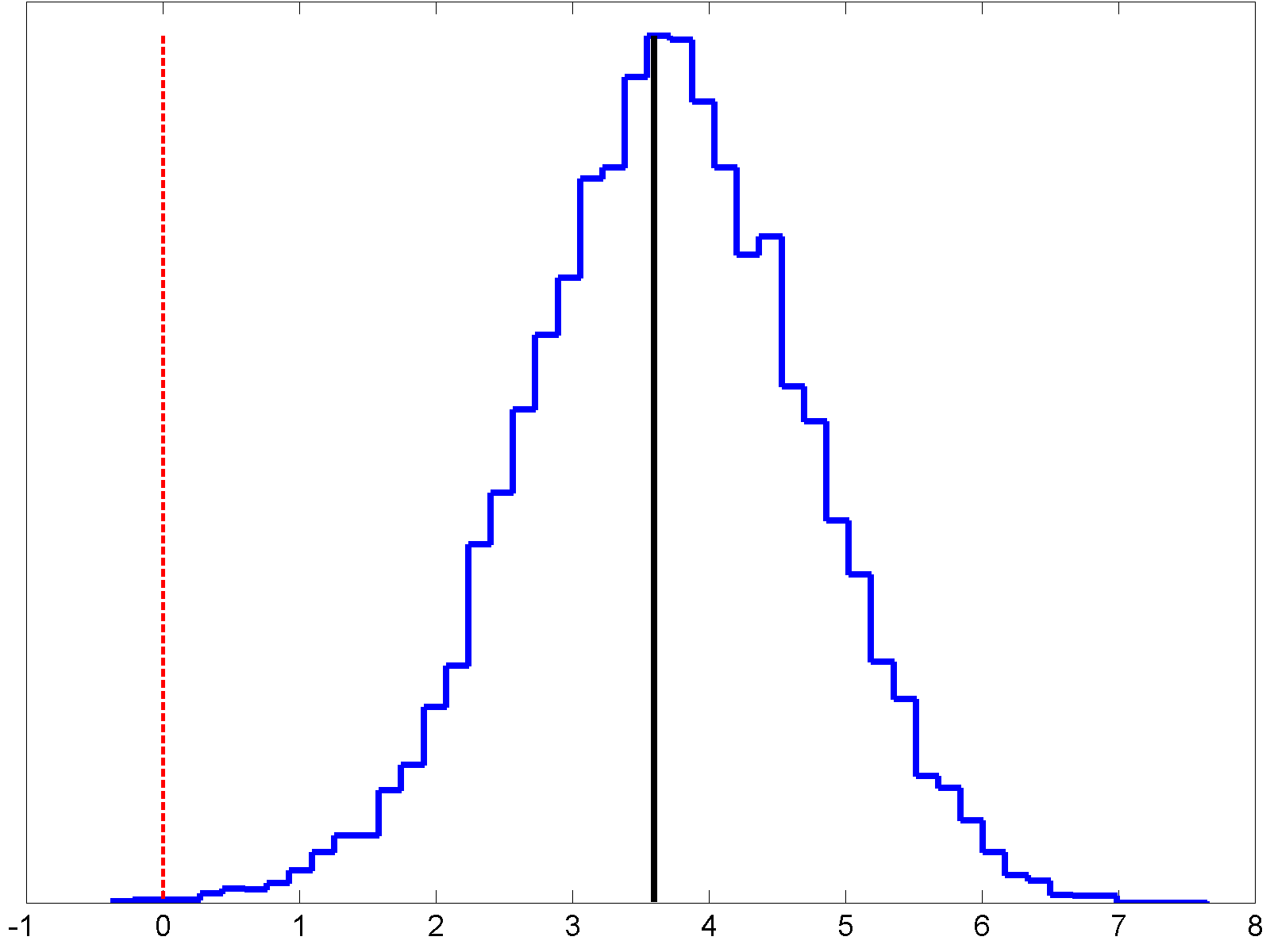}
    \end{center}
    \caption{Posterior distribution of \(\betan^{-1} \dimn \bigl(\theta - \tilde{\theta}_{\nsize}\bigr)\)
    for \(\betan = 1 / \log (\dimn)\), \(\betan = 1\), and \(\betan = \log (\dimn)\). Solid line is for posterior mean and dashed line is for true mean.}
  \label{figure: bvmPois}
  \end{figure}

\section{Examples}
\label{sec: examples}
This section presents a number of examples illustrating the general results of 
Section~\ref{sec: bvmTargetFinite}.

\subsection{Linear regression}
\label{sec: linearRegression}
  Let \(\Yv = (Y_{1}, \dots, Y_{\nsize})\) be a random vector
  \begin{EQA}
    \Yv & = & \fv + \epsv,
  \end{EQA}
  where the errors \(\epsv = (\eps_{1}, \ldots, \eps_{\nsize})\) are independent and the mean vector \(\fv \in \RR^{\nsize}\) is fixed. 
  Let \(\fv = \Psi \upsilonvs\), where \(\Psi = \{\Psi_{1}, \dots, \Psi_{\dimtotal}\}\) is a set of independent regressors and \(\upsilonvs \in \RR^{\dimtotal}\) is a vector of unknown parameters of the model. 

  First, we consider Gaussian case, i.e. \(\eps_{i} \in \mathcal{N}(0, \sigma_{\nsize}^{2} \Id_{\nsize}), i = 1, \dots, \nsize\), with \(\Id_{\nsize}\) the \(\nsize \times \nsize\) identity matrix. Variance of observations \(\sigma_{\nsize}^{2}\) is known but may depend on sample size \(\nsize\). The number of more specific settings are included in this model such as Gaussian sequence model and regression of a function in a Sobolev class. In this case approximation is valid globally and ordinary requirements for concentration of Gaussian distribution give us a condition \(\dimtotal = o(\nsize)\) BvM result to be valid, see \cite{Bo2011}. Also it is worth noting that our condition in case of a Gaussian prior (see Section~\ref{theorem: bvmGaussPrior}) coincide with one of \cite{Jo2012}.

  Now consider more general situation, when errors come from a general distribution with a density 
  \(f(\cdot)\): \( \eps_{i} \sim P_f \)
   Denote \(h(x) = \log f(x)\). The log-likelihood function for this problem reads as
  \begin{EQA}
    && \nquad 
    \LL(\upsilonv) 
    = \sum_{i = 1}^{n} \log f(Y_{i} - \Psi_{i}^{\T} \upsilonv) 
    = \sum_{i = 1}^{\nsize} h(Y_{i} - \Psi_{i}^{\T} \upsilonv).
  \end{EQA}
  Suppose that \(h(z)\) is twice continuously differentiable and 
  \( \fpp = \int h''(z) f(z) dz < \infty\). 
  If the model is correctly specified, then 
  \begin{EQA} 
    \DFc^{2} 
    &=& 
    -\nabla^{2} \E \sum_{i = 1}^{\nsize} h(Y_{i} - \Psi_{i}^{\T} \upsilonvs)
    =
    \int h''(z) f(z) dz \cdot \sum_{i = 1}^{\nsize} \Psi_{i} \Psi_{i}^{\T} 
    = 
    \fpp \sum_{i = 1}^{\nsize} \Psi_{i} \Psi_{i}^{\T} .
  \end{EQA}
  Similarly, 
  \begin{EQA} 
    \DF^{2}(\upsilonv) 
    &=& 
    -\nabla^{2} \E \sum_{i = 1}^{\nsize} h(Y_{i} - \Psi_{i}^{\T} \upsilonv)
    =
    \sum_{i = 1}^{\nsize} \Psi_{i} \Psi_{i}^{\T} \int h''(z - \Psi_{i}^{\T} (\upsilonv - \upsilonvs)) f(z) dz.
  \end{EQA}
  If distribution \(P_f\) is subexponential then \((ED_0)\) is clearly satisfied. 
  Then \((\LL_0)\) and \((\LL \rr)\) are also valid by arguments similar to the GLM case 
  under the Lipschitz continuity assumption on \(h''(\cdot)\).
  \begin{lemma}
  \label{lemma: expectApproxLinear}
    Let
    \begin{EQA}[c]
      |h''(z) - h''(z_{0})| \le L |z - z_{0}|, ~ z, z_0 \in \RR .
    \end{EQA}
    Then
    \begin{EQA}[c]
      \rddelta(\rr) \le \frac{L \, \rr}{\fpp \, N_{1}^{1/2}},
    \end{EQA}
    where
    \begin{EQA}
      N_{1}^{-1/2}  
      & \eqdef &
      \max_{i} \sup_{\gammav \in \R^{p}} 
      \frac{|\Psi_{i}^{\T} \gammav|}{\| \DFc \gammav \|} \, .
    \end{EQA}    
  \end{lemma}
  The next interesting question is checking the condition \((ED_{2})\). The stochastic part of the likelihood reads as follows:
  \begin{EQA}
    && \nquad
    \zeta(\upsilonv) = \sum_{i = 1}^{\nsize} h(Y_{i} - \Psi_{i}^{\T} \upsilonv) - \E h(Y_{i} - \Psi_{i}^{\T} \upsilonv).
  \end{EQA}
  To check the condition \((ED_{2})\) we need to compute the Hessian of the \(\zeta(\upsilonv)\):
  \begin{EQA}
    \nabla^{2} \zeta(\upsilonv) 
    &=& 
    \sum_{i = 1}^{\nsize} 
    \bigl\{ h''(Y_{i} - \Psi_{i}^{\T} \upsilonv) - \E h''(Y_{i} - \Psi_{i}^{\T} \upsilonv) \bigr\} 
    \Psi_{i} \Psi_{i}^{\T} .
  \end{EQA}
  Finally, we need to impose the mild condition on marginal likelihood:
  \begin{description}
    \item[\( \bb{(e_{2})} \)] 
      \emph{There exist some constant \( \nunu \) and for every \(\rr > 0\) exists \( \gm(\rr) > 0 \), 
            such that for all numbers \( \delta \) with
            \( | \delta | \le N_{2}^{-1/2} \rr\)
           } 
    \begin{EQA}[c]
      \log \E \exp\biggl(
      	\frac{\mubc}{\expzeta_{i}} \bigl\{ h''(Y_{i} + \delta) - \E h''(Y_{i} + \delta) \bigr\} 
      \biggr)
      \le 
      \frac{\nunu^{2} \mubc^{2}}{2},
      \qquad 
      |\mubc| \le \gm (\rr),
    \label{ed2OneDimRegr}
    \end{EQA}
    where \(\expzeta_{i}\) are some known values and
    \begin{EQA}
    N_{2}^{-1/2}  
    & \eqdef &
    \max_{i} \sup_{\gammav \in \R^{\dimtotal}} 
    \frac{\expzeta_{i} |\Psi_{i}^{\T} \gammav|}{\|\DFc \gammav \|} \, .
  \end{EQA} 
  \end{description}

  Then, we can check \((ED_{2})\):
  \begin{EQA}
    && \nquad 
    \sup_{\gammav_{1}, \gammav_{2} \in \R^{\dimtotal}} \log \E \exp\left\{
        \frac{\mubc}{\rhor} \frac{\gammav_{1}^{\T} \nabla^{2} \zeta(\upsilonv) \gammav_{2}}{\|\DFc \gammav_{1}\| \cdot \|\DFc \gammav_{2}\|}
    \right\}
    \\
    & = &
    \sup_{\gammav_{1}, \gammav_{2} \in \R^{\dimtotal}} 
    \log \E \exp\left\{
        \frac{\mubc}{\rhor} 
        \frac{ \sum_{i = 1}^{\nsize} (h''(Y_{i} - \Psi_{i}^{\T} \upsilonv) 
        		- \E h''(Y_{i} - \Psi_{i}^{\T} \upsilonv)) \gammav_{1}^{\T} \Psi_{i} \Psi_{i}^{\T} \gammav_{2}}
        	 {\|\DFc \gammav_{1}\| \cdot \|\DFc \gammav_{2}\|}
    \right\}
    \\
    & \le &
    \sup_{\gammav_{1}, \gammav_{2} \in \R^{\dimtotal}} 
    \frac{\nunu^{2} \mubc^{2}}{2 \rhor^{2}}  
    \frac{ \sum_{i = 1}^{\nsize} \expzeta_{i}^{2} 
    		|\Psi_{i}^{\T} \gammav_{1}^{\T}|^{2} \cdot |\Psi_{i}^{\T} \gammav_{2}|^{2}}
    	 {\|\DFc \gammav_{1}\|^{2} \cdot \|\DFc \gammav_{2}\|^{2}}
    \\
    & \le &
    \frac{\nunu^{2} \mubc^{2}}{2 \rhor^{2}}  \frac{\nsize}{N_{2}^{2}}.
  \end{EQA}
  It easily follows that \(\rhor = \frac{\sqrt{\nsize}}{N_{2}}\) does the job. 
  In regular situations \(N_{2}\) is of order of sample size \(\nsize\) and then 
  \(\rhor \sim \nsize^{-1/2}\).
  Finally,
  \begin{theorem}
    \label{theorem: bvmLinear}
    Let \( (e_{2}) \) and conditions of Lemma~\ref{lemma: expectApproxLinear} hold.
    Then the results of Theorem~\ref{theorem: bvmTarget} hold for linear model with 
    \begin{EQA}[c]
      \spread(\rups, \xx) = 
      \bigl\{ \rddelta(\rups) 
      + 6 \nunu \, \qqQ(\xx) \, \rhor \bigr\}\, \rups^{2}
      \le
      \biggl\{ \frac{L}{\gamma} \frac{\rups}{N_{1}^{1/2}}
      + 6 \nunu \, \qqQ(\xx) \, \frac{\sqrt{\nsize}}{N_{2}} \biggr\}\, \rups^{2}.
    \end{EQA}
  \end{theorem}

\subsection{Semiparametric linear regression with nuisance from Sobolev class}
\label{sec: semiLinRegr}
  Now suppose that 
  \begin{EQA}
    \fv &=& \Psi \thetavs + \etavs,
  \label{fvparlindef}
  \end{EQA}
  where  \(\Psi = (\Psi_{1}, \dots, \Psi_{\dimp})\),  \(\Psi_{i} = \{\psi_{1}(X_{i}), \dots, \psi_{\dimp}(X_{i})\}\) for some basis functions \(\{\psi_j(\cdot), j = 1, \dots, \dimp\}\), \(\etav^{*} = \{\etav(X_{1}), \dots, \etav(X_{\nsize})\}\) and \(\{X_{i}, ~ i = 1, \dots, N\}\) is a set of regressors. 
  Here \( \etavs \) is an element of a functional space, e.g. a Sobolev ball \( \Fs_{s} \):
\begin{EQA}
	\etavs \in \Fs_{s} 
	& \eqdef &
	\bigl\{ 
	\etav(x) = \sum_{k=1}^{\infty} \eta_{k} \phi_{k}(x) 
		\colon \sum_{k=1}^{\infty} \eta_k^{2} k^{2 s} \le C
	\bigr\},
\label{Fsdef}
\end{EQA}
for a given functional basis \(\{\phi_{k}\}_{k = 1}^{\infty}\) (e.g. Fourier, wavelet, etc.). 
Define also \(\Phi_{i} = \{\phi_{1}(X_{i}), \phi_{2}(X_{i}), \dots, \phi_{\msize}(X_{i}), \dots\}^{\T}\). Then the decomposition \eqref{fvparlindef} can be rewritten as
  \begin{EQA}
    \fv 
    &=& 
    \Xi \upsilonvs,
  \end{EQA}
  where \( \Xi_{i} = (\Psi_{i}^{\T}, \Phi_{i}^{\T})^{\T}, ~ i = 1, \dots, \nsize \) and 
  \(\Xi = (\Xi_{1}, \dots, \Xi_{\nsize})^{\T}\). 
  Then we have the following operator as infinite dimensional counterpart of the Fisher information matrix:
  \begin{EQA}[c]
    \DFc^{2} 
    = 
    \left(
      \begin{array}{cc}
        \DPc^{2} & \Ac \\
        \Ac^{\T} & \HHc^{2}
      \end{array}
    \right)
    =
    \fpp \sum_{i = 1}^{\nsize} 
    \begin{pmatrix}
      \Psi_{i} \Psi_{i}^{\T} & \Psi_{i} \Phi_{i}^{\T} \\
      \Phi_{i} \Psi_{i}^{\T} & \Phi_{i} \Phi_{i}^{\T} \\
    \end{pmatrix}
    =
    \fpp \sum_{i = 1}^{\nsize} 
    \begin{pmatrix}
      \Psi_{i} \Psi_{i}^{\T} & \Psi_{i} \Phi_{i}^{\T} \\
      \Phi_{i} \Psi_{i}^{\T} & \Phi_{i} \Phi_{i}^{\T} \\
    \end{pmatrix}.
  \end{EQA}
  We can equivalently write that \(\DFc^{2} = \{\fpp \sum_{i = 1}^{\nsize} \xi_j(X_{i}) \xi_k(X_{i})\}_{j, k = 1}^{\infty}\), where \(\xi_j\) is a properly numerated joint basis of \(\psi_j\) and \(\phi_k\).
  We introduce the following matrix: 
  \begin{EQA}[c]
    \DF^{2} = \biggl\{\fpp \int \xi_j(x) \xi_k(x) dx\biggr\}_{j, k = 1}^{\infty}.
  \end{EQA}
  Further assume that the basis \(\{\xi_{i}\}\) is orthonormal, i.e. 
  \begin{EQA}[c]
    \DF^{2} = \bigl\{\fpp \delta_{jk}\bigr\}_{j, k = 1}^{\infty}.
  \end{EQA}
  If the design is regular, then we have elementwise convergence of sums to integrals. Define
  \begin{EQA}[c]
    N_3^{1/2} = \max\Bigl(\|\DFc\|, \|\Ac\|^{2}\Bigr).
  \end{EQA}
  and also:
  \begin{EQA}[c]
    N_4^{-1/2} = \max\Bigl(\|\HHc^{-1}\|, \|\HHcm^{-1}\|, \|\DPc^{-1}\|\Bigr).
  \end{EQA}
  Now we check the bias. It holds
  \begin{EQA}[c]
    -\E L(\upsilonvsm, \upsilonvs) \leq -\E L((\upsilonvs)_{\msize}, \upsilonvs),
  \end{EQA}
  where \((\upsilonvs)_{\msize}\) is a projection of \(\upsilonvs\) on a sieve space.
  Moreover, 
  \begin{EQA}[c]
    - \E L(\upsilonv,\upsilonvs)
    =
    \| \DF(\upsilonvd) (\upsilonv - \upsilonvs) \|^{2} / 2,
  \label{excessLinQuadratic}
  \end{EQA}    
  where \( \upsilonvd \in [\upsilonvs,\upsilonv] \).
  Then,
  \begin{EQA}
    && \nquad
    \| \DFc (\upsilonvsm - \upsilonvs) \|^{2} \leq \frac{1 + \rddelta(\rups)}{1 - \rddelta(\rups)} \| \DFc ((\upsilonvs)_{\msize} - \upsilonvs) \|^{2}
    \\
    & \leq &
    \frac{1 + \rddelta(\rups)}{1 - \rddelta(\rups)} \|\DFc\|^{2} \frac{C}{\msize^{2 c}} = \frac{1 + \rddelta(\rups)}{1 - \rddelta(\rups)} \frac{C N_3}{\msize^{2 c}}.
  \end{EQA} 
  The further look gives us:
  \begin{EQA}[c]
    \|\DPrc (\thetavsm - \thetavs)\|^{2} \le \|\DPc (\thetavsm - \thetavs)\|^{2}
  \end{EQA}
  and
  \begin{EQA}
    && \nquad
    \|\DFc (\upsilonvsm - \upsilonvs)\|^{2} 
    \\
    & \ge &
    \|\DPc (\thetavsm - \thetavs)\|^{2} 
    - 2 \corrDF \|\DPc (\thetavsm - \thetavs)\| \|\HHc (\etavsm - \etavs)\|^{2} 
    + \|\HHc (\etavsm - \etavs)\|^{2}.
  \end{EQA}
  Thus, 
  \begin{EQA}
    && \nquad 
    \|\DPrc (\thetavsm - \thetavs)\|^{2} \le \|\DPc (\thetavsm - \thetavs)\|^{2}
    \\
    & \le &
    \biggl(\|\DFc (\upsilonvsm - \upsilonvs)\| + \sqrt{\corrDF}  \|\HHc (\etavsm - \etavs)\|\biggr)^{2}
    \\
    & \le &
    \biggl(\biggl[\frac{1 + \rddelta(\rups)}{1 - \rddelta(\rups)} \frac{C N_3}{\msize^{2 s}}\biggr]^{1/2} + \biggl[\corrDF \frac{C N_3}{\msize^{2 s}}\biggr]^{1/2}\biggr)^{2}
    \\
    & \le & 
    \biggl(\biggl[\frac{1 + \rddelta(\rups)}{1 - \rddelta(\rups)}\biggr]^{1/2} +\corrDF^{1/2}\biggr)^{2} \frac{C N_3}{\msize^{2 s}} .
  \end{EQA}
  
  Now we need check the second part of the condition \((B)\):
  \begin{EQA}
    \|\DPrc^{-1} (\DPrc^{2} - \DPrcm^{2}) \DPrc^{-1}\|
    &=&
    \| \DPrc^{-1}(\Ac \HHc^{-2} \Ac^{\T} - \Acm \HHcm^{-2} \Acm^{\T}) \DPrc^{-1}\|
    \\
    & \le &
    \|\DPrc^{-2}\| (\|\Ac \HHc^{-2} \Ac^{\T}\| + \|\Acm \HHcm^{-2} \Acm^{\T}\|)
    \\
    & \le &
    \|\DPrc^{-2}\| \Bigl(\|\Ac\|^{2} \|\HHc^{-2}\| + \|\Acm\|^{2} \|\HHcm^{-2}\| \Bigr) 
    \le 
    2 \frac{N_{3}}{N_{4}^{2}}.
  \end{EQA}
  Thus, we can state a bound:
  \begin{EQA}[c]
    \errSieveFisher \le 2 \frac{N_{3}}{N_{4}^{2}}.
  \end{EQA}
  In case of regular design we get \(N_3, N_4 \sim n\). Then,
  \begin{EQA}[c]
    \errSieveParam \le \CONST \frac{\nsize}{\msize^{2s}}, \qquad
    \errSieveFisher \le \CONST \frac{1}{\nsize} \, ,
  \end{EQA}
  where \(\CONST\) is some constant.
  For validity of BvM theorem with sieve prior we need \(\dimp + \msize \ll \nsize^{1/3}\). Then for \(s > 3/2\) under the proper choice of \(\msize\) the values of \(\errSieveParam\) and \(\errSieveFisher\) are asymptotically negligible.

\subsection{Generalized linear modeling}
\label{sec: glm}
  Now we consider a generalized linear modeling (GLM) which is often used for describing some categorical data. 
  Let \( \cc{P} = (P_{w}, w \in \Ups) \) be an exponential family with a canonical 
  parametrization; see e.g. \cite{mccu1989}. 
  The corresponding log-density can be represented as 
  \( \ell(y,w) = y w - d(w) \) for a convex function \( d(w) \).
  The popular examples are given by the binomial (binary response, logistic) model with 
  \( d(w) = \log\bigl( e^{w} + 1 \bigr) \),
  the Poisson model with \( d(w) = e^{w} \), the exponential model with 
  \( d(w) = - \log(w) \).
  Note that linear Gaussian regression is a special case with \( d(w) = w^{2}/2 \).

  A GLM specification means that every observation \( Y_{i} \) has a distribution from the family \( P \) with the parameter \( w_{i} \) which linearly depends on the regressor \( \Psi_{i} \in \R^{\dimp} \): 
  \begin{EQA}[c]
    Y_{i} \sim P_{\Psi_{i}^{\T} \upsilonvs} \, .
  \label{distribGlm}
  \end{EQA}    
  The corresponding log-density of a GLM reads as 
  \begin{EQA}[c]
    \LL(\thetav)
    =
    \sum \bigl\{ Y_{i} \Psi_{i}^{\T} \upsilonv - d(\Psi_{i}^{\T} \upsilonv) \bigr\} .
  \label{likelihoodGlm}
  \end{EQA}    

  Under \( \P_{\thetavs} \) each observation \( Y_{i} \) follows \eqref{distribGlm}, in particular, \( \E Y_{i} = d'(\Psi_{i}^{\T} \upsilonvs) \).
  However, similarly to the previous sections, it is accepted that the parametric model \eqref{distribGlm} is misspecified. 
  Response misspecification means that the vector \( \fv \eqdef \E \Yv \) cannot be represented in the form 
  \( d'(\Psi^{\T} \upsilonv) \) whatever \( \upsilonv \) is. 
  The other sort of misspecification concerns the data distribution. 
  The model \eqref{distribGlm} assumes that the \( Y_{i} \)'s are independent and the marginal distribution belongs to the given parametric family \( P \).
  In what follows, we only assume independent data having certain exponential moments.
  The target of estimation \( \upsilonvs \) is defined by 
  \begin{EQA}[c]
    \upsilonvs 
    \eqdef
    \argmax_{\upsilonv} \E \LL(\upsilonv).
  \label{truePointGlm}
  \end{EQA}    
  The quasi MLE \( \tilde{\upsilonv} \) is defined by maximization of \( \LL(\upsilonv) \):
  \begin{EQA}[c]
    \tilde{\upsilonv}
    =
    \argmax_{\upsilonv} \LL(\upsilonv)
    =
    \argmax_{\upsilonv} 
        \sum \bigl\{ Y_{i} \Psi_{i}^{\T} \upsilonv - d(\Psi_{i}^{\T} \upsilonv) \bigr\} .
  \label{mleGlmDescr}
  \end{EQA}
  Convexity of \( d(\cdot) \) implies that \( \LL(\upsilonv) \) is a concave function of 
  \( \upsilonv \), so that the optimization problem has a unique solution and can be effectively solved. 
  However, a closed form solution is only available for the constant regression or for the linear Gaussian regression.
  The corresponding target \( \upsilonvs \) is the maximizer of the expected log-likelihood:
  \begin{EQA}[c]
    \upsilonvs
    =
    \argmax_{\upsilonv} \E \LL(\upsilonv)
    =
    \argmax_{\upsilonv} 
        \sum \bigl\{ \fs_{i} \Psi_{i}^{\T} \upsilonv - d(\Psi_{i}^{\T} \upsilonv) \bigr\} 
  \label{truePointGlmDescr}
  \end{EQA}
  with \( \fs_{i} = \E Y_{i} \).
  The function \( \E \LL(\upsilonv) \) is concave as well and
  the vector \( \upsilonvs \) is also well defined.

  Define the individual errors (residuals) 
  \( \varepsilon_{i} = Y_{i} - \E Y_{i} \).
  Below we assume that these errors fulfill some exponential moment conditions.

  \begin{description}
    \item[\( \bb{(e_{1})} \)] 
      \emph{ There exist some constants \( \nunu \) and \( \gmiid > 0 \), 
      and for every \( i \)
      a constant \( \expzeta_{i} \) such that 
      \( \E \bigl( \varepsilon_{i}/\expzeta_{i} \bigr)^{2} \le 1 \) and
      } 
      \begin{EQA}[c]
        \log \E \exp\bigl(  {\mu \varepsilon_{i}}/{\expzeta_{i}} \bigr)
        \le 
        \nunu^{2} \mu^{2} / 2,
        \qquad 
        |\mu| \le \gmiid .
      \label{ed2OneDimGlm}
      \end{EQA}
  \end{description}
    
  A natural candidate for \( \expzeta_{i} \) is \( \sigma_{i} \) where 
  \( \sigma_{i}^{2} = \E \varepsilon_{i}^{2} \) is the variance of \( \varepsilon_{i} \). 
  Under \eqref{ed2OneDimGlm}, introduce a \( \dimp \times \dimp \) matrix \( \VFc \) defined by 
  \begin{EQA}[c]
    \VFc^{2} 
    \eqdef  
    \sum \expzeta_{i}^{2} \Psi_{i} \Psi_{i}^{\T} .
  \label{vMatrixGlm}
  \end{EQA}    
  Condition \( (e_{1}) \) effectively means that each error term 
  \( \varepsilon_{i} = Y_{i} - \E Y_{i} \) has some bounded exponential moments: for \( |\lambda| \le \gmiid \), it holds 
  \( f(\lambda) \eqdef \log \E \exp\bigl(  {\lambda \varepsilon_{i}}/{\expzeta_{i}} \bigr) < \infty \).
  In words, condition \( (e_{1}) \) requires light (exponentially decreasing) tail for the marginal distribution of each \( \varepsilon_{i} \).

  Define also
  \begin{EQA}
    N_{1}^{-1/2}  
    & \eqdef &
    \max_{i} \sup_{\gammav \in \R^{p}} 
    \frac{\expzeta_{i} |\Psi_{i}^{\T} \gammav|}{\| \VFc \gammav \|} \, .
  \label{effSampleSizeGlm}
  \end{EQA}
  Now conditions are satisfied due to following lemma, see \cite{Sp2011} for proof.

  \begin{lemma}
  \label{lemma: ed0Ed2Glm}
    Assume \( (e_{1}) \) and let \( \VFc^{2} \) be defined by \eqref{vMatrixGlm} and \( N_{1} \) by \eqref{effSampleSizeGlm}.
    Then condition \( (E\!D_{0}) \) follows  from \( (e_{1}) \) with this \( \VFc^{2} \) 
    and \( \gmb = \gmiid N_{1}^{1/2} \).
    Moreover, the stochastic component \( \zeta(\upsilonv) \) is linear in 
    \( \upsilonv \) and the condition \( (E\!D_{2}) \) is fulfilled with 
    \( \rhor(\rr) \equiv 0 \).
  \end{lemma}

  It only remains to bound the quality of quadratic approximation for the mean of the process \( L(\upsilonv, \upsilonvs) \) in a vicinity of \( \upsilonvs \). 
  An interesting feature of the GLM is that the effect of model misspecification disappears in the expectation of \( L(\upsilonv,\upsilonvs) \).

  \begin{lemma}
  \label{lemma: excessGlm}
    It holds
    \begin{EQA}
      - \E L(\upsilonv,\upsilonvs)
      &=&
      \sum \bigl\{ 
        d(\Psi_{i}^{\T} \upsilonv) 
        - d(\Psi_{i}^{\T} \upsilonvs)
        - d'(\Psi_{i}^{\T} \upsilonvs) \Psi_{i}^{\T} (\upsilonv - \upsilonvs) 
      \bigr\} 
	=
	\kullb\bigl( \P_{\upsilonvs},\P_{\upsilonv} \bigr),
    \label{excessGLMDescr}
    \end{EQA}
    where \( \kullb\bigl( \P_{\upsilonvs},\P_{\upsilonv} \bigr) \) is the Kullback-Leibler 
    divergence between measures \( \P_{\upsilonvs} \) and \( \P_{\upsilonv} \).
    Moreover, 
    \begin{EQA}[c]
      - \E L(\upsilonv,\upsilonvs)
      =
      \| \DF(\upsilonvd) (\upsilonv - \upsilonvs) \|^{2} / 2,
    \label{excessGlmQuadratic}
    \end{EQA}    
    where \( \upsilonvd \in [\upsilonvs,\upsilonv] \) and
    \begin{EQA}[c]
      \DF^{2}(\upsilonvd) 
      = 
      \sum d''(\Psi_{i}^{\T} \upsilonvd) \Psi_{i} \Psi_{i}^{\T} .
    \label{dMatrixGlm}
    \end{EQA}    
  \end{lemma}

  The proof of this lemma can also be found in \cite{Sp2011}.
  Define now the matrix \( \DFc^{2} \) by
  \begin{EQA}[c]
    \DFc^{2}
    \eqdef
    \DF^{2}(\upsilonvs)
    =
    \sum d''(\Psi_{i}^{\T} \upsilonvs) \Psi_{i} \Psi_{i}^{\T}.
  \label{fisherMatrixGlm}
  \end{EQA}    
  Let also \( \VFc^{2} \) be defined by \eqref{vMatrixGlm}.
  Note that the matrices \( \DFc^{2} \) and \( \VFc^{2} \) coincide if the model 
  \( Y_{i} \sim P_{\Psi_{i}^{\T} \upsilonvs} \) is correctly specified and 
  \( \expzeta_{i}^{2} = d''(\Psi_{i}^{\T} \upsilonvs) \). 
  The matrix \( \DFc^{2} \) describes a local elliptic neighborhood of the central point 
  \( \upsilonvs \) in the form
  \( \Upsilons(\rr) = \{ \upsilonv: \| \DFc ( \upsilonv - \upsilonvs ) \| \le \rr \} \).
  If the matrix function \( \DF^{2}(\upsilonv) \) is continuous in this vicinity 
  \( \Upsilons(\rr) \) then the value \( \rddelta(\rr) \) measuring the approximation quality of \( - \E L(\upsilonv, \upsilonvs) \) by the quadratic function \( \| \DFc (\upsilonv - \upsilonvs) \|^{2}/2 \) is small and the identifiability 
  condition \( (\LL_{0}) \) is fulfilled on \( \Upsilons(\rr) \). 
  The following lemma gives bounds for \(\rddelta(\rr)\).
  \begin{lemma}
  \label{lemma: expectApproxGlm}
    Let \(d''(z)\) be Lipschitz continuous:
    \begin{EQA}[c]
      |d''(z) - d''(z_{0})| \le L |z - z_{0}|, \quad z, z_0 \in \RR
    \end{EQA}
    Then
    \begin{EQA}[c]
      \rddelta(\rr) \le L \frac{\rr}{N_{2}^{1/2}},
    \end{EQA}
    where
    \begin{EQA}
      N_{2}^{-1/2}  
      & \eqdef &
      \max_{i} \sup_{\gammav \in \R^{p}} 
      \frac{|\Psi_{i}^{\T} \gammav|}{d''(\Psi_{i}^{\T} \upsilonvs) \cdot \| \DFc \gammav \|} \, .
    \end{EQA}
  \end{lemma}   
  Now we are prepared to state the local results for the GLM estimation. 
  \begin{theorem}
    \label{theorem: bvmGlm}
    Let \( (e_{1}) \) and conditions of Lemma~\ref{lemma: expectApproxGlm} hold.
    Then the results of Theorem~\ref{theorem: bvmTarget} hold for GLM with 
    \begin{EQA}[c]
      \spread(\rups, \xx)
      \le
      L \frac{\rups^3}{N_{2}^{1/2}}.
    \end{EQA}
  \end{theorem}
  If the function \( d(w) \) is quadratic then the approximation error 
  \( \rddelta \) vanishes as well and then quadratic approximation is valid globally, a localization step in not required.
  However, if \( d(w) \) is not quadratic, the result applies only locally and it has to be accomplished with a large deviation bound.
  The GLM structure is helpful in the large deviation zone as well.
  Indeed, the gradient \( \nabla \zeta(\upsilonv) \) does not depend on \( \upsilonv \) and hence, the most delicate condition \( (E\rr) \) is fulfilled automatically with 
  \( \gmb = \gmiid N^{1/2} \) for all local sets \( \Upsilons(\rr) \).
  Further, the identifiability condition \( (\cc{L}\rr) \) easily follows from 
  Lemma~\ref{lemma: excessGlm}: it suffices to bound from below the matrix 
  \( \DF(\upsilonv) \) for \( \upsilonv \in \Upsilons(\rr) \):
  \begin{EQA}[c]
    \DF(\upsilonv) 
    \ge 
    \gmi(\rr) \DFc,
    \qquad 
    \upsilonv \in \Upsilons(\rr) .
  \label{dMatrixGlmLowerBound}
  \end{EQA}    

  An interesting question, similarly to the i.i.d. case, is the minimal radius 
  \( \rups \) of the local vicinity \( \Upsilons(\rups) \) ensuring the desirable concentration property.
  The required value conditions are fulfilled for \( \rr^{2} \ge \rups^{2} = \CONST (\xx + \dimtotal) \), where \( \CONST \) only depends on \( \nunu, \gmi \), and \( \gm \). Thus, the results are valid if 
  \begin{EQA}
    \delta(\rups) \rups^{2} 
    & = & 
    \CONST \frac{\rups^3}{N_{1}^{1/2}} = \CONST \frac{(\xx + \dimtotal)^{3/2}}{N_{1}^{1/2}}
  \end{EQA}
  is small.

  The GLM model also allows a semiparametric extension, i.e.
  \begin{EQA}
    && \nquad
    w_{i} = \Psi_{i}^{\T} \thetavs + \etavs(X_{i}),
  \end{EQA}
  where \(\etavs(\cdot)\) is from a Sobolev class.
  This setup differs from Section~\ref{sec: semiLinRegr} only by few technicalities and leads to similar theoretical results.

\section{Supplementary}
\label{sec: supplementary}
  This section contains the imposed conditions and some supplementary statements which are of some interest by itself.

\subsection{Bracketing and upper function devices}
\label{sec: bracketing}
  This section briefly overviews the main constructions of \cite{Sp2011} including the
  bracketing bound and the upper function results.
  The bracketing bound describes the quality of quadratic approximation of the log-likelihood process \( \LL(\upsilonv) \) in a local vicinity of the point \( \upsilonvs \), while the upper function method is used to show that the full MLE \( \tilde{\upsilonv} \) belongs to this vicinity with a dominating probability.
  Introduce the notation \(L(\upsilonv, \upsilonvs) = \LL(\upsilonv) - \LL(\upsilonvs)\) for the (quasi) log-likelihood ratio.
  Given \( \rr > 0 \), define the local set
  \begin{EQA}[c]
    \Upsilons(\rr)
    \eqdef
    \bigl\{
      \upsilonv: \,
      (\upsilonv - \upsilonvs)^{\T} \DFc^{2} (\upsilonv - \upsilonvs) \le \rr^{2}
    \bigr\} .
  \label{localSetFull}
  \end{EQA}
  Define the quadratic processes
  \( \L(\upsilonv,\upsilonvs) \):
  \begin{EQA}
  \label{quadLikApproximator}
    \L(\upsilonv,\upsilonvs)
    & \eqdef &
    (\upsilonv - \upsilonvs)^{\T} \nabla \LL(\upsilonvs)
    - \| \DFc (\upsilonv - \upsilonvs) \|^{2}/2.
  \end{EQA}

  The next result states the local bracketing bound.
  The formulation assumes that some value \( \xx \) is fixed such that \( \ex^{-\xx} \) is sufficiently small.
  If the dimension \( \dimtotal \) is large, one can select \( \xx = \CONST \log(\dimtotal) \).
  We assume that a value \(\rr = \rups \) is fixed which separates the local and global zones.
  \begin{theorem}
  \label{theorem: basicQuadApprox}
    Suppose the conditions \( (E\!D_{0}) \), \( (E\!D_{2}) \), \( (\LL_{0}) \),
    and \( (\AssId) \) from Section~\ref{sec: conditions} hold for some \(\rups > 0\). 
    Then on a random set \(\Omega_{\rups}(\xx)\) of dominating probability at least \(1 - \ex^{-\xx}\)
    \begin{EQA}[c]
      |L(\upsilonv,\upsilonvs)
      -
      \L(\upsilonv,\upsilonvs)|
      \le \spread(\rups, \xx),
      \quad
      \upsilonv \in \Upsilons(\rups),
    \label{localQuadApprox}
    \end{EQA}
    where
    \begin{EQA}
      \spread(\rups, \xx)
      & \eqdef &
      \bigl\{ \rddelta(\rups) 
      + 6 \nunu \, \qqQ(\xx) \, \rhor \bigr\}\, \rups^{2},
    \label{spreadDef}
    \\
    \qqQ(\xx) 
    & \eqdef & 
    2 \dimtotal^{1/2} + \sqrt{2\xx} + \gm^{-1}(\gm^{-2} \xx + 1) 4 \dimtotal ,
    \label{qqQdeft}
    \end{EQA}
    and \(\Upsilons(\rups)\) is defined in \eqref{localSetFull}.
    Moreover, the random vector \( \xiv = \DFc^{-1} \nabla\LL(\upsilonvs)\) fulfills on a random set \(\Omega_{\BB}(\xx)\) of dominating probability at least \(1 - 2 \ex^{-\xx}\)
    \begin{EQA}[c]
      \| \xiv \|^{2}
      \le
      \qq^{2}(\BB, \xx),
    \label{scoreUpperBound}
    \end{EQA}
    where \(\qq^{2}(\BB, \xx) \eqdef \dimB + 6 \lambdaB \xx\), 
    \begin{EQA}[c]
      \BB 
      \eqdef
      \DFc^{-1} \VFc^{2} \DFc^{-1},
      \qquad
      \dimB
      \eqdef
      \tr \bigl( \BB \bigr),
      \qquad 
      \lambdaB \eqdef \lambda_{\max}\bigl( \BB \bigr). 
    \label{effDimMatrix}
    \end{EQA}
    Furthermore, assume \( (\LL\rr) \) with \( \gmi(\rr) \equiv \gmi \) yielding
    \begin{EQA}[c]
      - \E L(\upsilonv,\upsilonvs)
      \ge
      \gmi \, \| \DFc (\upsilonv - \upsilonvs) \|^{2}
    \label{excessLowerQuad}
    \end{EQA}
    for each \( \upsilonv \in \Upsilon \setminus \Upsilons(\rups) \).
    Let also 
    \begin{EQA}
	\rr 
	& \ge &
	\frac{2}{\gmi} \, 
	\Bigl\{ \qq(\BB, \xx) + 6 \nunu \, \qqQ\bigl(\xx + \log(2\rr/\rups) \bigr) \, \rhor \Bigr\} ,
	\qquad \rr \geq \rups 
\label{rrrupsBvM}
\end{EQA}
    %
    with \( \qqQ(\xx) \) from \eqref{qqQdeft}.
    Then,
    \begin{EQA}[c]
      L(\upsilonv,\upsilonvs) \le - \gmi \, \| \DFc (\upsilonv - \upsilonvs) \|^{2}/2,
      \qquad
      \upsilonv \in \Upsilon \setminus \Upsilons(\rups) .
    \label{excessUpperFunc}
    \end{EQA}
    holds on a random set \(\Omega(\xx)\) of probability at least \(1 - 4 \ex^{-\xx}\).
  \end{theorem}
  The result \eqref{localQuadApprox} is an improved version of approximation bound obtained in \cite{Sp2011}, Theorem~3.1. The result \eqref{scoreUpperBound} can be found in the supplement to \cite{Sp2011}. The result \eqref{excessUpperFunc} is very similar to Theorem~4.2 from \cite{Sp2011}.

  \subsection{Tail posterior probability for full parameter space}
  The next step in our analysis is to check that \(\upsilonv\) concentrates in a small vicinity
  \(\Upsilons(\rups)\) of the central point \(\upsilonvs\) with a properly selected \(\rups\). The concentration properties of the posterior will be described by using the random quantity
  \begin{EQA}
    \rho^{*}(\rups)
    &=&
    \frac{\int_{\Upsilon \setminus \Upsilons(\rups)} \exp \bigl\{L(\upsilonv, \upsilonvs) \bigr\} d \upsilonv}{\int_{\Upsilons(\rups)} \exp \bigl\{L(\upsilonv, \upsilonvs) \bigr\} d \upsilonv} \, .
  \label{concentrBvmFull}
  \end{EQA}

  \begin{theorem}
  \label{theorem: postConcentrFull}
    Suppose the conditions of Theorem~\ref{theorem: basicQuadApprox}.
    Then it holds on \( \Omega_{\rups}(\xx) \)
    \begin{EQA}[c]
      \rho^{*}(\rups)
      \le
      \exp\{ 2\spread(\rups, \xx) + \nub(\rups) \} \,\,
         \gmi^{-\dimtotal / 2}
      \P\bigl( \| \gammav \|^{2} \ge \gmi \rups^{2} \bigr),
    \label{postConcentrFull}
    \end{EQA}
    with
    \begin{EQA}
      \nub(\rups)
      & \eqdef &
      - \log \P\bigl(
          \bigl\| \gammav + \xiv \bigr\| \le \rups
          \cond \Yv
      \bigr).
    \end{EQA}
    If \(\rups \ge \qq(\BB, \xx) + \qq(\dimtotal, \xx)\), then on \(\Omega(\xx)\)
    \begin{EQA}[c]
      \nub(\rups)
      \le
      2\ex^{-\xx}.
    \label{logProbLocalBound}
    \end{EQA}
  \end{theorem}
  This result yields simple sufficient conditions on the value \(\rups\) which
  ensures the concentration of the posterior on \(\Upsilons(\rups)\).
  \begin{corollary}
  \label{corollary: postConcentrFullBound}
    Assume the conditions of Theorem~\ref{theorem: postConcentrFull}. Then additional inequality \( \gmi \rups^{2} \ge \qq^{2}(\dimtotal, \xx + \frac{p}{2} \log \frac{e}{b})\)
    ensures on a random set \(\Omega(\xx)\) of probability at least \(1 - 4 \ex^{-\xx}\)
    \begin{EQA}[c]
      \rho^{*}(\rups)
      \le
      \exp\{ 2\spread(\rups, \xx) + 2 \ex^{-\xx} - \xx\}.
    \end{EQA}
  \end{corollary}
  The result follows from Theorem~\ref{theorem: postConcentrFull} with use of Lemma~\ref{lemma: standGaussQuant}.

\subsection{Tail posterior probability for target parameter}
\label{sec: postConcentrTarget}
  The next major step in our analysis is to check that \(\thetav\) concentrates in a small vicinity \(\Thetas(\rups) = \bigl\{\thetav\colon \|\DPrc (\thetav - \thetavs)\| \le \rups\bigr\}\) of the central point \(\thetavs = \Proj \upsilonvs\) with a properly selected \(\rups\).
  The concentration properties of the posterior will be described by using the random quantity
  \begin{EQA}
    \rho(\rups)
    & \eqdef &
    \frac{\int_{\Upsilon} \exp \bigl\{L(\upsilonv, \upsilonvs) \bigr\} \priorden(\upsilonv)
            \Ind\bigl\{\thetav \notin \Thetas(\rups)\bigr\} d \upsilonv}
         {\int_{\Upsilon} \exp \bigl\{L(\upsilonv,\upsilonvs) \bigr\} \priorden(\upsilonv)
            \Ind\bigl\{\thetav \in \Thetas(\rups)\bigr\} d \upsilonv}.
  \end{EQA}
  In what follows we suppose that prior is uniform, i.e.
  \(\priorden(\upsilonv) \equiv 1\), \( \upsilonv \in \Upsilon\).
  This results in the following representation for \(\rho(\rups)\):
  \begin{EQA}
    \rho(\rups)
    &=&
    \frac{\int_{\Upsilon} \exp \bigl\{L(\upsilonv, \upsilonvs) \bigr\}
            \Ind\bigl\{\thetav \notin \Thetas(\rups)\bigr\} d \upsilonv}
         {\int_{\Upsilon} \exp \bigl\{L(\upsilonv, \upsilonvs) \bigr\}
            \Ind\bigl\{\thetav \in \Thetas(\rups)\bigr\} d \upsilonv}.
  \label{concentrBvmTarget}
  \end{EQA}
  Obviously
  \(\P\bigl(\thetav \not\in \Thetas(\rups) \cond \Yv \bigr) \le \rho(\rups)\).
  Therefore, small values of \(\rho(\rups)\) indicate a small posterior probability of the large deviation set \(\{\thetav \notin \Thetas(\rups)\}\).

  \begin{theorem}
  \label{theorem: postConcentrTarget}
    Suppose \eqref{localQuadApprox}.
    Then for \( \gmi \rups^{2} \ge \qq^{2}(\dimtotal, \xx + \frac{p}{2} \log \frac{e}{b})\) on \( \Omega(\xx)\) of probability at least \(1 - 4 \ex^{-\xx}\)
    \begin{EQA}[c]
      \rho(\rups)
      \le
      \rho^{*}(\rups)
      \le
      \exp\{ 2\spread(\rups, \xx) + 2 \ex^{-\xx} - \xx\}.
    \end{EQA}
  \end{theorem}

\subsection{Local Gaussian approximation of the posterior. Upper bound}
\label{sec: gaussUpperBound}
  It is convenient to introduce local conditional expectation: for a random variable \(\eta\), define
  \begin{EQA}[c]
    \Ec \eta
    \eqdef
    \E \Bigl[ \eta \Ind\bigl\{ \thetav \in \Thetas(\rups) \bigr\} \cond \Yv
    \Bigr].
  \end{EQA}
  The following theorem gives exact statement about upper bound of this posterior expectation.
  Let
  \begin{EQA}
    \thetavd
    & \eqdef &
    \thetavs + \DPrc^{-1} \xivr.
  \end{EQA}

  \begin{theorem}
  \label{theorem: gaussExpectUpperBound}
    Suppose \eqref{localQuadApprox}.
    Then for any \(\ff\colon \RR^{\dimp} \to \RR_+\) it holds on \(\Omega_{\rups}(\xx)\)
    \begin{EQA}[c]
    \label{gaussExpectUpper}
      \Ec \ff\bigl(\DPrc (\vthetav - \thetavd)\bigr) \le
      \exp\bigl\{
          \spread^{+}(\rups, \xx) \bigr\} \, \E \ff(\gammav),
    \end{EQA}
    where \( \gammav \sim \ND(0, \Id_{\dimp}) \) and
    \begin{EQA}
      \spread^{+}(\rups, \xx)
      & \eqdef &
       2 \spread(\rups, \xx) + \nub(\rups) + \rho_{\ff}(\rups),
      \\
      \rho_{\ff}(\rups)
      &\eqdef&
      \frac{\int_{\Upsilon \setminus \Upsilons(\rups)} \exp \bigl\{L(\upsilonv, \upsilonvs)\bigr\} \,
          \ff\bigl(\DPrc (\thetav - \thetavd)\bigr) \, d\upsilonv}{\int_{\Upsilons(\rups)} \exp \bigl\{L(\upsilonv, \upsilonvs)\bigr\} \,
          \ff\bigl(\DPrc (\thetav - \thetavd)\bigr) \, d\upsilonv}.
    \end{EQA}
  \end{theorem}
  Define for random event \(\eta \in A \subseteq \RR^{\dimp}\):
  \begin{EQA}[c]
    \Pc(\eta \in A) = \Ec \Ind\{\eta \in A\}.
  \end{EQA}
  The next result considers a special case with
  \( \ff(\uv) = \bigl| \lambdav^{\T} \uv \bigr|^{2} \) and \( \ff(\uv) = \Ind(\uv \in A) \) for any measurable set \(A\).
  \begin{corollary}
  \label{corollary: gaussExpectUpperQuadInd}
    For any \( \lambdav \in \R^{\dimp} \), it holds on \( \Omega_{\rups}(\xx) \)
    \begin{EQA}
    \label{gaussExpectUpperQuad}
      \Ec \bigl| \lambdav^{\T} \DPrc (\vthetav - \thetavd) \bigr|^{2}
      & \le &
      \exp\bigl\{\spread^{+}(\rups, \xx)\bigr\} \| \lambdav \|^{2}.
    \end{EQA}
    For any measurable set \( A \subseteq \RR^{\dimp}\), it holds on \( \Omega_{\rups}(\xx) \)
    \begin{EQA}
    \label{gaussProbUpper}
      \Pc\bigl( \DPrc (\vthetav - \thetavb) \in A \bigr)
      & \leq &
      \exp \bigl\{ \spread^{+}(\rups,\xx) \bigr\} 
      \P\bigl( \gammav \in A \bigr).
    \end{EQA}
    On \( \Omega(\xx) \) one obtains
\begin{EQA}
	\spread^{+}(\rups, \xx) 
	& \le &
	2 \spread(\rups, \xx) + 2 \ex^{-\xx} + 2 \exp\bigl\{\spread(\rups, \xx) + 4\ex^{-\xx} - \xx\bigr\} .
\label{spreadprxt}
\end{EQA}
  \end{corollary}
  The next corollary describes an upper bound for the posterior probability in case of changing of scaling .
  \begin{corollary}
  \label{corollary: changeMatrixUpper}
    Let \(\DPrd\) be symmetric \(\dimp \times \dimp\) matrix such that \(\| \Id -
    \DPrd^{-1}\DPrc^{2}\DPrd^{-1}\| \le \alpha\)
    Let also \(\hat{\thetav} \in \RR^{\dimp}\) be such that 
    \(\|\DPrc(\thetavd - \hat{\thetav})\| \le \, \beta\).
    Then for any measurable set \( A \subset \R^{\dimp} \), it holds on \( \Omega(\xx) \)
    with \( \deltav_{0} \eqdef \DPrd (\thetavd -  \hat{\thetav} ) \)
    \begin{EQA}
    \label{gaussProbUpperCnahgeMatrix}
      \Pc\bigl( \DPrd (\vthetav -  \hat{\thetav}) \in A \bigr)
      & \leq &
      \exp \bigl\{ \spread^{+}(\rups, \xx) \bigr\}
      \P\bigl( \DPrd \DPrc^{-1} \gammav + \deltav_{0} \in A \bigr)
      \\
      & \leq &
      \exp \bigl\{ \spread^{+}(\rups, \xx) \bigr\}
      \Bigl( \P\bigl( \gammav \in A \bigr) + \frac{1}{2} \sqrt{\alpha^{2} \dimp + (1 + \alpha)^{2} \beta^{2}} \Bigr).
    \end{EQA}
  \end{corollary}

\subsection{Local Gaussian approximation of the posterior. Lower bound}
\label{sec: gaussLowerBound}
  Now we present a local lower bound for the posterior measure.
  \begin{theorem}
  \label{theorem: gaussExpectLowerBound}
    Suppose \eqref{localQuadApprox}.
    Then for any \(\ff\colon \RR^{\dimp} \to \RR_+\) it holds on \( \Omega_{\rups}(\xx) \)
    \begin{EQA}
      \Ec \ff\bigl(\DPrc (\vthetav - \thetavd)\bigr)
      & \ge &
      \exp\bigl\{
          -\spread^{-}(\rups, \xx) \bigr\} \, 
          \E \bigl\{\ff(\gammav) \Ind\bigl(\|\gammav + \xivr\| \le \rups\bigr)\bigr\},
    \label{gaussExpectLower}
    \end{EQA}
    where
    \begin{EQA}
      \spread^{-}(\rups, \xx)
      & \eqdef &
      2 \spread(\rups, \xx) + \nub(\rups) + \rho^{*}(\rups) + 2\tilde{\rho}_{\ff}(\rups),
      \\
      \tilde{\rho}_{\ff}(\rups)
      & \eqdef &
      \frac{\int_{\RR^{\dimtotal} \setminus \Upsilons(\rups)}
      	\exp \bigl\{\L(\upsilonv, \upsilonvs)\bigr\} \,
          \ff\bigl(\DPrc (\thetav - \thetavd)\bigr) \, d\upsilonv}
      {\int_{\Upsilons(\rups)} \exp \bigl\{\L(\upsilonv, \upsilonvs)\bigr\} \,
          \ff\bigl(\DPrc (\thetav - \thetavd)\bigr) \, d\upsilonv}.
    \end{EQA}
  \end{theorem}
  This result means that posterior measure can be bounded from below by the standard normal law up to (small) multiplicative and additive constants. 
  As a corollary, we state the result for quadratic and indicator functions \( \ff(\uv) \).
  \begin{corollary}
  \label{corollary: gaussExpectLowerQuadInd}
    For any \( \lambdav \in \R^{\dimp} \), it holds on \( \Omega_{\rups}(\xx) \)
    \begin{EQA}
        \Ec \bigl| \lambdav^{\T} \DPrc (\vthetav - \thetavd) \bigr|^{2}
        & \ge &
        \exp\bigl\{-\spread^{\ominus}(\rups, \xx) + \ex^{-\xx}\bigr\} \| \lambdav \|^{2}.
    \end{EQA}
    For any measurable set \( A \subseteq \RR^{\dimp}\), it holds on \( \Omega_{\rups}(\xx) \)
    \begin{EQA}
    \label{gaussProbLower}
      \Pc\bigl( \DPrc (\vthetav - \thetavb) \in A \bigr)
      & \ge &
      \exp \bigl\{ \spread^{-}(\rups,\xx) \bigr\} 
      \P\bigl( \gammav \in A \bigr) - \ex^{-\xx}.
    \end{EQA}
    Let \(\DPrd^{2}\) be a symmetric \(\dimp \times \dimp\) matrix such that \(\| \Id - \DPrd^{-1} \DPrc^{2} \DPrd^{-1}\| \le \alpha\) and let
    \(\hat{\thetav} \in \RR^{\dimp}\) be such that
    \(\|\DPrc(\thetavd - \hat{\thetav})\| \le \beta\). Define \( \deltav_{0} \eqdef \DPrd (\thetavd -  \hat{\thetav} ) \).
    Then for any measurable subset \( A \) in \( \R^{\dimp} \),
    it holds  on \( \Omega(\xx) \)
    \begin{EQA}
    \label{gaussProbUpperChangeMatrix}
      \Pc\bigl( \DPrd (\vthetav - \hat{\thetav}) \in A \bigr)
      & \ge &
      \exp \bigl\{ \spread^{-}(\rups, \xx) \bigr\}
      \P\bigl( \DPrd \DPrc^{-1} \gammav + \deltav_{0} \in A \bigr)
      - \ex^{-\xx}
      \\
      & \ge &
      \exp \bigl\{ \spread^{-}(\rups, \xx) \bigr\}
      \Bigl\{ \P\bigl( \gammav \in A \bigr) - \frac{1}{2} \sqrt{\alpha^{2} \dimp + (1 + \alpha)^{2} \beta^{2}}\Bigr\}
      -  \ex^{-\xx}.
    \end{EQA}
    On \( \Omega(\xx) \) one obtains
    \( \spread^{-}(\rups, \xx) \le 2 \spread(\rups, \xx) + 3 \ex^{-\xx} + 4 \exp\bigl\{\spread(\rups, \xx) + 4\ex^{-\xx} - \xx\bigr\}\).
  \end{corollary}
  The proof of this corollary is similar to Corollary~\ref{corollary: changeMatrixUpper} and Corollary~\ref{corollary: gaussExpectUpperQuadInd}.

\section{Proofs}
\label{sec: proofs}
  This appendix collects the proofs of the results.

\subsection{Some inequalities for the normal law}
\label{sec: gaussInequalities}
  This section collects some simple but useful facts about the properties of the multivariate standard normal distribution.
  Many similar results can be found in the literature, we present the proofs to keep the presentation self-contained. 
  Everywhere in this section \( \gammav \) means a standard normal vector in \( \R^{\dimp} \).
  \begin{lemma}
    For any \( \uv \in \R^{\dimp} \),
    any unit vector \( \av \in \R^{\dimp} \), and any \( \qq > 0 \), it holds
    \begin{EQA}
    \label{standGaussProbQuantBound}
      \P\bigl( \| \gammav - \uv \| \ge \qq \bigr)
      & \leq &
      \exp\bigl\{ - \qq^{2}/4 + \dimp/2 + \| \uv \|^{2}/2 \bigr\} ,
      \\
      \E \bigl\{ | \gammav^{\T} \av |^{2} \Ind\bigl( \| \gammav - \uv \| \ge \qq \bigr) \bigr\}
      & \leq &
      (2 + |\uv^{\T} \av|^{2}) \exp\bigl\{ - \qq^{2}/4 + \dimp/2 + \| \uv \|^{2}/2 \bigr\} .
    \label{standGaussExpectQuadQuantBound}
    \end{EQA}
  \end{lemma}
  \begin{proof}
    By the exponential Chebyshev inequality, for any \( \lambda < 1 \) 
    \begin{EQA}
      \P\bigl( \| \gammav - \uv \| \ge \qq \bigr)
      & \leq &
      \exp \bigl( - \lambda \qq^{2}/2 \bigr) \E \exp\bigl( \lambda \| \gammav - \uv \|^{2}/2 \bigr)
      \\
      &=&
      \exp\Bigl\{ 
        - \frac{\lambda \qq^{2}}{2} - \frac{\dimp}{2} \log(1 - \lambda) 
        + \frac{\lambda}{2 (1-\lambda)} \| \uv \|^{2}
      \Bigr\}.
    \end{EQA}
    In particular, with \( \lambda = 1/2 \), this implies \eqref{standGaussProbQuantBound}.
    Further, for \( \| \av \| = 1 \)
    \begin{EQA}
      \E \bigl\{ | \gammav^{\T} \av |^{2} \Ind(\| \gammav - \uv \| \ge \qq) \bigr\}
      & \leq &
      \exp \bigl( - \qq^{2}/4 \bigr) 
      \E \bigl\{ | \gammav^{\T} \av |^{2} \exp\bigl( \| \gammav - \uv \|^{2}/4 \bigr) \bigr\}
      \\
      & \leq &
      (2 + | \uv^{\T} \av |^{2}) \exp \bigl( - \qq^{2}/4 + \dimp/2 + \| \uv \|^{2}/2 \bigr) 
    \label{вЂў}
    \end{EQA}
    and \eqref{standGaussExpectQuadQuantBound} follows.
  \end{proof}

  The next result explains the concentration effect for the norm \( \| \xiv \|^{2} \) of a Gaussian vector.
  We use a version from \cite{Sp2011}.
  \begin{lemma}
  \label{lemma: standGaussQuant}
    For each \( \xx \),
    \begin{EQA}[rclcl]
      \P\bigl( \| \gammav \| \ge \qq(\dimp,\xx) \bigr)
      & \leq &
      \exp\bigl( - \xx \bigr) ,
      \qquad
      \P\bigl( \| \gammav \| \leq \qq_{1}(\dimp,\xx) \bigr)
      & \leq &
      \exp\bigl( - \xx \bigr) ,
    \end{EQA}
    where 
    \begin{EQA}
      \qq^{2}(\dimp,\xx) 
      & \eqdef &
      \dimp + \sqrt{6.6 \dimp \xx} \vee (6.6 \xx) ,
      \qquad
      \qq_{1}^{2}(\dimp,\xx) \eqdef \dimp - 2 \sqrt{\dimp \xx}.
    \end{EQA}
  \end{lemma}

  The next lemma bounds from above the Kullback-Leibler divergence between 
  two normal distributions. 
  \begin{lemma}
  \label{lemma: kullbTwoNormals}
    Let \( \P_{0} = \ND(0,\Id_{\dimp}) \) and \( \P_{1} = \ND(\betav,(\DD^{\T} \DD)^{-1}) \)
    for some non-degenerated matrix \( \DD \).
    If 
    \begin{EQA}
      \| \DD^{\T} \DD - \Id_{\dimp} \| 
      & \leq &
      \rd \leq 1/2,
    \end{EQA}
    then 
    \begin{EQA}
      2 \kullb(\P_{0},\P_{1})
      &=&
      - 2 \E_{0} \log \frac{d\P_{1}}{d\P_{0}}
      \\
      & \leq &
      \tr (\DD^{\T} \DD - \Id_{\dimp})^{2} + (1 + \rd) \| \betav \|^{2}
      \leq  
      \rd^{2} \, \dimp + (1 + \rd) \| \betav \|^{2}.
    \end{EQA}
    For any measurable set \( A \subset \R^{\dimp} \), it holds with 
    \( \gammav \sim \ND(0,\Id_{\dimp}) \)
    \begin{EQA}[c]
      \bigl| \P_{0}(A) - \P_{1}(A) \bigr|
      =
      \bigl| \P\bigl( \gammav \in A \bigr)
      - \P\bigl( \DD (\gammav - \betav) \in A \bigr) \bigr| 
      \leq 
      \sqrt{\kullb(\P_{0},\P_{1}) / 2}.
    \end{EQA}
  \end{lemma}
  \begin{proof}
    It holds 
    \begin{EQA}[c]
      2 \log \frac{d\P_{1}}{d\P_{0}}(\gammav)
      =
      \log \det (\DD^{\T} \DD) 
      - (\gammav - \betav)^{\T} \DD^{\T} \DD (\gammav - \betav)  
      + \| \gammav \|^{2}
    \end{EQA}
    with \( \gammav \) standard normal and
    \begin{EQA}
      2 \kullb(\P_{0},\P_{1})
      &=&
      - 2 \E_{0} \log \frac{d\P_{1}}{d\P_{0}}
      =
      - \log \det (\DD^{\T} \DD) 
      + \tr (\DD^{\T} \DD - \Id_{\dimp}) 
      + \betav^{\T} \DD^{\T} \DD \betav .
    \end{EQA}
    Let \( a_{j} \) be the \( j \)th eigenvalue of \( \DD^{\T} \DD - \Id_{\dimp} \).  
    \( \| \DD^{\T} \DD - \Id_{\dimp} \| \leq \rd \leq 1/2 \) yields 
    \( |a_{j}| \le 1/2 \) and 
    \begin{EQA}
      2 \kullb(\P_{0},\P_{1})
      &=&
        \betav^{\T} \DD^{\T} \DD \betav 
      +
        \sum_{j=1}^{\dimp} \bigl\{ a_{j} - \log(1 + a_{j}) \bigr\}
      \leq 
      (1 + \rd) \| \betav \|^{2} 
      + \sum_{j=1}^{\dimp} a_{j}^{2} 
      \\
      & \leq &
      (1 + \rd) \| \betav \|^{2} + \tr (\DD^{\T} \DD - \Id_{\dimp})^{2} 
      \leq 
      (1 + \rd) \| \betav \|^{2} + \rd^{2} \, \dimp.
    \end{EQA}
    This implies by Pinsker's inequality 
    \begin{EQA}
      \sup_{A} | \P_{0}(A) - \P_{1}(A) |
      & \leq &
      \sqrt{\frac{1}{2} \kullb(\P_{0},\P_{1})}
    \end{EQA}
    as required.
  \end{proof}

\subsection{Proof of Theorem~\ref{theorem: postConcentrFull}}
\label{sec: proofPostConcentrFull}
  Define \( \pnnd(\upsilonv) = \gmi \, \| \DFc (\upsilonv - \upsilonvs) \|^{2} / 2 \).
  Now, by a change of variables, one obtains
  \begin{EQA}
    && \nquad
    \frac{\gmi^{\dimtotal/2} \det(\DFc)}{(2 \pi)^{\dimtotal/2}}
    \int_{\Upsilon \setminus \Upsilons(\rups)}
        \exp \bigl\{ -\pnnd(\upsilonv)  \bigr\} d \upsilonv
    \\
    & \le &
    \frac{\gmi^{\dimtotal/2} \det(\DFc)}{(2 \pi)^{\dimtotal/2}}
    \int_{\Upsilon \setminus \Upsilons(\rups)}
        \exp \bigl\{ - \gmi \, \| \DFc (\upsilonv - \upsilonvs) \|^{2} / 2  \bigr\} d \upsilonv
    =
    \P\bigl( \| \gammav \|^{2} \ge \gmi \rups^{2} \bigr).
  \end{EQA}
  For the integral in the nominator of \eqref{concentrBvmTarget}, it holds on \( \Omega(\xx) \)
  by \eqref{excessUpperFunc}
  \begin{EQA}[c]
    \int_{\Upsilon \setminus \Upsilons(\rups)} \exp \bigl\{ L(\upsilonv,\upsilonvs) \bigr\} \, d \upsilonv
    \le
    \int_{\Upsilon \setminus \Upsilons(\rups)} \exp \bigl\{ - \pnnd(\upsilonv) \bigr\} \, d \upsilonv.
  \end{EQA}
  For integral in the denominator it holds
  \begin{EQA}
    && \nquad
    \int_{\Upsilons(\rups)} \exp \bigl\{ L(\upsilonv, \upsilonvs)\bigr\} \, d\upsilonv
    \\
    & \ge &
    \exp\{ -\spread(\rups, \xx) - \Cc(\xiv) \} \,
    \int_{\Upsilons(\rups)}
    \exp \bigl\{\L(\upsilonv, \upsilonvs) + \Cc(\xiv)\bigr\} \, d\upsilonv.
    \qquad
  \label{gaussExcessIntLowerPrelim}
  \end{EQA}
  Inequality \eqref{gaussExcessIntLowerPrelim} implies by definition of \(\nub(\rups)\):
  \begin{EQA}[l]
    \int_{\Upsilons(\rups)} \exp\{L(\upsilonv, \upsilonvs) \} \, d \upsilonv
    \ge \exp \bigl\{-\spread(\rups, \xx) - \Cc(\xiv) - \nub(\rups)\bigr\}.
  \label{gaussExcessIntLower}
  \end{EQA}
  The bound \eqref{gaussExcessIntLower}  for the local integral
  \( \int_{\Upsilons(\rups)} \exp \bigl\{ L(\upsilonv,\upsilonvs) \bigr\} d \upsilonv \) implies that
  \begin{EQA}
    \rho^*(\rups)
    & \le &
    \exp\bigl\{ \spread(\rups, \xx) + \nub(\rups) + \Cc(\xiv) \bigr\}
        \int_{\Upsilon \setminus \Upsilons(\rups)} \exp \bigl\{ - \pnnd(\upsilonv) \bigr\} d \upsilonv.
  \end{EQA}
  Finally
  \begin{EQA}[c]
    \exp\bigl\{ \Cc(\xiv) \bigr\}
    =
    \exp\bigl\{- \| \xiv \|^{2}/2 \bigr\} \, (2\pi)^{-\dimtotal/2} \, \det(\DFc)
    \le
    (2\pi)^{-\dimtotal/2} \, \det(\DFc)
  \label{exCcmxivm}
  \end{EQA}
  and the assertion \eqref{postConcentrFull} follows.
  The bound \eqref{logProbLocalBound} is also straightforward
  \begin{EQA}
    \nub(\rups)
    & = &
    - \log \P\bigl(
        \| \gammav + \xiv \| \le \rups
        \cond \Yv
    \bigr)
    \le 
    - \log \P\bigl(
      \|\gammav\| + \|\xiv\| \le \rups
      \cond \Yv
    \bigr)
    \\
    & \le &
    - \log \P\bigl(
      \|\gammav\| \le \qq(\dimtotal, \xx)
      \cond \Yv
    \bigr)
    \le 
    2 \ex^{-\xx}.
  \end{EQA}

\subsection{Proof of Theorem~\ref{theorem: postConcentrTarget}}
\label{sec: proofPostConcentrTarget}
  Obviously \( \bigl\{ \thetav \notin \Thetas(\rups), \, \upsilonv \in \Upsilon\bigr\}
  	\subset 
  	\bigl\{\Upsilon \setminus \Upsilons(\rups)\bigr\} 
  \).
  Therefore, it holds for the integral in the nominator of \eqref{concentrBvmTarget} in a view of \eqref{excessUpperFunc}
  \begin{EQA}[c]
    \int_{\Upsilon} \exp \bigl\{L(\upsilonv, \upsilonvs) \bigr\}  \Ind\bigl\{\thetav \notin
    \Thetas(\rups)\bigr\} \, d \upsilonv
    \le
    \int_{\Upsilon \setminus \Upsilons(\rups)} \exp \bigl\{L(\upsilonv, \upsilonvs) \bigr\} \, d \upsilonv.
  \end{EQA} 
  For the denominator, the inclusion \( \Upsilons(\rups) \subset \bigl\{\thetav \in \Thetas(\rups), \upsilonv \in \Upsilon\bigr\}\) and \eqref{excessUpperFunc} imply
  \begin{EQA}[c]
    \int_{\Upsilon} \exp \bigl\{L(\upsilonv, \upsilonvs) \bigr\}  \Ind\bigl\{\thetav \in
    \Thetas(\rups)\bigr\} \, d \upsilonv
    \ge
    \int_{\Upsilons(\rups)} \exp \bigl\{L(\upsilonv, \upsilonvs) \bigr\} \, d \upsilonv.
  \end{EQA}
  Finally
  \begin{EQA}[c]
    \rho(\rups)
    =
    \frac{\int_{\Upsilon} \exp \bigl\{L(\upsilonv, \upsilonvs) \bigr\}
    	  \Ind\bigl\{\thetav \notin \Thetas(\rups)\bigr\} d \upsilonv}
    	 {\int_{\Upsilon} \exp \bigl\{L(\upsilonv, \upsilonvs) \bigr\}
	 		\Ind\bigl\{\thetav \in \Thetas(\rups)\bigr\} d \upsilonv}
    \le
    \frac{\int_{\Upsilon \setminus \Upsilons(\rups)}
    		\exp \bigl\{L(\upsilonv, \upsilonvs) \bigr\} d \upsilonv}
    	 {\int_{\Upsilons(\rups)} \exp \bigl\{L(\upsilonv, \upsilonvs) \bigr\} d \upsilonv}
    =
    \rho^{*}(\rups),
  \end{EQA}
  and the assertion follows from Theorem~\ref{theorem: postConcentrFull}.

\subsection{Proof of Theorem~\ref{theorem: gaussExpectUpperBound}}
\label{sec: proofGaussExpectUpperBound}
  We use that
  \(\L(\upsilonv, \upsilonvs) = \xiv^{\T} \DFc (\upsilonv - \upsilonvs) - \|\DFc (\upsilonv - \upsilonvs)\|^{2} / 2\)
  is proportional to the density of a Gaussian distribution.
  More precisely, define
  \begin{EQA}
    \Cc(\xiv)
    & \eqdef &
    - \|\xiv\|^{2}/2 + \log (\det \DFc) - \dimtotal \log(\sqrt{2\pi}) .
  \end{EQA}
  Then
  \begin{EQA}
    \Cc(\xiv) + \L(\upsilonv, \upsilonvs)
    &=&
    - \|\DFc (\upsilonv - \upsilonvs) - \xiv\|^{2}/2
    + \log(\det \DFc) - \dimtotal \log(\sqrt{2\pi})
    \qquad
  \end{EQA}
  is (conditionally on \(\Yv\)) the log-density of the normal law
  with the mean \(\upsilonv_{0} = \upsilonvs + \DFc^{-1} \xiv\) and the covariance matrix \(\DFc^{-2}\).
  If we perform integration and leave only  \(\thetav\) part of \(\upsilonv\) then \(\Cc(\xiv) + \L(\upsilonv, \upsilonvs) \)  is (conditionally on \(\Yv\)) the log-density of the normal law with the mean \(\thetavd = \DPrc^{-1} \xivr + \thetavs\) and the covariance matrix \(\DPrc^{-2}\). So, for any nonnegative function \(\ff: \RR^{\dimp} \to \RR_{+}\) we get
  \begin{EQA}
    && \nquad
    \int_{\Upsilon} \exp \bigl\{L(\upsilonv, \upsilonvs) + \Cc(\xiv) \bigr\} \,
        \ff\bigl(\DPrc (\thetav - \thetavd)\bigr) \, d\upsilonv
    \\
    & = &
    \int_{\Upsilons(\rups)} \exp \bigl\{L(\upsilonv, \upsilonvs) + \Cc(\xiv) \bigr\} \,
        \ff\bigl(\DPrc (\thetav - \thetavd\bigr) \, d\upsilonv
    \\
    &&
    + \, \int_{\Upsilon \setminus \Upsilons(\rups)} \exp \bigl\{L(\upsilonv, \upsilonvs) + \Cc(\xiv) \bigr\} \,
        \ff\bigl(\DPrc (\thetav - \thetavd)\bigr) \, d\upsilonv
    \\
    & = &
    \bigl(1 + \rho_{\ff}(\rups)\bigr)\int_{\Upsilons(\rups)}
    \exp \bigl\{L(\upsilonv, \upsilonvs) + \Cc(\xiv) \bigr\} \,
        \ff\bigl(\DPrc (\thetav - \thetavd)\bigr) \, d\upsilonv
    \\
    & \le &
    \ex^{\spread(\rups, \xx) + \rho_{\ff}(\rups)}
    \int_{\Upsilons(\rups)} \exp \bigl\{\L(\upsilonv, \upsilonvs) + \Cc(\xiv) \bigr\} \,
        \ff\bigl(\DPrc (\thetav - \thetavd)\bigr) \, d\upsilonv
    \\
    &\le&
    \ex^{\spread(\rups, \xx) + \rho_{\ff}(\rups)}
    \int_{\RR^{\dimtotal}} \exp \bigl\{\L(\upsilonv, \upsilonvs) + \Cc(\xiv) \bigr\} \,
        \ff\bigl(\DPrc (\thetav - \thetavd)\bigr) \, d\upsilonv
    \\
    &=&
    \ex^{\spread(\rups, \xx) + \rho_{\ff}(\rups)} \,\, \E \ff(\gammav).
  \end{EQA}
  Thus,
  \begin{EQA}
    \int_{\Upsilon} \exp \bigl\{L(\upsilonv, \upsilonvs) \bigr\} \,
        \ff\bigl(\DPrc (\thetav - \thetavd)\bigr) \, d\upsilonv
    & \le &
    \exp\{\spread(\rups, \xx) - \Cc(\xiv) + \rho_{\ff}(\rups)\} \, \E \ff(\gammav).
    \qquad
  \label{gaussExpectUpperNotNorm}
  \end{EQA}
  Now \eqref{gaussExpectUpperNotNorm} and \eqref{gaussExcessIntLower} imply
  \begin{EQA}
    \frac{
      \int_{\Upsilon} \exp \bigl\{L(\upsilonv, \upsilonvs) \bigr\}
      \ff\bigl(\DPrb (\thetav - \thetav_{\rd})\bigr) \, d\upsilonv
    }
    {
      \int_{\Upsilon} \exp \bigl\{L(\upsilonv, \upsilonvs) \bigr\} \, d\upsilonv
    }
    & \le &
    \exp\bigl\{2 \spread(\rups, \xx) + \nub(\rups) + \rho_{\ff}(\rups)\bigr\} \, \E \ff(\gammav)
  \end{EQA}
  and \eqref{gaussExpectUpper} follows by definition of \(\spread^{+}(\rups, \xx)\).

\subsection{Proof of Corollary~\ref{corollary: gaussExpectUpperQuadInd}}
\label{sec: proofCorollaryGaussExpectUpperQuadInd}
  As a direct implication of \eqref{gaussExpectUpper} one easily gets
  \begin{EQA}
    \Ec \bigl| \lambdav^{\T} \DPrc (\vthetav - \thetavd) \bigr|^{2}
    & \le &
    \exp(\spread^{+}(\rups, \xx)) \| \lambdav \|^{2}.
  \end{EQA}
  The only important step is to show that \(\rho_{x^{2}}(\rups)\) is small.
  Denote
  \begin{EQA}[c]
    \lambdav_{0}
    =
    \DFc^{-1}
    \left(
      \begin{array}{cc}
        \DPrc \lambdav \\
        \mathbf{0}
      \end{array}
    \right)
  \end{EQA}
  and \(\mathbf{0}\) is a zero vector of dimension \((\dimtotal - \dimp)\).
  We proceed separately with the nominator and denominator. 
 For the nominator by \eqref{standGaussProbQuantBound} and \eqref{standGaussExpectQuadQuantBound}
  on \(\Omega(\xx)\)
  \begin{EQA}
    && \nquad
    \int_{\Upsilon \setminus \Upsilons(\rups)}
    \bigl|\lambdav^{\T} \DPrc (\thetav - \thetavd)\bigr|^{2} \,
    \exp\bigl\{L(\upsilonv, \upsilonvs)\bigr\} \, \, d\upsilonv
    \\
    & \le &
    \int_{\Upsilon \setminus \Upsilons(\rups)}
    \bigl|\lambdav^{\T} \DPrc (\thetav - \thetavd)\bigr|^{2} \,
    \exp \bigl\{ - \gmi \| \DFc (\upsilonv - \upsilonvs) \|^{2}/2 \bigr\} \, d\upsilonv
    \\
    & = &
    \int_{\Upsilon \setminus \Upsilons(\rups)}
    \bigl|\lambdav_{0}^{\T} \DFc (\upsilonv - \upsilonv_{0})\bigr|^{2} \,
    \exp \bigl\{ - \gmi \| \DFc (\upsilonv - \upsilonvs) \|^{2} / 2  \bigr\} d\upsilonv
    \\
    & = &
    \exp\bigl\{(-(\dimtotal/2 + 2) \log \gmi - \log (\det \DFc) + \dimtotal \log(\sqrt{2\pi}) \bigr\}
    \E \bigl|\lambdav_{0}^{\T} (\gammav + \xiv)\bigr|^{2} \,
    \Ind(\| \gammav \|^{2} \ge \gmi \rups^{2})
    \\
    & \le &
    (4 + 2\|\xiv\|^{2}) \exp\Bigl\{-(\dimtotal/2 + 2) \log \gmi - \log (\det \DFc) + \dimtotal \log(\sqrt{2\pi}) - {\gmi \rups^{2}}/{4} +{\dimtotal}/{2}\Bigr\}
    \bigl\|\lambdav_{0}\bigr\|^{2}
    \\
    & = &
    \exp\Bigl\{2 \log 2 + \log (1 + \|\xiv\|^{2}/2) - (\dimtotal/2 + 2) \log \gmi - \log (\det \DFc) + \dimtotal \log(\sqrt{2\pi}) -{\gmi \rups^{2}}/{4} +{\dimtotal}/{2}\Bigr\}
    \bigl\|\lambdav_{0}\bigr\|^{2}
    \\
    & \le &
    \exp\Bigl\{\|\xiv\|^{2}/2 + (\dimtotal/2 + 2) \log (\ex/\gmi) - \log (\det \DFc) + \dimtotal \log(\sqrt{2\pi}) -{\gmi \rups^{2}}/{4} \Bigr\}
    \bigl\|\lambdav_{0}\bigr\|^{2}
    \\
    & \le &
    \exp\bigl\{ - \log (\det \DFc) + \dimtotal \log(\sqrt{2\pi}) - \xx\bigr\}
    \bigl\|\lambdav_{0}\bigr\|^{2}
  \end{EQA}
  for \(\gmi \rups^{2} \ge (2 \dimtotal + 4) \log (\ex/\gmi) + 2 \qq(\BB, \xx) + 4 \xx\).
  For the denominator, it holds on \(\Omega(\xx)\)
  \begin{EQA}
    && \nquad
    \int_{\Upsilons(\rups)}
    \bigl| \lambdav^{\T} \DPrc (\thetav - \thetavd)\bigr|^{2}  \,
    \exp \bigl\{L(\upsilonv, \upsilonvs)\bigr\} d\upsilonv
    \\
    & \ge &
    \ex^{-\spread(\rups, \xx)} \int_{\Upsilons(\rups)}
    \bigl|\lambdav^{\T} \DPrc (\thetav - \thetavd)\bigr|^{2} \,
    \exp \bigl\{\L(\upsilonv, \upsilonvs)\bigr\} d\upsilonv
    \\
    & = &
    \ex^{-\spread(\rups, \xx)} \int_{\Upsilons(\rups)}
    \bigl|\lambdav_{0}^{\T} \DFc (\upsilonv - \upsilonv_{0})\bigr|^{2}  \,
    \exp \bigl\{\L(\upsilonv, \upsilonvs)\bigr\} d\upsilonv
    \\
    & = &
    \ex^{-\spread(\rups, \xx) - \Cc(\xiv)} 
    \E \bigl|\lambdav_{0}^{\T} (\gammav + \xiv)\bigr|^{2} \,
    \Ind(\| \gammav + \xiv \|^{2} \leq \rups^{2})
    \\
    & = &
    \ex^{-\spread(\rups, \xx) - \Cc(\xiv)} 
    \Bigl\{
    \E \bigl|\lambdav_{0}^{\T} (\gammav + \xiv)\bigr|^{2}
    -
    \E \bigl|\lambdav_{0}^{\T} (\gammav + \xiv)\bigr|^{2} \,
    \Ind(\| \gammav + \xiv \|^{2} \ge \rups^{2})
    \Bigr\}
    \\
    & \ge &
    \ex^{-\spread(\rups, \xx) - \Cc(\xiv)} 
    \Bigl\{
    \bigl\|\lambdav_{0}\bigr\|^{2}
    +
    \bigl|\lambdav_{0}^{\T} \xiv\bigr|^{2}
    -
    2 \bigl\|\lambdav_{0}\bigr\|^{2} \exp\bigl\{ - \rups^{2}/4 + \dimtotal/2 + \| \xiv \|^{2}/2 \bigr\} 
    \Bigr\}
    \\
    & \ge &
    \ex^{-\spread(\rups, \xx) - \Cc(\xiv)} 
    \bigl\|\lambdav_{0}\bigr\|^{2}
    \{1 - 2 \ex^{-\xx}\} 
    \ge
    \exp\bigl\{-\spread(\rups, \xx) - \Cc(\xiv) - 4\ex^{-\xx}\bigr\} 
    \bigl\|\lambdav_{0}\bigr\|^{2}\
  \end{EQA}
  for \(\rups^{2} \ge 2 \dimtotal + 2 \qq_{\BB}(\xx) + 4 \xx\) on \(\Omega(\xx)\).
  This yields on \(\Omega(\xx)\)
  \begin{EQA}
    \rho_{x^{2}}(\rups)
    &=&
    \frac{
      \int_{\Upsilon \setminus \Upsilons(\rups)}
      \bigl|\lambdav^{\T} \DPrc (\thetav - \thetavd)\bigr|^{2} \,
		  \exp\bigl\{L(\upsilonv, \upsilonvs)\bigr\} \, \, d\upsilonv
    }
    {
      \int_{\Upsilons(\rups)}
		  \bigl| \lambdav^{\T} \DPrc (\thetav - \thetavd)\bigr|^{2}  \,
		  \exp \bigl\{L(\upsilonv, \upsilonvs)\bigr\} d\upsilonv
    }
    \\
    & \le &
    \frac{
      2 \exp\bigl\{-\log (\det \DFc) + \dimtotal \log(\sqrt{2\pi})\bigr\} - \xx\bigr\}
      \bigl\|\lambdav_{0}\bigr\|^{2}
    }
    {
      \exp\bigl\{-\spread(\rups, \xx) + \Cc(\xiv) - 4\ex^{-\xx}\bigr\} 
      \bigl\|\lambdav_{0}\bigr\|^{2}
    }
    \\
    & = &
    2 \exp\bigl\{\spread(\rups, \xx) - \|\xiv\|^{2}/2 \ + 4\ex^{-\xx} - \xx\bigr\} 
    \leq 
    2 \exp\bigl\{\spread(\rups, \xx) + 4\ex^{-\xx} - \xx\bigr\} .
  \end{EQA}

\subsection{Proof of Corollary~\ref{corollary: changeMatrixUpper}}
\label{sec: proofCorollaryChangeMatrixUpper}
  The first statement \eqref{gaussProbUpperChangeMatrix} follows from Theorem~\ref{theorem: gaussExpectUpperBound} with
  \( f(\uv) = \Ind\bigl( \DPrd \DPrc^{-1} \uv + \deltav_{0} \in A \bigr) \).
  Further, it holds on \( \Omega(\xx) \) for \( \deltav_{0} \eqdef \DPrd (\thetavd - \hat{\thetav}) \)
  \begin{EQA}
    \| \deltav_{0} \|^{2}
    & = &
    \|\DPrd (\thetavd - \hat{\thetav})\|^{2}
    \leq
    (1 + \alpha) \|\DPrc (\thetavd - \hat{\thetav})\|^{2}
    \leq 
    (1 + \alpha) \beta^{2}.
  \end{EQA}
  For proving \eqref{gaussProbUpperCnahgeMatrix}, we compute the Kullback--Leibler divergence between two multivariate normal distributions and apply Pinsker's inequality.
  Let \( \gammav \) be standard normal in \( \R^{\dimp} \).
  The random variable \( \DPrd \DPrc^{-1} \gammav + \deltav_{0} \)
  is normal with mean \( \deltav_{0} \) and variance
  \( \DPcb_{1}^{-1} \eqdef \DPrd \DPrc^{-2} \DPrd \).
  Obviously
  \begin{EQA}[c]
  	\|\Id_{\dimp} - \DPcb_{1}\|
  	=
  	\|\Id_{\dimp} - \DPrd^{-1} \DPrc^{2} \DPrd^{-1}\| \le \alpha.
  \end{EQA}
  Thus,  by Lemma~\ref{lemma: kullbTwoNormals} for any measurable set \( A \), it holds
  \begin{EQA}[c]
    \P\bigl( \DPrd \DPrc^{-1} \gammav + \deltav_{0} \in A \cond \Yv \bigr)
    \leq
    \P\bigl( \gammav \in A \bigr)
    + \frac{1}{2} \sqrt{\alpha^{2} \dimp + (1 + \alpha)^{2} \beta^{2}}.
  \end{EQA}

  \subsection{Proof of Theorem~\ref{theorem: gaussExpectLowerBound}}
  \label{Sprooftheorem: gaussExpectLowerBound}
  As in proof of Theorem~\ref{theorem: gaussExpectUpperBound}, for any nonnegative function \(\ff: \RR^{\dimp} \to \RR_{+}\), it holds
  \begin{EQA}
    && \nquad
    \int_{\Upsilon} \exp \bigl\{L(\upsilonv, \upsilonvs)\bigr\} \,
    \ff\bigl(\DPrc (\thetav - \thetavd) \bigr) \Ind\bigl\{\thetav \in \Thetas(\rups)\bigr\} \, d\upsilonv
    \\
    & \ge &
    \int_{\Upsilons(\rups)} \exp \bigl\{L(\upsilonv, \upsilonvs)\bigr\} \,
    \ff\bigl(\DPrc (\thetav - \thetavd) \bigr) \, d\upsilonv
    \\
    & \ge &
    \exp\{-\spread(\rups, \xx) - \Cc(\xiv)\} \,
    \int_{\Upsilons(\rups)} \exp \bigl\{\L(\upsilonv, \upsilonvs)\bigr\} \,
    \ff\bigl(\DPrc (\thetav - \thetavd) \bigr) \, d\upsilonv
    \\
    & \ge &
    \exp\{-\spread(\rups, \xx) - \Cc(\xiv)\} \,
    \int_{\RR^{\dimtotal}} \exp \bigl\{\L(\upsilonv, \upsilonvs)\bigr\} \,
    \ff\bigl(\DPrc (\thetav - \thetavd) \bigr) \, d\upsilonv
    \\
    &&
    - \, \exp\{-\spread(\rups, \xx) - \Cc(\xiv)\} \,
    \int_{\RR^{\dimtotal} \setminus \Upsilons(\rups)}
    \exp \bigl\{\L(\upsilonv, \upsilonvs)\bigr\} \,
    \ff\bigl(\DPrc (\thetav - \thetavd)\bigr) \, d\upsilonv
    \\
    & = &
    \exp\{-\spread(\rups, \xx) - \Cc(\xiv)\} (1 - \tilde{\rho}_{\ff}(\rups))\, \int_{\RR^{\dimtotal}} \exp \bigl\{\L(\upsilonv, \upsilonvs)\bigr\} \,
    \ff\bigl(\DPrc (\thetav - \thetavd) \bigr) \, d\upsilonv
    \\
    & \ge &
    \exp\{-\spread(\rups, \xx) - \Cc(\xiv)\} (1 - \tilde{\rho}_{\ff}(\rups))\,
    \int_{\Thetas(\rups) \times \RR^{(\dimtotal - \dimp)}}
    	\exp \bigl\{\L(\upsilonv, \upsilonvs)\bigr\} \,
    	\ff\bigl(\DPrc (\thetav - \thetavd) \bigr) \, d\upsilonv
    \\
    & \ge &
    \exp\{-\spread(\rups, \xx) - \Cc(\xiv) - 2 \tilde{\rho}_{\ff}(\rups)\} \,
    	\E \ff(\gammav) \Ind\{\|\gammav + \xivr\| \le \rups\},
    \qquad
  \label{gaussExpectLowerNotNorm}
  \end{EQA}
  Here we used that \(1 - \alpha \ge \ex^{-2 \alpha}\) for \(0 \le \alpha \le \frac{1}{2}\).
  Similarly,
  \begin{EQA}
    && \nquad
    \int_{\Upsilon} \exp \bigl\{L(\upsilonv, \upsilonvs) \bigr\} \, d \upsilonv
    =
    \int_{\Upsilons(\rups)} \exp \bigl\{L(\upsilonv, \upsilonvs) \bigr\} \, d \upsilonv
    	+ \int_{\Upsilon \setminus \Upsilons(\rups)}
    	\exp \bigl\{L(\upsilonv, \upsilonvs)\bigr\} \, d\upsilonv
    \\
    & = &
    \{1 + \rho^{*}(\rups)\} \int_{\Upsilons(\rups)}
    	\exp \bigl\{L(\upsilonv, \upsilonvs) \bigr\} \, d \upsilonv
    \\
    & \le &
    \{1 + \rho^{*}(\rups)\} \exp\{\spread(\rups, \xx) - \Cc(\xiv)\} \,
    \P\bigl(\bigl\| \gammav + \xiv \bigr\| \le \rups \cond \Yv \bigr),
    \qquad
  \end{EQA}
  and finally
  \begin{EQA}[c]
    \int_{\Upsilon} \exp \bigl\{L(\upsilonv, \upsilonvs) \bigr\} \, d \upsilonv
    \le
    \exp\{\spread(\rups, \xx) - \Cc(\xiv) + \nub(\rups) + \rho^{*}(\rups)\}.
    \qquad
  \label{gaussExcessIntUpper}
  \end{EQA}
  The bounds
  \eqref{gaussExpectLowerNotNorm} and \eqref{gaussExcessIntUpper} imply
  \begin{EQA}
    && \nquad
    \frac{\int_{\Upsilon} \exp \bigl\{L(\upsilonv, \upsilonvs) \bigr\}
            \ff\bigl(\DPrc (\thetav - \thetavd) \bigr) \, d\upsilonv}
         {\int_{\Upsilon} \exp \bigl\{L(\upsilonv, \upsilonvs) \bigr\} \, d\upsilonv}
    \\
    & \ge &
    \frac{\exp\{-\spread(\rups, \xx) - \Cc(\xiv) - 2 \tilde{\rho}_{\ff}(\rups)\} \, \E \ff(\gammav) \Ind\{\|\gammav + \xivr\| \le \rups\}}
         {\exp\bigl\{\spread(\rups, \xx) - \Cc(\xiv) + \nub(\rups) + \rho^{*}(\rups) \bigr\}}
    \\
    & \ge &
    \exp\{-2 \spread(\rups, \xx) - 2 \tilde{\rho}_{\ff}(\rups) - \nub(\rups) - \rho^{*}(\rups)\}
    		\E \ff(\gammav) \Ind\{\|\gammav + \xivr\| \le \rups\}.
  \end{EQA}
  This yields \eqref{gaussExpectLower}.

\subsection{Proof of Theorem~\ref{theorem: bvmTarget}}
\label{sec: proofBvmTarget}
  Due to our previous results, it is convenient to decompose the r.v. \( \vthetav \)
  in the form
  \begin{EQA}[c]
    \vthetav
    =
    \vthetav \Ind\bigl\{ \vthetav \in \Thetas(\rups) \bigr\} +
    \vthetav \Ind\bigl\{ \vthetav \not\in \Thetas(\rups) \bigr\}
    =
    \vthetav^{\circ} + \vthetav^{c}.
  \label{vthetavTTc}
  \end{EQA}
  The large deviation results yields that the posterior distribution of the part
  \( \vthetav^{c} \) is negligible provided a proper choice of \( \rups \).
  Below we show that \( \vthetav^{\circ} \) is nearly normal which yields the BvM result.
  Define
  \begin{EQA}[c]
    \vthetavd
    \eqdef
    \Ec \vthetav ,
    \qquad
    \Covpostd^{2}
    \eqdef
    \Cov (\vthetav^{\circ})
    \eqdef
    \Ec \bigl\{ (\vthetav - \vthetavd) (\vthetav - \vthetavd)^{\T} \bigr\}.
  \end{EQA}
  It suffices to show that holds on \( \Omega(\xx) \)
  \begin{EQA}
    \| \DPrc (\vthetavd - \thetavd) \|^{2}
    & \le &
    2 \spread^{*}
    \\
    \bigl\| \Id_{\dimp} - \DPrc \Covpostd^{2} \DPrc \bigr\|
    & \le &
    2 \spread^{*},
  \end{EQA}
  where \(\spread^{*} = \max\bigl\{\spread^{\oplus}, \spread^{\ominus}\bigr\}\).

  Consider \( \etav \eqdef \DPrc (\vthetav - \thetavd) \).
  Corollaries~\ref{corollary: gaussExpectUpperQuadInd} and \ref{corollary: gaussExpectLowerQuadInd} yield for any \( \lambdav \in \R^{\dimp} \) that
  \begin{EQA}[rcccl]
  \label{expectQuadBounds}
    \| \lambdav \|^{2} \exp (-\spread^{-})
    & \le &
    \Ec \bigl| \lambdav^{\T} \etav \bigr|^{2}
    & \le &
    \| \lambdav \|^{2} \exp (\spread^{+})
  \end{EQA}
  with \( \spread^{-} = \spread^{\ominus}\) and
  \( \spread^{+} = \spread^{\oplus}\).
  Define the first two moments of \( \etav \):
  \begin{EQA}[c]
    \etavb
    \eqdef
    \Ec \etav,
    \qquad
    \Covd^{2}
    \eqdef
    \Ec \bigl\{ (\etav - \etavb) (\etav - \etavb)^{\T} \bigr\}
    =
    \DPrc \Covpostd^{2} \DPrc.
  \end{EQA}
  Use the following technical statement.
  \begin{lemma}
  \label{lemma: postBoundLocalMoments}
    Assume \eqref{expectQuadBounds}.
    Then with \( \spread^{*} = \max\bigl\{ \spread^{+}, \spread^{-} \bigr\} \le 1/2 \)
    \begin{EQA}[c]
      \| \etavb \|^{2} \le 2 \spread^{*},
      \qquad
      \| \Covd^{2} - \Id_{\dimp} \|
      \le
      2 \spread^{*} .
    \label{postBoundLocalMoments}
    \end{EQA}
  \end{lemma}

  \begin{proof}
    Let \( \uv \) be any unit vector in \( \R^{\dimp} \).
    We obtain from \eqref{expectQuadBounds}
    \begin{EQA}[c]
      \exp(- \spread^{-}) \le \Ec \bigl| \uv^{\T} \etav \bigr|^{2} \le \exp(\spread^{+}).
    \end{EQA}
    Note now that
    \begin{EQA}[c]
      \Ec \bigl| \uv^{\T} \etav \bigr|^{2}
      =
      \uv^{\T} \Covd^{2} \uv + |\uv^{\T} \etavb|^{2}.
    \end{EQA}
    Hence
    \begin{EQA}[c]
      \exp(- \spread^{-})
      \le
      \uv^{\T} \Covd^{2} \uv + |\uv^{\T} \etavb|^{2}
      \le
      \exp (\spread^{+})  .
    \label{expectQuadBoundsSeparated}
    \end{EQA}
    In a similar way with \( \uv = \etavb / \| \etavb \| \) and
    \( \gammav \sim \ND(0,\Id_{\dimp}) \)
    \begin{EQA}
      \Ec \bigl| \uv^{\T} (\etav - \etavb) \bigr|^{2}
      & \ge &
      \ex^{- \spread^{-}} \E \bigl| \uv^{\T} (\gammav - \etavb) \bigr|^{2}
      =
      \ex^{- \spread^{-}} \bigl( 1 + \| \etavb \|^{2} \bigr)
    \end{EQA}
    yielding
    \begin{EQA}[c]
      \uv^{\T} \Covd^{2} \uv
      \ge
      \bigl( 1 + \| \etavb \|^{2} \bigr) \exp (- \spread^{-}).
    \end{EQA}
    This inequality contradicts \eqref{expectQuadBoundsSeparated} if
    \( \| \etavb \|^{2} > 2 \spread^{*} > 1 \), and \eqref{postBoundLocalMoments} follows.
  \end{proof}

  The bound for the first moment implies with \( \vthetavd = \Ec \vthetav \)
  \begin{EQA}[c]
    \bigl\| \DPrc (\vthetavd - \thetavd) \bigr\|^{2}
    \le
    2 \spread^{*}
  \end{EQA}
  while the second bound yields with
  \begin{EQA}[c]
    \bigl\| \DPrc \Covpostd^{2} \DPrc - \Id_{\dimp} \bigr\|
    \le
    2 \spread^{*}.
  \end{EQA}
  The last result follows from \eqref{gaussProbUpper} and \eqref{gaussProbLower} with and additional assumption that \(\xx\) is large enough to ensure \(\spread^{+}(\rups, \xx) \le 2 \spread(\rups, \xx) + 5 \ex^{-\xx} \) and \(\spread^{-}(\rups, \xx) \ge 2 \spread(\rups, \xx) - 8 \ex^{-\xx} \).

\subsection{Proof of Theorem~\ref{theorem: bvmGaussPrior}}
  It suffices to check \eqref{postBoundContPrior}.
  First evaluate the ratio \( \prior(\upsilonv)/\prior(\upsilonvs) \)  for any \( \upsilonv \in \Upsilons(\rups) \). 
  It holds
  \begin{EQA}[c]
    \log \frac{\prior(\upsilonv)}{\prior(\upsilonvs)}
    =
    - \| \GP \upsilonv \|^{2} / 2 + \| \GP \upsilonvs \|^{2} / 2
    =
    - (\upsilonv - \upsilonvs)^{\T} \GP^{2} \upsilonvs
    - \| \GP (\upsilonv - \upsilonvs) \|^{2}/2 .
  \label{logpriorgauss}
  \end{EQA}
  It follows from the definition of \( \Upsilons(\rups) \) and  \eqref{gaussPriorFlatCond} that for 
  \( \upsilonv \in \Upsilons(\rups) \)
  \begin{EQA}[c]
    \| \GP (\upsilonv - \upsilonvs) \|^{2}
    =
    \| \GP \DPc^{-1} \DFc (\upsilonv - \upsilonvs) \|^{2}
    \le
    \| \DFc^{-1} \GP^{2} \DFc^{-1} \| \, \rups^{2} 
    \leq 
    \eps^{2} \, \rups^{2}\, .
  \label{GPttsVP1}
  \end{EQA}
  Similarly
  \begin{EQA}[c]
    \bigl| (\upsilonv - \upsilonvs)^{\T} \GP^{2} \upsilonvs \bigr|
    \le
    \| \GP \upsilonvs \| \cdot \| \GP (\upsilonv - \upsilonvs) \|
    \le
    \| \GP \upsilonvs \| \cdot \| \GP \DFc^{-1} \| \, \rups 
    \leq 
    \eps \, \rups \| \GP \upsilonvs \| .
  \label{ttvsGPtvs}
  \end{EQA}
  This obviously implies
  \begin{EQA}
    - \frac{1}{2}\eps^{2} \, \rups^{2}
    \leq 
    \log \frac{\prior(\upsilonv)}{\prior(\upsilonvs)}
    & \leq &
    \eps \, \rups \| \GP \upsilonvs \|
  \label{12epsruC}
  \end{EQA}
  and \eqref{postBoundContPrior} follows with
  \( \alpha(\rups) = \max\bigl\{ \eps \, \rups \| \GP \upsilonvs \|, \eps^{2} \rups^{2}/2 \bigr\} \).

\subsection{Proof of Theorem~\ref{theorem: bvmSemi}}
\label{sec: proofBvmSemi}
  From finite dimensional theorem we have that 
    \begin{EQA}
    \label{meanPostBvmSemiSieve}
      \| \DPrcm (\vthetavb - \thetavdm) \|^{2}
      & \le &
      \spread^{*}(\rups, \xx),
      \\
      \bigl\| \Id_{\dimp} - \DPrcm \Covpost^{2} \DPrcm \bigr\|
      & \le &
      \spread^{*}(\rups, \xx).
    \label{covPostBvmSemiSieve}
    \end{EQA}

    If \((B)\) is true, then
    \begin{EQA}  
      && \nquad
      \bigl\| \Id_{\dimp} - \DPrc \Covpost^{2} \DPrc \bigr\|
      \\
      & = &
      \sup_{\|\gamma\| = 1} |\gammav^{\T} (\Id_{\dimp} - \DPrc \Covpost^{2} \DPrc ) \gammav|
      \\
      & = &
      \sup_{\|\gamma\| = 1} |(\DPrc \gammav)^{\T} (\DPrc^{-2} -  \Covpost^{2}) (\DPrc  \gammav)|
      \\
      & \le &
      \sup_{\|\gamma\| = 1} |(\DPrc \gammav)^{\T} (\DPrc^{-2} -  \DPrcm^{-2}) (\DPrc  \gammav)|
      +
      \sup_{\|\gamma\| = 1} |(\DPrc \gammav)^{\T} (\DPrcm^{-2} -  \Covpost^{2}) (\DPrc  \gammav)|
      \\
      & = &
      \|\Id_{\dimp} - \DPrc \DPrcm^{-2}\DPrc\|
      +
      \sup_{\|\gamma\| = 1} |(\DPrcm^{-1} \DPrc \gammav)^{\T} (\Id_{\dimp} -  \DPrcm \Covpost^{2} \DPrcm) (\DPrcm^{-1} \DPrc  \gammav)|
      \\
      & \le &
      \|\Id_{\dimp} - \DPrc \DPrcm^{-2}\DPrc\|
      +
       \|\Id_{\dimp} -  \DPrcm \Covpost^{2} \DPrcm\| \sup_{\|\gamma\| = 1} | \gammav^{\T} \DPrc \DPrcm^{-2} \DPrc \gammav|
      \\
      & \le &
      \|\Id_{\dimp} - \DPrc \DPrcm^{-2}\DPrc\|
      +
       \|\Id_{\dimp} -  \DPrcm \Covpost^{2} \DPrcm\| \|\DPrc \DPrcm^{-2} \DPrc\|
      \\
      & \le &
      \errSieveFisher + (1 + \errSieveFisher) \spread^{*}(\rups, \xx).
    \end{EQA}

    For the mean it follows
    \begin{EQA}
      \|\DPrc(\vthetavb - \thetavd)\|^{2} 
      & \le &
      \|\DPrc (\vthetavb - \thetavdm)\|^{2} + \|\DPrc (\thetavdm - \thetavd)\|^{2}
      \\
      & \le &
      \|(\DPrc \DPrcm^{-1}) \DPrcm (\vthetavb - \thetavdm)\|^{2} + \|\DPrc (\thetavs - \thetavsm)\|^{2}
      \\
      & \le &
      \|\DPrcm^{-1} \DPrc^{2} \DPrcm^{-1}\| \|\DPrcm (\vthetavb - \thetavdm)\|^{2} + \|\DPrc (\thetavs - \thetavsm)\|^{2}
      \\
      & \le &
      (1 + \errSieveFisher) \|\DPrcm (\vthetavb - \thetavdm)\|^{2} + \|\DPrc (\thetavs - \thetavsm)\|^{2}
      \\
      & \le &
      (1 + \errSieveFisher) \spread^{*}(\rups, \xx) + \errSieveParam.
    \end{EQA}

    Now we can conclude that for validity of our results we need \(\errSieveParam\) and \(\errSieveFisher\) to be of order \(\spread^{*}(\rups, \xx)\).

\subsection{Proof of Theorem~\ref{theorem: bvmIid}}
\label{sec: proofBvmIid}
  The bracketing bound and the large deviation result of Theorem~\ref{theorem: basicQuadApprox} apply if the sample size
  \( \nsize \) fulfills \( \nsize \ge \CONST (\dimn + \xx)\) for a fixed constant
  \( \CONST \).
  It appears that the BvM result requires a stronger condition.
  Indeed, in the regular i.i.d. case it holds
  \begin{EQA}[c]
    \rddelta(\rups) \asymp \rups/\sqrt{\nsize},
    \qquad 
    \qqQ^{2}(\xxn) \asymp \dimn + \xxn,
    \qquad
    \rhor \asymp 1/\sqrt{\nsize}.
  \end{EQA}
  The radius \( \rups \) should fulfill \( \rups^{2} \ge \CONST (\dimn + \xx)\)
  to ensure the large deviation result.
  This yields
  \begin{EQA}[c]
    \spread(\rups, \xx)
    =
    (\rddelta(\rups) + 3 \nunu \qqQ^{2}(\xxn) \rhor) \rups^{2}
    \ge
    \CONST \sqrt{(\dimn^3 + \xx)/\nsize}.
  \end{EQA}
  If we fix \(\xx = \CONST \dimn\), our BvM result effectively requires the condition \( \dimn^{3}/\nsize \to 0 \) as
  \( \nsize \to \infty \).

\subsection{Proof of Theorem~\ref{theorem: bvmPois}}
\label{sec: proofBvmPois}
  First we check that the required conditions of Section~\ref{sec: conditions} are fulfilled in the considered example.
  This can be easily done if we slightly change the definition of the local set \( \Upsilons(\rups) \).
  Namely, for \( \uvs = (\us_{1},\ldots,\us_{\dimn})^{\T} \), define
  \( \Upsilons(\sqrt{\zz}) \) as a rectangle
  \begin{EQA}[c]
    \Upsilons(\sqrt{\zz})
    \eqdef
    \bigl\{ \uv: \,\, \Mn \kullb(u_{j},\us_{j}) \le \zz ,
    \,\, j=1,\ldots,\dimn
    \bigr\} .
  \end{EQA}
  Here \( \kullb(u,\us) \) is the Kullback-Leibler divergence for the Poisson family:
  \begin{EQA}[c]
    \kullb(u,\us)
    =
    \ex^{u} (u - \us) - \ex^{u} + \ex^{\us}.
  \end{EQA}

  \begin{lemma}
  \label{lemma: mleConcentrProductSpace}
    Let \( \zz_{\nsize} \) be such that \( 2 \dimn \ex^{-\zz_{\nsize}} \le 1/2 \).
    Then it holds
    \begin{EQA}[c]
      \P\bigl( \tilde{\uv} \in \Upsilons(\sqrt{\zz_{\nsize}}) \bigr)
      \ge
      1 - 4 \dimn \ex^{-\zz_{\nsize}}.
    \label{mleConcentrQuant}
    \end{EQA}
    In particular, the choice \( \zz_{\nsize} = \xx_{\nsize} + \log(\dimn) \)
    with \( \xx_{\nsize} = \CONST \log \nsize \) provides
    \begin{EQA}[c]
      \P\bigl( \tilde{\uv} \in \Upsilons(\sqrt{\zz_{\nsize}}) \bigr)
      \ge
      1 - 4 \ex^{-\xx_{\nsize}} .
    \label{mleConcentrDim}
    \end{EQA}
  \end{lemma}
  \begin{proof}
    We use the bound from \cite{PoSp2006}
    \begin{EQA}[c]
      \P\bigl( \Mn \kullb(\tilde{u}_{j},\us_{j}) > \zz_{\nsize} \bigr)
      \le
      2 \ex^{-\zz_{\nsize}}.
    \end{EQA}
  This yields
  \begin{EQA}[c]
    \P\bigl( \tilde{\uv} \in \Upsilons(\sqrt{\zz_{\nsize}}) \bigr)
    \ge
    \bigl( 1 - 2 \ex^{-\zz_{\nsize}} \bigr)^{\dimn} .
  \end{EQA}
  Now the elementary inequalities
  \(\log(1 - \alpha) \ge - 2 \alpha \) for
  \(0 \le \alpha \le 1/2\) and \(\ex^{-\delta} \ge 1 - \delta\) for \(\delta \ge 0\)
  applied with \( \alpha_{\nsize} = 2 \ex^{-\zz_{\nsize}} \) and
  \(\delta_{\nsize} = 2 \alpha_{\nsize} \dimn\) imply
  \begin{EQA}[c]
    (1 - \alpha_{\nsize})^{\dimn}
    =
    \ex^{\log(1 - \alpha_{\nsize}) \dimn}
    \ge
    \ex^{-2 \alpha_{\nsize} \dimn}
    \ge
    1 - 2 \alpha_{\nsize} \dimn
  \end{EQA}
  and \eqref{mleConcentrQuant} follows.
  \end{proof}

  In the special case \( \us_{1} = \ldots = \us_{\dimn} = \us \), the set \( \Upsilons(\sqrt{\zz}) \) is a cube which can be also viewed as a ball in the sup-norm.
  Moreover, if \( \zz_{\nsize}/(\Mn \ex^{\us}) \le 1/2 \), this cube is contained in the cube
  \( \bigl\{ \uv: \, \| \uv - \uvs \| \le \sqrt{\zz_{\nsize}/(\Mn \ex^{\us})} \bigr\} \)
  in view of
  \( \ex^{x} - 1 - x \le a^{2} \le 1/2 \) for \( |x| \le a \le 1 \).
  The concentration bound \eqref{mleConcentrDim} enables us to check the local conditions only on
  the cube \( \Upsilons(\sqrt{\zz_{\nsize}}) \).
  Especially the condition \( (\CS \DF_{1}) \) is trivially fulfilled because \( \zetav(\uv) = \LL(\uv) - \E \LL(\uv) \) is linear in \( \uv \) and
  \( \theta \) is a linear functional of \( \uv \).
  Condition \( (\LL_{0}) \) can be checked on \( \Upsilons(\sqrt{\zz_{\nsize}}) \) with
  \( \rddelta(\zz_{\nsize}) = \sqrt{\zz_{\nsize}/(\Mn \ex^{\us})} \).

  It remains to compute the value \( \DPrc^{2} \). Define \( \betan = \dimn / \Mn^{1/2} = \dimn^{3/2} / \nsize^{1/2} \). If \( \nsize = \dimn^{3} \), then \( \betan = 1 \).
  \begin{lemma}
  \label{lemma: effFisherCritExamplePois}
    Let \( \vs = 1/\dimn \).
    Then it holds
    \begin{EQA}[c]
      \DPrc^{2} = \dimn^{2} \betan^{-2}.
    \end{EQA}
  \end{lemma}

  Now we are ready to finalize the proof Theorem~\ref{theorem: bvmPois}.
  \begin{proof}
    Let \( \betan \) be bounded.
    The definition implies
    \begin{EQA}[c]
      \dimn \bigl(\theta - \tilde{\theta}_{\nsize} \bigr)
      =
      \sum_{j=1}^{\dimn} \log \biggl(\frac{\upsilon_{j}}{Z_{j}  / \Mn}\biggr)
    \end{EQA}
    The posterior distribution  \(\upsilon_{j} \cond \Yv\) is \(\operatorname{Gamma}(\alpha_{j}, \mu_{j})\) with \(\alpha_{j} = 1 + Z_{j}\) and \(\mu_{j} = \frac{\mu}{\Mn \mu + 1}\).
    We use following decomposition
    \begin{EQA}[c]
      \frac{\upsilon_{j}}{Z_{j} / \Mn}
      =
      \frac{\Mn \mu_{j} \alpha_{j}}{\alpha_{j} - 1}\bigl(1 + \alpha_{j}^{-1/2} \gamma_{j} \bigr)\, ,
    \end{EQA}
    where
    \begin{EQA}[c]
      \gamma_{j}
      \eqdef
      (\alpha_{j} \mu_{j}^{2})^{-1/2} \bigl(\upsilon_{j} - \alpha_{j} \mu_{j} \bigr)
    \end{EQA}
    has zero mean and unit variance.
    We can use the Taylor expansion
    \begin{EQA}
      \dimn \bigl(\theta - \tilde{\theta}_{\nsize} \bigr)
      &=&
      \sum_{j=1}^{\dimn} \log\biggl(1 - \frac{1}{\Mn \mu + 1}\biggr) + \sum_{j=1}^{\dimn} \log\biggl(1 + \frac{1}{\alpha_{j} - 1}\biggr) + \sum_{j=1}^{\dimn} \log\biggl(1 + \alpha_{j}^{-1/2} \gamma_{j} \biggr).
    \end{EQA}
    Now take into account properties of real data distribution.
    \begin{EQA}[c]
      \alpha_{j} = \frac{\Mn}{\dimn} \biggl(1 + \sqrt{\frac{\dimn}{\Mn}}\delta_{j}\biggr),
    \end{EQA}
    where \( \delta_{j} \) is asymptotically standard normal.

    Suppose now that \( {\betan^{3}}/{\sqrt{\dimn}} \to 0\) as \(\dimn \to \infty\).
    Then \({\Mn}/{\dimn} = \bigl(\sqrt{\dimn}/\betan^{3}\bigr)^{2 / 3} \dimn^{2/3} \to \infty\) as \(\dimn \to \infty\).
    Thus for \(\dimn\) sufficient large, \(\alpha_{j} \approx \Mn / \dimn \).
    Moreover, it holds for \( \dimn \) sufficiently large that
    \( \max_{j=1,\ldots,\dimn} \alpha_{j}^{-1/2}  |\gamma_{j}| \le 1/2 \) with a high probability.
    Below we can restrict ourselves to the case when
    \( \alpha_{j}^{-1/2}  |\gamma_{j}| \le 1/2 \).
    This allows to use the Taylor expansion
    \begin{EQA}
      \dimn \bigl(\theta - \tilde{\theta}_{\nsize} \bigr)
      &=&
      \sum_{j=1}^{\dimn} \log\biggl(1 - \frac{1}{\Mn \mu + 1}\biggr)
      + \sum_{j = 1}^{\dimn} \log \biggl(1 + \frac{1}{\alpha_{j} - 1}\biggr)
      + \sum_{j=1}^{\dimn} \log\bigl( 1 + \frac{\gamma_{j}}{\sqrt{\alpha_{j}}} \bigr)
      \\
      &=&
      \sum_{j=1}^{\dimn} \frac{1}{\alpha_{j} - 1} + \sum_{j=1}^{\dimn} \frac{1}{\sqrt{\alpha_{j}}}  \gamma_{j}
      -  \sum_{j=1}^{\dimn} \frac{1}{2 \alpha_{j}} \gamma_{j}^{2} + R.
    \end{EQA}
    One can easily check that the remainder \( R \) is of order
    \( \betan^{3} / \sqrt{\dimn} \to 0 \).
    Moreover, \( \dimn^{-1/2} \sum_{j=1}^{\dimn} \gamma_{j} \) is asymptotically standard
    normal, while \( \dimn^{-1} \sum_{j=1}^{\dimn} \gamma_{j}^{2} \toP 1 \).
    CLT here can be easily checked because of the Lyapunov condition being valid.
    Also \(\sum_{j=1}^{\dimn} (\alpha_{j} - 1)^{-1} = \frac{\dimn^{2}}{\Mn} + o_{\nsize}(\betan^{2})\).
    Now check what happens if \( \betan \to 0\)
    \begin{EQA}
      \betan^{-1} \dimn \bigl( \theta - \tilde{\theta}_{\nsize} \bigr)
      &=&
      \betan + \frac{1}{\sqrt{\dimn}} \sum_{j=1}^{\dimn} \gamma_{j}
      -
      \frac{\betan}{2 \dimn} \sum_{j=1}^{\dimn} \gamma_{j}^{2} + o_{\nsize}(1)
      \tow
      \ND(0, 1).
    \end{EQA}
    Similarly, with \( \betan \equiv \beta \),
    \begin{EQA}
      \beta^{-1} \dimn \bigl( \theta - \tilde{\theta}_{\nsize} \bigr)
      &=&
      \beta
      +
      \frac{1}{\sqrt{\dimn}} \sum_{j=1}^{\dimn} \gamma_{j}
      -
      \frac{\betan}{2 \dimn} \sum_{j=1}^{\dimn} \gamma_{j}^{2} + o_{\nsize}(1)
      \tow
      \ND(\beta/2, 1).
    \end{EQA}
    This proves the result for \( \betan \equiv \beta \).
    Finally in the case when \( \betan \) grows to infinity, but
    \( \betan^{3} / \sqrt{\dimn} \to 0 \), then
    \( \betan^{-1} (\theta - \tilde{\theta}_{\nsize}) \toP \infty \).
  \end{proof}

\subsection{Proof of Lemma~\ref{lemma: effFisherCritExamplePois}}
\label{sec: proofEffFisherCritExamplePois}
  \begin{proof}
    Let \(\bar{u}_{j} = u_{j} - u_{j}^{*}\). Then
    \begin{EQA}[c]
      L(\uv, \uvs)
      =
      \LL(\uv) - \LL(\uvs)
      =
      \sum_{j=1}^{\dimn} \bigl\{ Z_{j} \bar{u}_{j} - \Mn \dimn^{-1} (\ex^{\bar{u}_{j}} - 1)\bigr\}.
    \end{EQA}
    The expected value of \(Z_{j}\) is \({\Mn}/{\dimn}\) which leads to following expectation of likelihood:
    \begin{EQA}[c]
      \E L(\uv, \uvs) = \frac{\Mn}{\dimn} \sum_{j=1}^{\dimn} \bigl(\bar{u}_{j} - (e^{\bar{u}_{j}} - 1)\bigr)
      =
      - \frac{\Mn}{\dimn} \sum_{j=1}^{\dimn} \frac{\bar{u}_{j}^{2}}{2} + O\bigl(\|\bar{\uv}\|^{3}\bigr).
    \end{EQA}
    Then we substitute \(\bar{u}_{1} = \dimn \bar{\theta} - \sum_{j=2}^{\dimn} \bar{u}_{j}\), where \(\bar{\theta} = \theta - \thetas\). Thus we get
    \begin{EQA}[c]
      \E L(\uv, \uvs) = - \frac{\Mn}{\dimn} \frac{1}{2} \bigl(\dimn \bar{\theta} - \sum_{j=2}^{\dimn} \bar{u}_{j} \bigr)^{2} - \frac{\Mn}{\dimn} \sum_{j=2}^{\dimn} \frac{\bar{u}_{j}^{2}}{2} + O\bigl(\|\bar{\uv}\|^{3}\bigr).
    \end{EQA}
    This Taylor expansion allows us compute components of Fisher information matrix:
    \begin{EQA}[c]
      \DFc^{2} = -\nabla^{2} \E \LL(\uvs) =
      \frac{\Mn}{\dimn}
      \begin{pmatrix}
        \dimn^{2} & -\dimn & \dots & \dots & -\dimn \\
        -\dimn & 2 & 1 & \dots & 1 \\
        \vdots & 1 & \ddots & \ddots & \vdots \\
        \vdots & \vdots & \ddots & \ddots & 1 \\
        -\dimn & 1 & \dots & 1 & 2
      \end{pmatrix}
    \end{EQA}
    The Fisher information for the target parameter \( \theta \) can be computed as follows:
    \begin{EQA}[c]
      \DPrc^{2} = {\Mn}{\dimn} \bigl(1 - \epv^{\T} \QQ^{-1} \epv\bigr),
    \end{EQA}
    where \( \epv = (1,\ldots,1)^{\T} \) and \( \QQ = \Id + \Edn \) with
    \( \Edn = \epv \epv^{\T} \) being the matrix of ones of size
    \((\dimn - 1) \times (\dimn - 1)\).

    It follows
    \begin{EQA}[c]
      \epv^{\T} \QQ^{-1} \epv
      =
      \tr \bigl(\epv^{\T} \QQ^{-1}  \epv \bigr)
      =
      \tr \bigl( \QQ^{-1} \epv \epv^{\T} \bigr)
      =
      \tr \bigl((\Edn + \Id)^{-1} \Edn\bigr) .
    \end{EQA}
    Further, \((\Edn + \Id)^{-1} \Edn = \Id - (\Edn + \Id)^{-1}\) yielding
    \begin{EQA}[c]
      \epv^{\T} \QQ^{-1} \epv
      =
      \tr \bigl\{ \Id - (\Edn + \Id)^{-1} \bigr\}
      =
      (\dimn - 1) - \tr\bigl\{ (\Edn + \Id)^{-1}\bigr\}
      =
      (\dimn - 1) - \sum_{j = 1}^{\dimn} \lambda_{j},
    \end{EQA}
    where \(\lambda_{j}\) are eigenvalues of matrix \( (\Edn + \Id)^{-1} \).
    It is easy to see that
    \(\lambda_{1} = \dimn^{-1} \) while \( \lambda_{2} = \dots = \lambda_{\dimn - 1} = 1\).
    Thus
    \begin{EQA}
    	\epv^{\T} \QQ^{-1} \epv
    	&=&
    	(\dimn - 1) - \bigl\{ \dimn^{-1} + (\dimn - 2) \bigr\}
    	=
    	1 - \dimn^{-1},
    	\\
    	\DPrc^{2}
    	&=&
    	{\Mn}{\dimn} \bigl(1 - \epv^{\T} \QQ^{-1} \epv\bigr)
    	=
    	{\Mn}{\dimn} \bigl\{ 1 - (1 - \dimn^{-1})\bigr\}
    	=
    	\Mn
    	=
    	\dimn^{2} \betan^{-2},
    \end{EQA}
    which completes the proof.
  \end{proof}

\subsection{Proof of Lemma~\ref{lemma: expectApproxGlm}}
\label{sec: proofExpectApproxGlm}
It holds
  \begin{EQA}
    && \nquad
    \|\DFc^{-1} (\DF^{2}(\upsilonv) - \DFc^{2}) \DFc^{-1}\| 
    =
    \sup_{\gammav \in \RR^{\dimtotal}\colon \|\gammav\| = 1} |\gammav^{\T} \DFc^{-1} (\DF^{2}(\upsilonv) - \DFc^{2}) \DFc^{-1} \gammav|
    \\
    & \le &
    \sup_{\gammav \in \RR^{\dimtotal}\colon \|\gammav\| = 1} \Bigl|\sum_{i = 1}^{\nsize}  (d''(\Psi_{i}^{\T} \upsilonv) - d''(\Psi_{i}^{\T} \upsilonvs)) \gammav^{\T} \DFc^{-1} \Psi_{i} \Psi_{i}^{\T} \DFc^{-1} \gammav \Bigr|
    \\
    & \le &
    \sup_{\gammav \in \RR^{\dimtotal}\colon \|\gammav\| = 1} \sum_{i = 1}^{\nsize}  |d''(\Psi_{i}^{\T} \upsilonv) - d''(\Psi_{i}^{\T} \upsilonvs)| \gammav^{\T} \DFc^{-1} \Psi_{i} \Psi_{i}^{\T} \DFc^{-1} \gammav
    \\
    & \le &
    \sup_{\gammav \in \RR^{\dimtotal}\colon \|\gammav\| = 1} \sum_{i = 1}^{\nsize}  L |\Psi_{i}^{\T} (\upsilonv - \upsilonvs)| \gammav^{\T} \DFc^{-1} \Psi_{i} \Psi_{i}^{\T} \DFc^{-1} \gammav
    \\
    & \le &
    L N^{-1/2} \|\DFc (\upsilonv - \upsilonvs)\| \sup_{\gammav \in \RR^{\dimtotal}\colon \|\gammav\| = 1} \gammav^{\T} \DFc^{-1} \bigl(\sum_{i = 1}^{\nsize} d''(\Psi_{i}^{\T} \upsilonvs) \Psi_{i} \Psi_{i}^{\T}\Bigr) \DFc^{-1} \gammav
    \\
    & \le &
    L \frac{\rr}{N_{2}^{1/2}}.
  \end{EQA}

\bibliography{semi_bvm,exp_ts,listpubm-with-url}
\end{document}